\documentclass[11pt]{amsart}

\usepackage{geometry,graphicx,amssymb,amsmath,amsbsy,eucal,amsfonts,mathrsfs,amscd,bm,tcolorbox,mathtools}

\usepackage[all]{xy}
\usepackage{tikz}
\usepackage{tikz-3dplot}
\usepackage{tkz-euclide}
\usetikzlibrary{cd}
\usepackage{pgfplots}
\usepackage{mathtools}

\numberwithin{equation}{section}

\allowdisplaybreaks[3]

\newcommand{\Zc}{\mathcal{Z}}
\newcommand{\Vc}{\mathcal{V}}

\newtheorem{theorem}{Theorem}[section]
\newtheorem{lemma}[theorem]{Lemma}

\newtheorem{corollary}[theorem]{Corollary}
\newtheorem{proposition}[theorem]{Proposition}
\theoremstyle{definition}

\theoremstyle{remark}
\newtheorem{remark}[theorem]{Remark}

\tikzset{commutative diagrams/.cd,
mysymbol/.style={start anchor=center,end anchor=center,draw=none}
}

\newcommand{\dofcnt}[1]{{\tiny\text{{#1}}}}
\newcommand{\R}{\mathbb{R}}
\newcommand{\V}{\mathbb{V}}
\newcommand{\M}{\mathbb{M}}
\newcommand{\SSS}{\mathbb{S}}
\newcommand{\K}{\mathbb{K}}

\newcommand{\veps}{\varepsilon}

\newcommand{\om}{{\Omega}}

\newcommand{\p}{{\partial}}

\newcommand{\mct}{\mathcal{T}_h}
\newcommand{\mcta}{\mct^\text{A}}
\newcommand{\dA}{} 
\newcommand{\ds}{} 
\newcommand{\dx}{} 
\renewcommand\div{\operatorname{div}}
\newcommand\tr{\operatorname{tr}}
\newcommand\sym{\operatorname{sym}}
\newcommand\ran{\operatorname{range}}

\newcommand{\trF}{{\ensuremath\mathop{\mathrm{tr}_{\scriptscriptstyle{F\,}}}}}
\newcommand{\devF}{{\ensuremath\mathop{\mathrm{dev}_{\scriptscriptstyle{F\,}}}}}
\newcommand{\curlF}{{\ensuremath\mathop{\mathrm{curl}_{\scriptscriptstyle{F\,}}}}}
\newcommand{\rotF}{{\ensuremath\mathop{\mathrm{rot}_{\scriptscriptstyle{F\,}}}}}
\newcommand{\gradF}{{\ensuremath\mathop{\mathrm{grad}_{\scriptscriptstyle{F\,}}}}}
\newcommand{\divF}{{\ensuremath\mathop{\mathrm{div}_{\scriptscriptstyle{\!F\,}}}}}

\newcommand{\curl}{{\ensuremath\mathop{\mathrm{curl}\,}}}
\newcommand{\grad}{{\ensuremath\mathop{\mathrm{ grad}\,}}}
\newcommand{\dive}{{\ensuremath\mathop{\mathrm{div}\,}}} 
\newcommand{\Div}{{\rm div}\,}

\newcommand{\pol}{\EuScript{P}}
\newcommand{\mpol}{\EuScript{M}}

\newcommand{\Pih}[2]{{\Pi_{{#1}, h}^{#2}}}
\newcommand{\Pith}[2]{{\tilde{\Pi}_{{#1}, h}^{#2}}}
\newcommand{\Ta}{T^{\text{A}}}
\newcommand{\Tha}{\mathcal{T}_h^{\text{A}}}
\newcommand{\Th}{\mathcal{T}_h}
\newcommand{\Fz}{F^{z}}
\newcommand{\mV}{\mathring{V}}
\newcommand{\mS}{\mathring{S}}
\newcommand{\mL}{\mathring{L}}
\newcommand{\mW}{\mathring{W}}
\newcommand{\mZ}{\mathring{Z}}
\newcommand{\tZ}{\tilde{Z}}
\newcommand{\mU}{\mathring{U}}
\newcommand{\Wo}{\mathring{W}}
\newcommand{\Vo}{\mathring{V}}
\newcommand{\Lo}{\mathring{L}}
\newcommand{\Wdr}[1]{W^{\partial, {#1}}_r}
\newcommand{\Ldr}[1]{L^{\partial, {#1}}_r}

\newcommand\vskw{\operatorname{vskw}}
\newcommand\mskw{\operatorname{mskw}}
\newcommand\skw{\operatorname{skw}}
\newcommand\deff{\operatorname{def}}
\newcommand\inc{\operatorname{inc}}
\newcommand{\revj}[1]{{\color{black}#1}}
\newcommand{\revjj}[1]{{\color{black}#1}}

\newcommand{\RM}{\mathcal{R}}
\newcommand{\FF}{{\scriptscriptstyle{F\!F}}}
\newcommand{\F}{{\scriptscriptstyle{F}}}
\newcommand{\nF}{{n\F}}
\newcommand{\Fn}{{\F n}}

\title[]{A discrete elasticity complex \\
  on three-dimensional Alfeld splits}

\author[]{Snorre H. Christiansen}
\address{Department of Mathematics,
University of Oslo,
P.O. Box 1053 Blindern,
NO-0316 Oslo, Norway}
\email{snorrec@math.uio.no}

\author[]{Jay Gopalakrishnan}
\address{Portland State University (MTH),
  PO Box 751, Portland, OR 97207, USA}
\email{gjay@pdx.edu}

\author[]{Johnny Guzm\'an}
\address{Division of Applied Mathematics,
Brown University,
Box F,
182 George Street,
Providence, RI 02912, USA}
\email{johnny\_guzman@brown.edu}

\author[]{Kaibo Hu}
\address{School of Mathematics, University of Minnesota,
 206 Church St. SE,
Minneapolis, MN 55455-0488,
USA}
\email{khu@umn.edu}

\thanks{\revjj{This work was supported in part by the National Science
    Foundation (USA) under grants DMS~1913083 and DMS~1912779.}
    The first and last two authors would like to thank the Isaac Newton Institute for Mathematical Sciences for support and hospitality during the programme Geometry, compatibility and structure preservation in computational differential equations (EPSRC grant number EP/R014604/1).}


\begin{document}

\maketitle

\begin{abstract}

  We construct conforming finite element elasticity complexes on the
  Alfeld splits of tetrahedra. The complex consists of vector fields
  and symmetric tensor \revjj{fields, interlinked via the linearized
    deformation operator}, the linearized curvature operator, and the
  divergence operator, respectively. The construction is based on an
  algebraic machinery that derives the elasticity complex from de~Rham
  complexes, and smoother finite element differential forms.
  
\end{abstract}

\section{Introduction}

Differential complexes have become a powerful tool in the construction
and analysis of numerical methods in the framework of finite element
exterior calculus~\cite{arnold2018finite,arnold2006finite}. The
example of the de Rham complex together with its various finite
element applications, especially in computational electromagnetism, is
now very well known. The elasticity complex is another example with
important applications in continuum mechanics and geometry.

In three space dimensions (3D), the elasticity complex reads as follows.
\begin{equation}\label{continuous-elasticity}
\begin{tikzcd}
0 \arrow{r} &\RM \arrow{r}{\subset} &C^{\infty}\otimes \V \arrow{r}{\deff} & C^{\infty}\otimes \mathbb{S}  \arrow{r}{\inc} & C^{\infty}\otimes \mathbb{S}\arrow{r}{\div} &C^{\infty}\otimes {\V}  \arrow{r}&0,
 \end{tikzcd}
\end{equation}
where $\V = \R^3$, $\RM=\{ a+b \times x: a, b \in \mathbb{R}^3\}$,
$\mathbb{S}$ denotes the set of symmetric $3 \times 3$ matrices,
$\deff$ denotes 
the deformation operator equaling $\sym\grad$
which we shall also often write simply
as $\veps$ (also known as the linearized strain), 
$\inc =\curl \circ \mathrm{T}\circ \curl$ gives the incompatibility
operator, which in 3D is equivalent to the linearized Einstein tensor
or the linearized Riemannian curvature. Here $\curl$ and $\div$ denote the curl and divergence operators
 applied row by row on a matrix field, respectively.  The notation
$\mathrm{T}$ in the definition of $\inc$ denotes the operation that
maps a matrix to its transpose, which we also often denote simply by
$'$ (prime).  In mechanics, \eqref{continuous-elasticity} bears the
name of the {\it Kr\"oner complex} \cite{vangoethem2010kroener,
  Krone60}, due to Kr\"oner's pioneering work on modeling defects of
the continuum by the violation of Saint-Venant's compatibility
condition, $\inc\circ\deff=0$. In the context of elasticity, the
spaces after $\RM$ in \eqref{continuous-elasticity} correspond to the
displacement, strain, stress (incompatibility), and the load,
respectively.  In geometry, the sequence \eqref{continuous-elasticity}
is referred to as the linearized {\it Calabi complex}
\cite{angoshtari2015differential,calabi1961compact} and the spaces
correspond to the embedding, the metric, and the curvature,
respectively.

The complex~\eqref{continuous-elasticity}, and its Sobolev space
version (see~\cite{arnold2020complexes}),
\begin{equation}\label{sobolev-elasticity}
\begin{tikzcd}[column sep=0.85cm]
0 \arrow{r} &\RM \arrow{r}{\subset} &\revjj{H^{2}}\otimes {\V} \arrow{r}{\veps} & H(\inc, \mathbb{S})  \arrow{r}{\inc} & H(\div, \mathbb{S})\arrow{r}{\div} &L^{2}\otimes {\V}  \arrow{r}&0,
 \end{tikzcd}
\end{equation}
where
$H(\inc, \mathbb{S}):=\{u\in H^{1}\otimes \mathbb{S}: \; \inc u \in
L^{2}\otimes \mathbb{S} \}$ and
$H(\div, \mathbb{S}):=\{u\in L^{2}\otimes \mathbb{S}: \; \div u \in
L^{2}\otimes {\V} \}$, are also relevant to variational principles in
elasticity such as the Hellinger-Reissner principle. Therefore a
discrete version of \eqref{continuous-elasticity} should be useful to
understand the behavior of structure-preserving numerical methods.  In
this paper, we shall construct a discrete finite element subcomplex
of~\eqref{sobolev-elasticity}. To the best of our knowledge, this is
the first known finite element subcomplex of the elasticity complex, 
complete with conforming subspaces of all the spaces in the sequence
and accompanying cochain projectors.

The Hellinger-Reissner principle involves the last two spaces, i.e.,
$H(\div, \mathbb{S})$ and $L^{2}\otimes {\V}$, in
\eqref{sobolev-elasticity}.  The symmetry of the tensors makes it a
challenging problem to construct conforming finite element
discretization for these spaces. In two space dimensions (2D), Johnson and
Mercier \cite{johnson1978some} constructed a stable finite element
elasticity pair on the Clough-Tocher
split. 
Later, Arnold and Winther \cite{arnold2002mixed} constructed the first
finite element elasticity pair on triangular meshes with polynomial
shape functions. This work was extended to 3D
in~\cite{arnold2008finite} and further refined to
reduce the number of degrees of freedom (dofs)
in~\cite{hu2015family}.

Despite the above-mentioned significant progress in the construction
of finite elements for the last (stress-displacement) part of the
elasticity complex \eqref{sobolev-elasticity}, the question of how to
construct an entire finite element subcomplex
of~\eqref{sobolev-elasticity} seems to have been left largely
unanswered yet.  The question is entirely natural from the viewpoint
of completing the mathematical structure.  Besides satisfying a
mathematical curiosity, there are many other utilitarian and numerical
reasons to tackle the question of constructing a discrete subcomplex.
For example, a discrete complex contains an explicit characterization
of the kernel of differential operators. This is crucial in designing
robust solvers and
preconditioners~\cite{farrell2020reynolds} in the framework of
kernel-capturing subspace correction methods
\cite{lee2007robust,schoeberlthesis}. 
Another reason is that the elasticity complex
\eqref{continuous-elasticity} is not only important for elasticity,
but also for various applications where other parts of the complex are
involved, e.g., the intrinsic elasticity \cite{ciarlet2009intrinsic}
(involving compatible strain tensors), continuum modeling of defects
\cite{amstutz2019incompatibility} (involving the $\inc$ operator), and
relativity \cite{christiansen2011linearization,li2018regge} (involving
the metric and curvature). One needs no stretch of imagination to see
progress in these areas being enabled by a discrete version of
\eqref{continuous-elasticity}. Having said that, let us also note that
this paper does not give any numerical method; the paper's sole focus
is to reveal a mathematical structure analogous to
\eqref{continuous-elasticity} inherent in certain discrete spaces.

Specifically, we construct conforming finite element spaces that form
a subcomplex of \eqref{sobolev-elasticity}, with accompanying
cochain projectors (defined on smoother subspaces), also referred to
as ``commuting projections.''  We use the Bernstein-Gelfand-Gelfand
(BGG) construction \cite{arnold2006finite,arnold2020complexes} as a
tool to guide our construction. The BGG construction is an algebraic
machinery that originated in Lie theory of geometry
\cite{vcap2001bernstein}. Later, it was introduced into numerical
analysis as a way to derive differential complexes, such as the
elasticity complex, from de Rham sequences
\cite{arnold2006defferential,arnold2020complexes,eastwood2000complex}.
The idea of the BGG construction (see \eqref{eq:5} below) is to derive
the elasticity complex from two copies of vector-valued de Rham
complexes.  To match the two complexes diagonally, the spaces of the
same form degree in the two complexes should have different
regularity. This was already noted by Arnold, Falk and Winther in
their use of the BGG construction applied to the Hellinger-Reissner
principle to derive finite element methods with weakly imposed
symmetry~\cite{arnold2007mixed}.  To match the two de Rham sequences,
they chose finite element spaces that satisfy certain algebraic
conditions. When the degrees of certain spaces in the two sequences
match exactly, one can see that the scheme with weakly imposed
symmetry actually leads to strong symmetry.  This was first observed
in~\cite{gopalakrishnan2012second} where a provably stable set of
spaces for a method imposing weak symmetry was shown to yield exactly
symmetric stress approximations, by establishing connections between
Stokes and elasticity systems, which can now be understood from
the BGG viewpoint.

The BGG machinery was used to reinterpret the 2D Arnold-Winther
element in \cite{arnold2006defferential}. Another elasticity pair by
Hu and Zhang \cite{hu2014family} was also explained in this way in
\cite{christiansen2018nodal}, where the two de Rham sequences start
with the Argyris and the Hermite elements, respectively.  In 2D, there
is another elasticity strain complex connecting the displacement,
strain (metric) and incompatibility (curvature).  Using the BGG
diagrams, Christiansen and Hu \cite{christiansen2019finite}
constructed conforming discrete strain complexes with applications in
discrete geometry.
These works helped put the pieces of the puzzle  into place in
2D. Yet, several challenges remained to get to the 3D elasticity
complex, due to the complexity of the differential structures in
\eqref{sobolev-elasticity} and the difficulties in constructing smooth
3D discrete de Rham sequences.  Thanks to recent progress on smooth
finite element de Rham complexes \cite{fu2018exact,guzman2018inf}, a
way out of the impasse finally emerged, 
at least on meshes of Alfeld splits of tetrahedra~\cite{Alfel84}.
In this paper, we are thus finally able to
construct a discrete elasticity complex on meshes 
of Alfeld splits.

Some parallel tracks of investigation by other groups of authors are
related and interesting.  Approaches to a discrete elasticity complex
from a discrete geometric perspective can be found
in~\cite{christiansen2011linearization,hauret2007diamond,li2018regge}.
The Regge calculus was originally proposed by Regge
\cite{regge1961general} and has several applications in quantum and
numerical gravity. Due to its very weak continuity, establishing
convergence might need further innovations.  Christiansen
\cite{christiansen2011linearization} put the Regge calculus into a
finite element context and fitted it into a discrete elasticity
complex and Regge interpolation was used for shells recently in
\cite{NeuntSchob20}.  From this perspective, one of the results in
this paper can be seen as providing a smoother analogue of the Regge
elements, with $H(\inc)$-conformity (and additionally $C^0$
continuity).  Our smoother spaces, while mathematically pleasing, do
come at the price of increased number of dofs, so we emphasize again
that this paper's goal is not to construct competitive numerical
methods, but rather to reveal previously unknown mathematical
structures.

The rest of the paper is organized as follows. In
Section~\ref{sec:preliminary}, we quickly present the essentials for
the remainder of the paper, including results on spaces on Alfeld
splits and the BGG resolution.  In Sections~\ref{sec:first} and
\ref{sec:second}, we present the two finite element de Rham complexes
that will be used in the BGG construction. Section~\ref{sec:local},
the centerpiece of this paper, presents finite elements on Alfeld
splits for each member of the discrete elasticity complex.  Section
\ref{sec:global} remarks on how the corresponding global finite
element spaces may be constructed. A standalone appendix
(Appendix~\ref{sec:supersmoothness}) gives an elementary argument for
establishing supersmoothness results on three-dimensional Alfeld
splits.

\section{Preliminaries}\label{sec:preliminary}


To build an elasticity complex we shall employ  two de~Rham complexes of
discrete
spaces with extra smoothness (in comparison with the standard finite
element spaces). We shall
construct these spaces in the next two sections using the results
of~\cite{fu2018exact}, which we recall in this section.

We  work on Alfeld simplicial complexes and start by establishing
notation associated to an Alfeld split.
Starting with a tetrahedron $T=[x_0, \cdots, x_3]$,
let $\Ta$ be an Alfeld triangulation of $T$, i.e.,
we choose an interior point $z$ of $T$ and
we let $T_0=[z, x_1, x_2, x_3]$,  $T_1=[z, x_0, x_2, x_3],$
$T_2=[z, x_0, x_1, x_3],$
$T_3=[z, x_0, x_1, x_2]$
and set $\Ta= \{ T_0, T_1, T_2, T_3\} $. \textcolor{black}{Let $\Delta_{i}(T)$ be the set of all $i$-dimensional subsimplexes of $T$.}

The following spaces are well-known finite element spaces:
\begin{alignat*}{1}
W_{r}^0(\Ta) &=  \{\omega \in H^{1}(T): \omega |_K \in \pol_r(K)  \; \text{ for all } K \in \Ta \},  \\
W_{r}^1(\Ta) &=\{ \omega \in H(\curl\!, T): \omega |_K \in [\pol_r(K)]^3   \; \text{ for all } K \in \Ta  \}, \\
W_{r}^2(\Ta) &= \{ \omega \in H(\dive\!, T): \omega |_K \in [\pol_r(K)]^3   \; \text{ for all } K \in \Ta  \}, \\
W_r^3(\Ta) & =  \{ \omega \in L^2(T): \omega |_K \in \pol_r(K)   \; \text{ for all } K \in \Ta  \} .
\end{alignat*}
Their analogues with  boundary conditions are
\begin{alignat*}{1}
\mathring{W}_{r}^0(\Ta) &=   \{\omega \in W_{r}^0(\Ta) : \omega =0  \text{ on  }\partial T  \},\\
\mathring{W}_{r}^1(\Ta) &=\{ \omega \in W_{r}^1(\Ta): \omega \times n=0 \text{ on  }\partial T \},\\
\mathring{W}_{r}^2(\Ta) &= \{ \omega \in W_{r}^2(\Ta):   \omega \cdot n=0 \text{ on  }\partial T    \}, \\
\mathring{W}_r^3(\Ta) & =  \{ \omega  \in W_{r}^3(\Ta): \int_T \omega=0 \}. 
\end{alignat*}
We define the following Lagrange spaces 
\begin{equation*}
  L_{r}^0(\Ta) =W_r^0(\Ta),
  \; L_{r}^1(\Ta) =[W_r^0(\Ta)]^3, \;
  L_{r}^2(\Ta) =L_r^1(\Ta),
  \;
  L_r^3(\Ta) = L_{r}^0(\Ta),
\end{equation*}
and their analogues  with boundary conditions:
\begin{equation*}
  \mathring{L}_{r}^0(\Ta) = \mathring{W}_r^0(\Ta),
  \;\mathring{L}_{r}^1(\Ta) =[\mathring{W}_r^0(\Ta)]^3,
  \;
  \mathring{L}_{r}^2(\Ta) = \mathring{L}_{r}^1(\Ta),\;
  \mathring{L}_{r}^3(\Ta) = L_r^3(\Ta) \cap \mathring{W}_r^3(\Ta).
\end{equation*}

Apart from the above standard spaces, we also need
the following ``smoother'' spaces:
\begin{alignat*}{3}
S_{r}^0(\Ta) &= \{\omega \in  L_r^0(\Ta): \grad \omega \in L_{r-1}^1(\Ta)  \},  \quad && \mathring{S}_{r}^0(\Ta) && = \{\omega \in  \mathring{L}_r^0(\Ta): \grad \omega \in \mathring{L}_{r-1}^1(\Ta)  \},\\
S_{r}^1(\Ta) &=\{ \omega \in L_r^1(\Ta): \curl \omega \in L_{r-1}^2(\Ta)  \}, \quad && \mathring{S}_{r}^1(\Ta) &&=\{ \omega \in \mathring{L}_r^1(\Ta): \curl \omega \in \mathring{L}_{r-1}^2(\Ta)  \}, \\
S_{r}^2(\Ta) &= \{ \omega \in L_r^2(\Ta): \dive \omega \in L_{r-1}^3(\Ta)  \}, \quad && \mathring{S}_{r}^2(\Ta) &&= \{ \omega \in \mathring{L}_r^2(\Ta): \dive \omega \in \mathring{L}_{r-1}^3(\Ta)  \}, \\
S_r^3(\Ta) & = L_r^3(\Ta),   \quad  &&\mathring{S}_r^3(\Ta) && = \mathring{L}_r^3(\Ta).
\end{alignat*}
\revjj{When $r \le 3$, the space $S_r^0(\Ta)$ coincides with
  $\pol_r(T)$. More generally, the} $S$-spaces have ``extra''
smoothness at the vertices as given in the next proposition.

\begin{proposition} \hfill
  \label{prop:extraC}
  \begin{enumerate}
  \item \label{item:0_prop:extraC}  Every function in $S_r^0(\Ta)$ is $C^2$ at the vertices of
    $T$.
  \item \label{item:1_prop:extraC} Every function in
    $\mathring{S}_r^0(\Ta)$ has vanishing second derivatives at the
    vertices of $T$.

  \item \label{item:2_prop:extraC} Every function in $S_r^1(\Ta)$ is $C^1$ at the vertices of
    $T$.

  \item \label{item:3_prop:extraC} Every function in
    $\mathring{S}_r^1(\Ta)$ has vanishing first derivatives at the
    vertices of $T$.

  \end{enumerate}
\end{proposition}

The first two items follow from~\cite{Alfel84} and the remainder can
be proved using a dimension counting argument found
in~\cite{fu2018exact}. Nonetheless, all the statements of the
proposition follow from elementary arguments (without counting
dimensions) detailed in Appendix~\ref{sec:supersmoothness}.

Consider the following sequences:
\begin{subequations} \label{exact}
  \begin{eqnarray}
    \label{exact1}
    &&
      \begin{tikzcd}[column sep=2em]        
        \mathbb{R} \arrow{r}
        &
        W_{r}^0(\Ta)
        \arrow{r}{\grad}
        & 
        W_{r-1}^1(\Ta)
        \arrow{r}{\curl}
        &
        W_{r-2}^2(\Ta)
        \arrow{r}{\dive}
        &
        W_{r-3}^3(\Ta)
        \arrow{r}
        & 0        
      \end{tikzcd}
    \\
    \label{exact2} 
    &&
      \begin{tikzcd}[column sep=2.15em]        
        \mathbb{R} \arrow{r}
        &
        S_r^0(\Ta)
        \arrow{r}{\grad}
        & 
        L_{r-1}^1(\Ta)
        \arrow{r}{\curl}
        &
        W_{r-2}^2(\Ta)
        \arrow{r}{\dive}
        &
        W_{r-3}^3(\Ta)
        \arrow{r}
        & 0        
      \end{tikzcd}
    \\
    \label{exact3} 
    &&
       \begin{tikzcd}[column sep=2.2em]                
         \mathbb{R} \arrow{r}
         &
         S_r^0(\Ta)
         \arrow{r}{\grad}
         & 
         S_{r-1}^1(\Ta)
         \arrow{r}{\curl}
         &
         L_{r-2}^2(\Ta)
         \arrow{r}{\dive}
         &
         W_{r-3}^3(\Ta)
         \arrow{r}
         & 0        
       \end{tikzcd}
    \\
    \label{exact4} 
    &&
       \begin{tikzcd}[column sep=2.26em]               
         \mathbb{R} \arrow{r}
         &
         S_r^0(\Ta)
         \arrow{r}{\grad}
         & 
         S_{r-1}^1(\Ta)
         \arrow{r}{\curl}
         &
         S_{r-2}^2(\Ta)
         \arrow{r}{\dive}
         &
         S_{r-3}^3(\Ta)
         \arrow{r}
         & 0.        
       \end{tikzcd}
  \end{eqnarray}
\end{subequations}
The first sequence is well known to be exact.  The last three were
shown to be exact in \cite{fu2018exact} \revjj{for $r\ge 1$} and so were
 the corresponding
 sequences with boundary conditions (see for example (4.2) in
 \cite{fu2018exact} for the case with boundary conditions).
\revjj{Throughout, when a subscript indicating the degree is negative, the
space is considered $0$-dimensional.}
 In fact, in the course of proving the exactness,
 the following representation of potentials
 was
established in \cite[proof of Theorem~3.1]{fu2018exact}.  Here and
throughout, we let $\mu \in C^{0}(T)$ denote the piecewise linear function
on $\Ta$ such that $\mu(z)=1$ and $\mu(x_i)=0$ for $0 \le i \le 3$
(i.e., $\mu$ is a bubble function on $\Ta$).

\begin{proposition}
  \label{prop:rep-div0}
  Let $r \ge 0$.
  For any $w \in \Wo_r^2(\Ta)$ with $\dive w = 0$, there exists
  $\gamma_j \in \pol_j(T)^3,~j=0, 1, \cdots, r$, such that
  \begin{equation}
    \label{eq:1}
    u = \mu \sum_{\ell =0 }^r \mu^\ell \gamma_{r- \ell}    
  \end{equation}
  satisfies $\curl u = w$.  Similarly,  any
  $w \in \Wo_r^3(\Ta)$ also has (possibly different)
  $\gamma_j \in \pol_j(T)^3$,
  which when combined to make 
  the function $u$ as in~\eqref{eq:1}, 
  satisfies $\dive u = w$.
\end{proposition}

We now collect the dimensions of the above-introduced spaces for any
degree $r
\ge 1$.
A detailed discussion of 
first two counts below can be found, e.g.,
in~\cite{fu2018exact}. The others are standard.
\begin{subequations}
  \label{eq:dimstd}
  \begin{alignat}{1}
    \label{eq:2}
\text{dim }\mathring{S}_r^0(\Ta)& =  \max(\frac{2}{3}
{(r-4)}(r-3)(r-2), \; 0), \\
\text{dim }S_r^0(\Ta) & = {r+3 \choose 3}
+
\frac  1 2 {(r-3)}(r-2)(r-1)
\\
\text{ dim } \mathring{L}_r^0(\Ta)& =
1+4(r-1)+6\frac{{(r-2)}(r-1)}{2}+4\frac{{(r-3)}(r-2)(r-1)}{6},
\\
\text{ dim } L_r^0(\Ta)& = 5+10(r-1)+10\frac{{(r-2)}(r-1)}{2}+4\frac{{(r-3)}(r-2)(r-1)}{6}, \\
\text{ dim } \mathring{W}_r^1(\Ta)& = 4(r+1)+6(r-1)(r+1)+4\frac{{(r-2)}(r-1)(r+1)}{2}, \\
\text{ dim } \mathring{W}_r^2(\Ta)& = 6\frac{(r+1)(r+2)}{2} +  4\frac{(r-1)(r+1)(r+2)}{2}, \\
\text{ dim } \mathring{W}_r^3(\Ta)& = 4\frac{(r+1)(r+2)(r+3)}{6}-1, \\
\text{ dim } W_r^1(\Ta)& = 10(r+1)+10(r-1)(r+1)+4\frac{{(r-2)}(r-1)(r+1)}{2}, \\   
\text{ dim } W_r^2(\Ta)& = 10\frac{(r+1)(r+2)}{2}+ 4\frac{(r-1)(r+1)(r+2)}{2},\\
 \text{ dim } W_r^3(\Ta)& = 4\frac{(r+1)(r+2)(r+3)}{6}. 
\end{alignat}
\end{subequations}

We conclude this section by outlining the basic approach we shall adopt for
constructing the elasticity complex on Alfeld splits.  The approach
is the same
as what others \cite{ArnolFalkWinth07} have pursued, known under the
previously noted name,  the BGG resolution. This
theme is developed further in another recent
work~\cite{arnold2020complexes}.  For our purposes here, it is
sufficient to have the following simple result.  Suppose
$\Zc_i, \Vc_i$ are Banach spaces, $r_i : \Zc_i \to \Zc_{i+1}$,
$t_i : \Vc_i \to \Vc_{i+1}$, and $s_i: \Vc_i \to \Zc_{i+1}$ are
bounded linear operators such that the following diagram commutes:
\begin{equation}
  \label{eq:5}
\begin{tikzcd}
  \Zc_0 \arrow{r}{r_0}
  & \Zc_1 \arrow{r}{r_1}
  & \Zc_2\arrow{r}{r_2}
  & \Zc_3
 \\
 \Vc_0   \arrow[r, "t_0"] \arrow[ru, "s_0"]
 & \Vc_1 \arrow[r, "t_1"] \arrow[ru, "s_1"]
 & \Vc_2 \arrow[r, "t_2"] \arrow[ru, "s_2"]
 & \Vc_3
\end{tikzcd}
\end{equation}
i.e., $r_{i+1} s_i = s_{i+1} t_i$ for $i=0, 1$. We are interested in
the situation where the top ($\Zc$) sequence and the bottom ($\Vc$)
sequence are complexes that form exact sequences.
Then we have the following result that employs the
Cartesian products $\Zc_i \times \Vc_i$, which we write as
$[\begin{smallmatrix} \Zc_i \\ \Vc_i
\end{smallmatrix}]$ so as to use the matrix multiplication pattern 
as a mnemonic.  Specifically,
$
  [r_0 \; s_0]:
  [
  \begin{smallmatrix}
    \Zc_0 \\ \Vc_0 
  \end{smallmatrix}
  ] \to \Zc_1$
and
$ [
  \begin{smallmatrix}
    s_2 \\ t_2
  \end{smallmatrix}
  ]: \Vc_2 \to
[
  \begin{smallmatrix}
    \Zc_3 \\ \Vc_3
  \end{smallmatrix}
  ]
  $
are defined, respectively, by
\begin{align*}
  \begin{bmatrix}
    r_0 & z_0
  \end{bmatrix}
          \begin{bmatrix}
            z \\ v 
          \end{bmatrix}
  = 
  r_0 z + s_0 v,
  \qquad\quad
    \begin{bmatrix}
      s_2 \\ t_2    
    \end{bmatrix}
  v
  = 
    \begin{bmatrix}
      s_2v \\ t_2  v  
    \end{bmatrix}.
\end{align*}

\begin{proposition}
  \label{prop:bgg_basic}
  \revjj{Suppose $s_1$ is a bijection.}
  \begin{enumerate}
  \item\label{item:1:prop:bgg_basic} \revjj{If $\Zc_i$ and $\Vc_i$
      are exact sequences and the diagram~\eqref{eq:5} commutes,
      then the following is an exact sequence:}
    \[
      \begin{tikzcd}[column sep=huge]
        \begin{bmatrix}
          \Zc_0 \\ \Vc_0        
        \end{bmatrix}
        \arrow{r}
        {\left[\begin{smallmatrix}r_0 \; s_0 \end{smallmatrix}\right]}
        &
        \Zc_1
        \arrow[r, "t_1\, \circ \,s_1^{-1}\, \circ\, r_1"]
        & \Vc_2
        \arrow{r}{\left[\begin{smallmatrix}
              s_2 \\ t_2 \end{smallmatrix}\right]}
        &
        \begin{bmatrix}
          \Zc_3 \\ \Vc_3
        \end{bmatrix}
      \end{tikzcd}.
    \]
  \item\label{item:2:prop:bgg_basic} \revjj{For the surjectivity of
      the last map
      $\left[\begin{smallmatrix} s_2 \\ t_2 \end{smallmatrix}\right]$,
      it is sufficient that $r_2$ and $t_2$ are surjective,
      $t_1 \circ t_2=0$, and $s_2 t_1 = r_2 s_1$.}
  \end{enumerate}
\end{proposition}
\begin{proof}
  The range of
  $ \left[r_0 \; s_0 \right]$ is
  contained in the kernel of $t_1 \circ s_1^{-1} \circ r_1$ because
  for any $(z, v) \in \Zc_0 \times \Vc_0$
  \[
    t_1 s_1^{-1} r_1 (r_0 z + s_0 v) = t_1 s_1^{-1} r_1 s_0 v
    =  t_1 s_1^{-1} s_1 t_0 v
    =     t_1 t_0 v = 0,
  \]
  where we have used the given assumptions that the
  top and bottom sequences in~\eqref{eq:5} are complexes and
  that the commutativity property, $r_1 s_0 = s_1t_0$, holds.  For the
  reverse inclusion, if $z \in \Zc_1$ is in the kernel of
  $t_1 \circ s_1^{-1} \circ r_1$, then $s_1^{-1} r_1 z$ is in
  $\ker t_1 = \ran t_0,$ so there is a $v \in \Vc_0$ such that
  $s_1^{-1} r_1 z = t_0 v$, i.e.,
  $0 = r_1 z - s_1t_0 v = r_1 z - r_1 s_0 v $ using the commutativity
  property again. Hence,  $z - s_0 v$ is in
  $\ker r_1 = \ran r_0$, i.e.,
  $z = s_0 v + r_0 z_0$ for some $z_0 \in
  \Zc_0$, thus showing that  $z$ is in the
  range of $ \left[r_0 \; s_0 \right]$ and completing the proof of
  $\ran \left[ r_0 \; s_0 \right] =
  \ker(t_1 \circ s_1^{-1} \circ r_1).$

  We use the other commutativity property, $r_2s_1 = s_2 t_1,$ to prove
  $\ran (t_1 \circ s_1^{-1} \circ r_1) = \ker
  \left[\begin{smallmatrix} s_2 \\ t_2 \end{smallmatrix}\right]$.
  Consider a $v_2$ in the latter kernel, i.e., $s_2v_2 =0$ and
  $t_2v_2=0$. Since $\ker(t_2)= \ran(t_1)$, there is a
  $v_1 \in \Vc_1$ such that $v_2 = t_1(v_1)$, so
  $0 = s_2v_2 = s_2 t_1 v_1 = r_2 s_1 v_1$. Since
  $s_1v_1 \in \ker(r_2) = \ran(r_1)$, there is a $z_1 \in \Zc_1$ such
  that $s_1 v_1 = r_1 z_1$, i.e., $v_1 = s_1^{-1} r_1 z_1$. Thus
  $v_2 = t_1 v_1 = t_1 s_1^{-1} r_1 z_1$, i.e., 
  $\ran (t_1 \circ s_1^{-1} \circ r_1) \supseteq \ker
  \left[\begin{smallmatrix} s_2 \\ t_2 \end{smallmatrix}\right]$. The
  reverse inclusion is easy.

  To prove the final statement, consider a $z_3 \in \Zc_3$ and
  $v_3 \in \Vc_3$. If $t_2$ is surjective, then there is a
  $\tilde{v}_2 \in \Vc_2$ such that $t_2 \tilde{v}_2 = v_3$.  If $r_2$
  is also surjective, then we can find a $z_2$ and $\tilde{z}_2$ in
  $\Zc_2$ such that $r_2 z_2 = z_3$ and
  $r_2 \tilde{z}_2 = s_2 \tilde{v}_2$. It may now be easily verified
  that \revjj{$v_2 = t_1 s_1^{-1} z_2 + \tilde{v}_2 - t_1 s_1^{-1}
  \tilde{z}_2$} in $\Vc_2$ satisfies 
  $ [
  \begin{smallmatrix}
    s_2 \\ t_2
  \end{smallmatrix}
  ] v_2 = 
[
  \begin{smallmatrix}
    z_3 \\ v_3
  \end{smallmatrix}
  ].
  $
\end{proof}

In the next two sections, we shall construct specific instances
of the $\Zc$ and $\Vc$ sequences in such a way that
Proposition~\ref{prop:bgg_basic} may then be applied to produce an
elasticity complex.

\section{The first exact sequence of spaces}\label{sec:first}

In this section, we develop
one of the above-mentioned two sequences of 
spaces. This sequence is comprised of the following spaces:
\begin{alignat*}{1}
V_{r}^0(\Ta) &=S_{r}^0(\Ta), \\
V_{r}^1(\Ta) &=\{ \omega \in L_r^1(\Ta):   \text{ $\omega$ is $C^1$ at vertices of $T$}\}, \\
V_{r}^2(\Ta) &= \{ \omega \in W_r^2(\Ta):  \text{ $\omega$ is $C^0$ at vertices of $T$}\}, \\
V_{r}^3(\Ta)& =W_r^3(\Ta).
\end{alignat*}
Note that we do not impose additional continuity at the interior
vertex $z,$ making these spaces slightly different from a sequence of
similar spaces considered in \cite{fu2018exact}. The corresponding
spaces with boundary conditions are given as follows:
\begin{alignat*}{1}
\mV_{r}^0(\Ta) &= \mathring{S}_{r}^0(\Ta), \\
\mV_{r}^1(\Ta) &=\{  \omega \in \mathring{L}_r^1(\Ta):   \;\grad
\omega=0 \text{ at vertices of } T  \}, \\
\mV_{r}^2(\Ta) &= \{ \omega \in \mathring{W}_{r}^2(\Ta): \;\omega=0
\text{ at vertices of } T   \}, \\
\mV_{r}^3(\Ta)& =\mathring{W}_r^3(\Ta).
\end{alignat*}
Note an $\omega \in \mL_r^1(\Ta)$ generally has a multivalued
$\grad \omega$ at the vertices of $T$, so the statement
``$\grad \omega=0$ at vertices of $T$'' above should be understood as
follows: $\grad\omega$ exists at $x_i$ (i.e., the multiple limiting
values coincide) and equals 0. The statement ``$\omega=0$ at vertices
of $T$'' for $\omega \in \mW_r^2(\Ta)$ above, and similar such statements
later in the paper, carry the same tacit understanding.

\subsection{Characterizations and dimensions of the $V$ spaces}

We shall now provide some characterizations of the $\mV$ spaces which
makes their dimensions obvious.
Let
$\Fz$ denote the set of interior facets (2-subsimplices) of the mesh
$\Ta$. Each $f \in \Fz$ has $z$ as a vertex. The subcollection of
three facets in $\Fz$ having $x_i$ as a vertex is denoted by
$\Fz_i$. Let $w_{n, f}$ denote the normal component of $w$ on an
$f\in \Fz$, i.e., $w_{n, f} = w \cdot n|_f$ where $n$ is a unit normal
to $f$ of arbitrarily fixed orientation.

\begin{lemma} \label{lem:Vochar}
  The following equalities hold:
  \begin{align}
    \label{eq:Vo1char}
    \mV_r^1(\Ta)
    & = \{ \mu \, p: \; p \in L_{r-1}^1(\Ta)
      \text{ satisfying } p(x_i)=0, \text{ for } i=0,1,2,3\},
    \\     \label{eq:Vo2char}
    \mV_r^2(\Ta)
    & = \{ w \in \mathring W_r^2(\Ta): \; w_{n, f}(x_i) =0
  \text{ for all } f \in \Fz_i, i=0,1,2,3\}.
  \end{align}
\end{lemma}
\begin{proof}
  Let $v \in \mV_r^1(\Ta)$. On each $T_i$, since $v$ vanishes on the
  facet where $\mu=0$, we may factor it uniquely as $v = \mu p$ for some
  $p \in \pol_{r-1}(T_i)^3$. Since $v$ is continuous on $T$, we
  conclude that $p \in L_{r-1}^1(\Ta)$. Moreover, since
  $\grad v(x_i) = (\grad \mu p)(x_i)$ and $\mu(x_i)$ are zero, 
  $  p(x_i) \grad \mu(x_i) = 0$. Hence
  $p(x_i)=0$,  so $v$ is in the set on the right hand side
  of~\eqref{eq:Vo1char}.  Since the reverse inclusion
  ``$\supseteq$'' is easy to see, the set equality
  of~\eqref{eq:Vo1char} follows.

  For~\eqref{eq:Vo2char}, since the ``$\subseteq$''-part is easy, let us
  focus on proving the reverse. Let $v$ be in the set on the right
  hand side of~\eqref{eq:Vo2char}.  Consider a vertex, say $x_1$.  Three facets
  of $T_1=[z, x_1, x_2, x_3],$ namely
  $f_1 = [x_1, x_2, x_3], f_2 = [z, x_1, x_2], f_3 = [z, x_1, x_3],$
  meet at $x_1$. Letting $n_i$ denote the outward unit normal on
  $f_i$, observe that $\{n_1, n_2, n_3\}$ is a linearly independent
  set since $T_1$ has positive volume. The given conditions on $v$
  imply that $v_{n_1, f_1} \equiv 0$ and
  $v_{n_2, f_2} (x_1) = v_{n_3, f_3} (x_1) = 0$, i.e., three
  independent components of $v|_{T_1}(x_1) \in \R^3$ vanish, so
  $v|_{T_1}(x_1)=0$. Repeating this argument at other $x_i$ and $T_j$,
  we conclude that all limiting values of $v$ at every vertex $x_i$
  vanish. Therefore $v \in \mV_r^2(\Ta)$.
\end{proof}


Let $W_r^k(T) = \pol_r\Lambda^k(T)$, not to be confused with
$W_r^k(\Ta)$. The well-known canonical degrees of freedom of this
space provides the direct decomposition \cite{arnold2006finite}
\[
  W_r^k(T) = \Wo_r^k(T) \oplus \Wdr k(T),
\]
where $ \Wo_r^k(T)$ is the span of all interior shape functions of
$ \pol_r\Lambda^k(T)$ and $\Wdr k (T)$ is the span of the remaining shape
functions. Let $\Ldr 1 (T) = [\Wdr 0 (T)]^3$.

\begin{lemma}
  \label{lem:V_Vo}
  The following equalities hold:
  \begin{align}
    \label{eq:Vo1Vo2cap}
    \Vo^1_r(\Ta)  & = V_r^1(\Ta) \cap \Lo_r^1 (\Ta),
    & \Vo^2_r(\Ta) & = V_r^2(\Ta) \cap \Wo_r^2(\Ta),
    \\     \label{eq:V1V2decomp}
    V_r^1(\Ta) & = \Vo_r^1(\Ta) \oplus \Ldr 1 (T),
    & V_r^2(\Ta) & = \Vo_r^2(\Ta) \oplus \Wdr 2 (T).
  \end{align}
\end{lemma}
\begin{proof}
  Let $v \in V_r^1(\Ta) \cap \Lo_r^1 (\Ta).$ Then $v$ is $C^1$ at
  $x_i$, so $\grad v (x_i)$ is well-defined. Since $v$ is zero along
  the three edges of $T$ connected to $x_i$, three linearly
  independent components of the vector $\grad v_j (x_i)$ are zero for each $1 \le j \le 3$, so
  $\grad v (x_i)=0$. Hence
  $\Vo^1_r(\Ta) \supseteq V_r^1(\Ta) \cap \Lo_r^1 (\Ta)$.  Together
  with the obvious reverse inclusion, the first equality
  of~\eqref{eq:Vo1Vo2cap} follows. The proof of the second is similar:
  indeed, if $w \in V_r^2(\Ta) \cap \Wo_r^2(\Ta)$, then $w$ is $C^0$
  at $x_i$, so $w(x_i)$ is a single-valued vector whose three
  independent components $w(x_i) \cdot n_j$ (for $j\ne i$) vanish,
  where $n_j$ is the unit normal to the facet opposite to $x_j$.
  Hence 
  $w \in \Vo^2_r(\Ta)$.

  To prove the first decomposition of~\eqref{eq:V1V2decomp}, first
  observe that the sets $\{ u|_{\partial T} : u \in L_r^0(\Ta)\}$ and
  $\{ u|_{\partial T} : u \in \pol_r(T)\}$ coincide.  Consequently,
  the trace of any $v \in V_r^1(\Ta) \subseteq L_r^1(\Ta)$ has a
  unique extension in $\Ldr 1(T)$, which we shall call $v_L$. Put
  $v_0 = v - v_L$. We claim that
  \[
    v = v_0 + v_L
  \]
  is the required decomposition. Indeed, since $v_L$ is a polynomial
  on $T$ (and hence smooth), the function $v_0 = v - v_L$ is in
  $V_r^1(\Ta)$. Moreover, since the trace of $v_0$ is zero,
  $v_0 \in V_r^1(\Ta) \cap \Lo_r^1 (\Ta)$, so by~\eqref{eq:Vo1Vo2cap},
  $v_0 \in \Vo^1_r(\Ta)$. (The directness of the decomposition follows
  easily by examining the boundary values of the component spaces.)
  
  The second decomposition in~\eqref{eq:V1V2decomp} is proved
  similarly.
  \end{proof}

\begin{lemma}
  \label{lem:dimVoV}
  The dimensions of the $\Vo$ and $V$ spaces for any $r \ge 1$
  are as follows.
  \begin{align*}
    \dim \Vo^0_r(\Ta)
    & = \dim \mathring{S}^0_r(\Ta),
    & \dim V^0_r(\Ta) & =  \dim S_r^0(\Ta),
    \\
    \dim \Vo^1_r(\Ta)
    & = \max(2r^3 - 3r^2 + 7r - 15,\; 0),                     
    & \dim V^1_r(\Ta) & = 6(r^2 + 1) + \dim \Vo^1_r(\Ta),
    \\
    \dim \Vo^2_r(\Ta) & = 2 r^3 + 7r^2 + 7r -10,
    & \dim V^2_r(\Ta) & = 2r^3 + 9r^2 + 13 r - 6,
    \\
    \dim \Vo^3_r(\Ta) & = \frac 2 3 (r+1)(r+2)(r+3)  -1,
    & \dim V^3_r(\Ta) & = \frac 2 3 (r+1)(r+2)(r+3).                       
  \end{align*}
\end{lemma}
\begin{proof}
  The counts in the first and last rows are obvious from the
  definition and~\eqref{eq:dimstd}. For the remainder, we first claim that
  \begin{align}
    \label{eq:Vodim-12}
    \dim \mV_r^1(\Ta) & = \dim L_{r-1}^1(\Ta) - 12,
    &
     \quad \dim \mV_r^2(\Ta) & = \dim \mathring{W}_r^2(\Ta) - 12.
  \end{align}
  Indeed, by virtue of~\eqref{eq:Vo1char} of Lemma~\ref{lem:Vochar},
  the dimension of $\mV_r^1(\Ta)$ must equal that of $L_{r-1}^1(\Ta)$
  minus the number of independent constraints imposed by the condition
  ``$p(x_i)=0$'' there, which amounts to three linearly independent
  constraints (one for each component) per vertex $x_i$. This yields
  the first count in~\eqref{eq:Vodim-12}. The second
  count in~\eqref{eq:Vodim-12} follows
  from~\eqref{eq:Vo2char} of Lemma~\ref{lem:Vochar}, where at each
  vertex $x_i$, there are three independent constraints, one for each
  interior facet connected to $x_i$. The lemma's expressions of
  $ \dim \mV_r^1(\Ta) $ and $ \dim \mV_r^2(\Ta) $ are now easily
  obtained by substituting~\eqref{eq:dimstd} in~\eqref{eq:Vodim-12},
  and simplifying, noting that $\mV_r^1(\Ta)$ is trivial for $r=1$.

  It only remains to count $\dim V^1_r(\Ta)$ and $\dim
  V^2_r(\Ta)$. From~\eqref{eq:V1V2decomp} of Lemma~\ref{lem:V_Vo},
  \begin{equation}
    \label{eq:dimV12}
    \begin{aligned}    
    \dim V^1_r(\Ta) & = \dim \Vo^1_r(\Ta) + \dim \Ldr 1 (T),
    \\
    \dim V^2_r(\Ta) & = \dim \Vo^2_r(\Ta) + \dim \Wdr 2 (T).
    \end{aligned}
  \end{equation}
  It is easy to see from the canonical set of degrees of freedom of
  $\pol_r\Lambda^k(T)$ that
  \[
    \dim \Ldr 1 (T) = 3\left(
      4 + 6(r-1) + 4 \frac{(r-1)(r-2)}{2}
    \right),
    \qquad
    \dim\Wdr 2 (T)
    = 4\frac{(r+1)(r+2)}{2}.
  \]
  Using this in~\eqref{eq:dimV12} as well as previously computed
  dimensions of $\Vo$~spaces, we obtain the  stated expressions for
  $\dim V^1_r(\Ta)$ and $\dim V^2_r(\Ta)$.
\end{proof}

\subsection{Exactness}

We now proceed to show that the following local sequences
are exact:
\begin{equation}
  \label{localV}
  \begin{tikzcd}[column sep=small]
    0\arrow{r}{} &[1em]
    \mathbb{R}\arrow{r}{\subset} &[1em]
    V_{r}^0(\Ta)\arrow{r}{\grad} & [1em]
    V_{r-1}^1(\Ta)\arrow{r}{\curl} &[1em]
    V_{r-2}^2(\Ta)\arrow{r}{\div} &[1em]
    V_{r-3}^3(\Ta)\arrow{r}{} &0,
  \end{tikzcd}
\end{equation}
\begin{equation}
  \label{mlocalV}
  \begin{tikzcd}[column sep=small]
    0\arrow{r}{} &[1em]
    \mV_{r}^0(\Ta)\arrow{r}{\grad} & [1em]
    \mV_{r-1}^1(\Ta)\arrow{r}{\curl} &[1em]
    \mV_{r-2}^2(\Ta)\arrow{r}{\div} &[1em]
    \mV_{r-3}^3(\Ta)\arrow{r}{} &0.
  \end{tikzcd}
\end{equation}
\revjj{In the sequel, to indicate the null space of a differential
  operator $\mathscr{D}$  in relation to its domain $Y$, 
we use 
$ \ker(\mathscr{D}, Y):=\{w\in Y : \mathscr{D}w=0\}$.}
Note that
$\mV_r^0(\Ta) = \mS_r^0(\Ta)$ is nontrivial only for $r \ge 5$
(see~\eqref{eq:dimstd}).

\begin{lemma} \label{lem:exactVo}
  The sequence \eqref{mlocalV} is exact for any $r \ge 5$.
  \revjj{For $r=3$ and $4$, the subsequence of~\eqref{mlocalV}
    starting from $\mV_{r-1}^1$ is exact.}
\end{lemma}
\begin{proof}
  By Proposition~\ref{prop:extraC}, any $w \in \mV_{r}^0(\Ta)$ is
  $C^2$ at $x_i$, so $\grad w$ vanishes at $x_i$. Therefore,
  $\grad \mV_{r}^0(\Ta) \subseteq \mV_{r-1}^1(\Ta).$ Of course, due to
  the boundary condition,
  $\grad: \mV_{r}^0(\Ta) \to \mV_{r-1}^1(\Ta) $ is also injective.

  Proceeding to the next operator, it's easy to see that
  \revjj{$\curl \Vo_{r-1}^1(\Ta) \subseteq \ker(\div,
    \Vo^2_{r-2}(\Ta)).$} To prove the reverse inclusion, consider a
  $w \in \Vo_{r-2}^2(\Ta)$ with $\dive w =0$. Then, by
  Proposition~\ref{prop:rep-div0}, there is a $u = \mu v$ such that
  $\curl u = w$, where
  {$v = \sum_{\ell=0}^{r-2}\mu^\ell \gamma_{r-2-\ell}$} and
  $\gamma_\ell \in \pol_\ell(T)^3$. This implies that
  \[
    \grad \mu  \times v  = w - \mu \,\curl v.
  \]
  Since $w \in \Vo_{r-2}^2(\Ta)$, the right hand side above vanishes
  at all the vertices of $T$, and so does~$v$. Hence $u = \mu v$ has
  vanishing $\grad u$ at the vertices of $T$, which implies
  $u \in \Vo_{r-1}^1(\Ta)$. Thus
  \revjj{$\curl \Vo_{r-1}^1(\Ta) = \ker(\div, \Vo^2_{r-2}(\Ta))$}.

  Finally, consider the divergence operator. Since
  $\Vo_{r-2}^2(\Ta) \subseteq \Wo_{r-2}^2(\Ta)$, and since the
  standard de Rham complex---the version of \eqref{exact1} with
  boundary conditions---implies that 
  $\dive \Wo_{r-2}^2(\Ta) = \Wo_{r-3}^3(\Ta)$, we have
  $\dive \Vo_{r-2}^2(\Ta) \subseteq \Vo_{r-3}^3(\Ta).$ To improve this
  inclusion to equality, recall that by the exactness of the version
  of~\eqref{exact3} with boundary conditions proved
  in~\cite{fu2018exact}, $\Wo_{r-3}^3(\Ta) = \dive
  \Lo_{r-2}^2(\Ta)$. Since
  $\Vo_{r-2}^2(\Ta)\supseteq \Lo_{r-2}^2(\Ta)$, we conclude that
  $\dive \Vo_{r-2}^2(\Ta) = \Vo_{r-3}^3(\Ta).$
\end{proof}

\begin{lemma} \label{lem:exactV}
The sequence \eqref{localV} is exact for any $r \ge 3$.
\end{lemma}
\begin{proof}
  The exactness of
  $\begin{tikzcd}[column sep=small]
    \R \arrow{r}{} 
    & V_{r}^0(\Ta)
    \arrow{r}{\grad}
    &[0.7em] V_{r-1}^1(\Ta)
  \end{tikzcd}
  $
  and
  $\begin{tikzcd}[column sep=small]
    V_{r-2}^2(\Ta)\arrow{r}{\div}
    &[0.6em]
    V_{r-3}^3(\Ta)\arrow{r}{} &0
  \end{tikzcd}$ follow easily using standard exactness results, so we
  shall only consider the $\curl$ case.  Since it is obvious that
  \revjj{$\curl V_{r-1}^1(\Ta) \subseteq \ker (\div,
    V_{r-2}^2(\Ta))$,} let us prove the reverse inclusion. Consider a
  $\rho \in V_{r-2}^2(\Ta)$ with $\dive \rho=0$.  Let
  $\Pi \rho \in [\pol_{r-2}(T)]^3$ be the canonical interpolant of
  $\rho$ per the standard {N\'ed\'elec (second type) degrees of
    freedom~\cite{Nedel86},} defined for $r-2 \ge 1$. Let
  $\psi=\rho-\Pi \rho$.  By the well-known properties of $\Pi$,
  $\dive \psi=0$, and $\psi \cdot n=0$ on $\partial T$.  Since $\rho$
  is $C^0$ at $x_i$, $\psi$ must vanish at the vertices of $T$.  Thus
  $\psi$ is in
\revjj{$\ker(\div, \mV_{r-2}^2(\Ta))$}.
By
  Lemma~\ref{lem:exactVo}, there is an $\omega \in \mV_{r-1}^1(\Ta)$
  such that $\curl \omega_1=\psi$.  By a standard exactness result, we
  also know there is an $\omega_2 \in [\pol_{r-1}(T)]^3$ such that
  $\curl \omega_2=\Pi \rho$. Hence, $\curl \omega=\rho$ where
  $\omega=\omega_1+\omega_2 \in V_{r-1}^1(\Ta)$.
\end{proof}

\subsection{Degrees of freedom of the $V$ spaces} 

The degrees of freedom (dofs) of the $V$ spaces are almost the same as
the ones used in~\cite{fu2018exact}, the   only difference being    that
some of our bubble spaces are less smooth at the interior
point~$z$. The unisolvency proofs are substantially similar to
those in \cite{fu2018exact}, so we do not write them out here. We only
state the degrees of freedom.

Let $r\ge 5$. Then, a function $\omega \in V_r^0(\Ta)$ is uniquely
determined by the following dofs (see \cite[Lemma 4.8]{fu2018exact}) :
\begin{subequations}
\label{eqn:Sigmadofs}
\begin{alignat}{4}
\label{eqn:C1DOF1}
&D^\alpha \omega (a),\qquad &&  |\alpha|\le 2,\quad && a\in \Delta_0(T)\qquad &&\dofcnt{($40$ dofs)},\\
\label{eqn:C1DOF21}
&\int_e \omega \,\sigma {\ds},\qquad && \sigma\in \pol_{r-6}(e),\quad && e\in\Delta_1(T)\qquad &&
\dofcnt{($6(r-5)$ dofs)},\\
\label{eqn:C1DOF22}
&\int_e \frac{\p \omega}{\p n_e^\pm}\,\sigma\,{\ds},\qquad 
&& \sigma\in \pol_{r-5}(e),\quad && e\in\Delta_1(T)\qquad &&
\dofcnt{($12(r-4)$ dofs)},\\
\label{eqn:C1DOF31}
&\int_F \omega \, \sigma \,{\dA},
\qquad 
&& \sigma\in \pol_{r-6}(F),\quad && F\in\Delta_2(T)\qquad &&
\dofcnt{($4\frac{(r-5)(r-4)}{2}$ dofs)},\\
\label{eqn:C1DOF32}
&\int_F \frac{\p \omega}{\p n^\F}\, \sigma \,{\dA},
\qquad 
&& \sigma\in \pol_{r-4}(F),\quad && F\in\Delta_2(T)\qquad &&
\dofcnt{($4\frac{(r-3)(r-2)}{2}$ dofs)},\\
\label{eqn:C1DOF4}
&
\int_T {\grad \omega \cdot\grad\sigma}\,{\dx},
\qquad 
&& \sigma\in \mV_{r}^{0}(\Ta),&& &&
\dofcnt{($2\frac{(r-4)(r-3)(r-2)}{3}$ dofs)}.
\end{alignat}
\end{subequations}
Here, $\{n_{e_+}, n_{e_-}\}$ is an orthonormal set spanning the plane
orthogonal to the edge $e$, $n^\F$ denotes the outward unit normal of
face $F$, and $n^\F \cdot \grad \omega$ is abbreviated to
$\partial \omega/ \partial n^F$.  In the case $r = 5$, the sets listed
in \eqref{eqn:C1DOF21} and \eqref{eqn:C1DOF31} are omitted. It is easy
to see that the sum of dofs above equal $\dim V^0_r(\Ta)$ given in
Lemma~\ref{lem:dimVoV}.

A function $\omega\in V_{r-1}^1(\Ta)$ is uniquely determined by the
values (see \cite[Lemma 4.15]{fu2018exact})
\begin{subequations}
\label{eqn:VecHermitedofs}
\begin{alignat}{4}
\label{eqn:VecHermitedofs1}
&D^\alpha \omega(a), &&|\alpha|\le 1, \; a\in \Delta_0(T)
&&\dofcnt{($48$ dofs)}\\
\label{eqn:VecHermitedofs2}
&\int_e \omega\cdot\kappa\,{\ds}, 
&& \kappa\in [\pol_{r-5}(e)]^3, \; e\in \Delta_1(T)
&& 
\dofcnt{($18(r-4)$ dofs)}\\
\label{eqn:VecHermitedofs3}
&
\int_e (\curl \omega|_F \cdot n^\F)\kappa\, {\ds},\quad
&& \kappa\in \pol_{r-4}(e), 
\\
\nonumber
& && e\in \Delta_1(F), F\in \Delta_2(T),
&&\dofcnt{($12(r-3)$ dofs)}\\
\label{eqn:VecHermitedofs4}
&\int_F (\omega\cdot n^\F)\kappa\, {\dA}, && \kappa\in
\pol_{r-4}(F), \; F\in \Delta_2(T)
&& 
\dofcnt{($2(r-2)(r-3)$ dofs)}\\
\label{eqn:VecHermitedofs6}
&\int_F (n^\F\times (\omega \times n^\F))\cdot \kappa \, {\dA},\quad
&& \kappa\in D_{r-5}(F), \; F\in \Delta_2(T) \hspace{1cm}
&& 
\dofcnt{($4(r-3)(r-5)$ dofs)}\\
\label{eqn:VecHermitedofs7}
&\int_T \omega \cdot \kappa\, {\dx}, \quad
&& \kappa\in \grad \mV_r^0(\Ta),
&&
\dofcnt{($2\frac{(r-4)(r-3)(r-2)}{3}$ dofs)} \\
\label{eqn:VecHermitedofs8}
&\int_T \curl \omega \cdot \kappa\, {\dx}, \qquad
&& \kappa\in \curl \mV_{r-1}^1(\Ta),
&& 
\dofcnt{($\frac{4r^3 - 9r^2 + 5r-33}{3}$ dofs)}
\end{alignat}
\end{subequations}
where
\begin{equation}
  \label{eq:RT}
{  D_{r-5}(F)=[\pol_{r-6}(F)]^2+x \pol_{r-6}(F)  }
\end{equation}
is the
local Raviart--Thomas space on the face $F$. The count of dofs
in~\eqref{eqn:VecHermitedofs8} is obtained as a consequence of the
exactness established in Lemma~\ref{lem:exactVo}, i.e., 
\[
  \dim \big(\curl \Vo_{r-1}^1(\Ta)\big)
  = \dim \Vo_{r-2}^2(\Ta) - \dim \Vo_{r-3}^3(\Ta) = \frac{4r^3 - 9r^2 + 5r-33}{3},
\]
where we have also used the $\Vo$-dimensions given in
Lemma~\ref{lem:dimVoV}. The sum of the dofs in the last column above
can be easily verified to equal the expression for $\dim V^1_{r-1}(\Ta)$
given in Lemma~\ref{lem:dimVoV}. 

\begin{remark}\label{remark121}
Note that if a function $\omega \in V_{r-1}^1(\Ta)$ has  vanishing
degrees of freedom
\eqref{eqn:VecHermitedofs1}--\eqref{eqn:VecHermitedofs6} then
$\omega|_{\partial T}=0$ (see \cite [Lemma 4.15]{fu2018exact}). 
\end{remark}

A function $\omega\in V_{r-2}^2(\Ta)$ is uniquely determined by the
values (see \cite[Lemma 4.17]{fu2018exact})
\begin{subequations}
\label{eqn:StenbergHDiv}
\begin{alignat}{4}
\label{eqn:StenbergHDiv1}
&\omega(a),\quad && a\in \Delta_0(T)\quad &&\dofcnt{($12$ dofs)}\\
\label{eqn:StenbergHDiv2}
&\int_e (\omega\cdot n^\F) \kappa\, {\ds},\ && \kappa \in
\pol_{r-4}(e), \; {e\in \Delta_1(F)},\;  F\in \Delta_2(T)\quad &&
{\dofcnt{($12(r-3)$ dofs)}}\\
\label{eqn:StenbergHDiv3}
&\int_F (\omega\cdot n^\F)\kappa\, {\dA}, \quad &&{\kappa\in
\revjj{{\pol}_{r-5}(F)}}, \; F\in \Delta_2(T)\qquad &&{\dofcnt{($2(r-3)(r-4)$ dofs)}}\\
\label{eqn:StenbergHDiv4}
&\int_T \omega \cdot \kappa\, {\dx}, \quad &&  \kappa \in \curl \mV_{r-1}^1(\Ta)
 && \dofcnt{($\frac{4r^3 - 9r^2 + 5r-33}{3}$ dofs)}
\\
\label{eqn:StenbergHDiv5}
&\int_T (\Div\omega) \,(\dive \kappa) \, {\dx},\quad && \kappa \in \mV_{r-2}^2 (\Ta) 
 &&\dofcnt{($\frac 2 3 (r-2)(r-1)r - 1$ dofs)}
\end{alignat}
\end{subequations}
where we have used {the} exactness result of Lemma~\ref{lem:exactVo} again
to count the number of dofs in \eqref{eqn:StenbergHDiv5}, 
$\dim( \dive \mV_{r-2}^2 (\Ta)) = \dim \Vo_{r-3}^3(\Ta)$. The total
number of dofs above can again be easily seen to equal
$\dim V_{r-2}^2(\Ta)$ in Lemma~\ref{lem:dimVoV}.

Finally, for  $r\ge 5$, 
a function $\omega \in V_{r-3}^3 (\Ta)$ is uniquely determined by
the following degrees of freedom (see \cite[Lemma 4.18]{fu2018exact}):
\begin{subequations}
\label{Wdofs}
\begin{alignat}{4}
\label{Wdofs1}
&\int_T \omega \, {\dx}, && &&\quad\qquad \dofcnt{($1$ dof)}\\
\label{Wdofs2}
&\int_T \omega \, \kappa \, {\dx}, \quad &&\quad \kappa\in
\mV_{r-3}^3(\Ta), &&\quad\qquad \dofcnt{($\frac 2 3 r(r-1)(r-2)  -1$ dofs)}.
\end{alignat}
\end{subequations}
With these dofs, the following result can be proved along the same
lines as  \cite{fu2018exact}.

\begin{theorem} \label{thm:PiV-local}
  Let $\Pi_i^V$ denote the canonical interpolation into
  $V_{r-i}^i(\Ta)$ defined by the dofs of $V_{r-i}^i(\Ta)$ set
  above. Then, for $r\ge 5$ the following diagram commutes
 \[
   \begin{tikzcd}
     \mathbb{R} \ar{r}
     & C^\infty(\bar T)  \ar{r}{\grad} \ar{d}{\Pi_0^V}
     & \left[C^\infty(\bar T)\right]^3  \ar{r}{\curl} \ar{d}{\Pi_1^V}
     & \left[C^\infty(\bar T)\right]^3 \ar{r}{\dive} \ar{d}{\Pi_2^V}
     & C^\infty(\bar T)  \ar{r} \ar{d}{\Pi_3^V}
     & 0
     \\
     \mathbb{R} \ar{r}
     & V_r^0(\Ta) \ar{r}{\grad} 
     & V_{r-1}^1(\Ta) \ar{r}{\curl}
     & V_{r-2}^2(\Ta) \ar{r}{\dive}
     & V_{r-3}^3(\Ta) \ar{r}
     & 0. 
   \end{tikzcd}
 \]
\end{theorem}

\section{The second exact sequence}\label{sec:second}

In this section, we introduce the second exact sequence with smoother
spaces required for the construction of the elasticity complex in the
next section.
Let
\begin{alignat*}{1}
Z_{r}^0(\Ta) &=S_{r}^0(\Ta), \\
Z_{r}^1(\Ta) &=\{ \omega \in S_r^1(\Ta):   \text{ $\curl \omega$ is $C^1$ at vertices of $T$}\}, \\
Z_{r}^2(\Ta) &= \{ \omega \in L_r^2(\Ta):  \text{ $\omega$ is $C^1$ at vertices of $T$}\}, \\
Z_{r}^3(\Ta)& =\{ \omega \in W_r^3(\Ta): \text{ $\omega$ is $C^0$ at vertices of $T$} \}.
\end{alignat*}
The corresponding spaces with boundary conditions are defined as follows:
\begin{alignat*}{1}
  \mZ_{r}^0(\Ta) &= \mathring{S}_{r}^0(\Ta),
\\
\mZ_{r}^1(\Ta) &=\{ \omega \in \mathring{S}_r^1(\Ta):\;  \grad \curl \omega=0  \text{ at vertices of } T \}, \\
\mZ_{r}^2(\Ta) &= \{ \omega \in \mathring{L}_r^2(\Ta): \; \grad \omega =0  \text{ at vertices of } T   \}, \\
\mZ_{r}^3(\Ta)& =\{ \omega \in \mathring{W}_r^3(\Ta): \; \omega=0 \text{ at vertices of } T \}.
\end{alignat*}
It is easy to verify from these definitions that 
\begin{equation}
  \label{eq:Zcontains}
  \grad Z_{r}^0(\Ta) \subseteq Z_{r-1}^1(\Ta),
  \quad
  \curl Z_{r}^1(\Ta) \subseteq Z_{r-1}^2(\Ta),
  \quad
  \dive Z_{1}^2(\Ta) \subseteq Z_{r-1}^3(\Ta),
\end{equation}
and that similar inclusions hold for the $\mZ$ spaces with boundary
conditions.

\subsection{Exactness and dimensions}

It is obvious from the above definitions that,
we have, as equalities of sets,  the identities
\begin{equation}\label{z2=v1}
Z_r^2(\Ta)=V_r^1(\Ta) \quad \text{ and  }  \quad \mZ_r^2(\Ta)=\mV_r^1(\Ta), 
\end{equation}
even though their form degrees do not match. 
The next relationships are also interesting.
\begin{lemma} \label{lem:ZofromZ}
The following identities hold:
\begin{align}\label{iden-1}
\mZ_{r}^1(\Ta) &=Z_r^1(\Ta) \cap  \mathring{S}_r^1(\Ta),  \\ \label{iden-2}
\mZ_{r}^2(\Ta) &=  Z_{r}^2(\Ta)  \cap \mathring{L}_r^2(\Ta) .
\end{align}
\end{lemma}
\begin{proof}
  By \eqref{z2=v1}, the equality of sets in\eqref{iden-2}
  is immediate since it is  exactly the left
  equality in \eqref{eq:Vo1Vo2cap}. To show \eqref{iden-1}, it
  suffices to prove that
  $\mathring{S}_r^1(\Ta) \cap Z_r^1(\Ta)\subset \mZ_{r}^1(\Ta)$. Let
  $\omega \in \mathring{S}_r^1(\Ta) \cap Z_r^1(\Ta) \subset
  \mathring{L}_r^1(\Ta) \cap V_r^1(\Ta)$.
  Then $\curl \omega \in \mathring{L}_{r-1}^2(\Ta) \cap
  Z_{r-1}^2(\Ta)$, so  by
  \eqref{iden-2},  $\curl \omega \in \mZ_{r-1}^2(\Ta)$. 
  Hence $\grad \curl \omega=0$  at the vertices of $T$ and   we
  conclude that $\omega \in \mZ_{r}^1(\Ta)$.
\end{proof}

Now, we turn to the exactness properties of the
sequences of $Z$ and $\mZ$ spaces. Note that $\mZ_{r+1}^0(\Ta) =
\mS_{r+1}^0(\Ta)$ is nontrivial only for $r\ge 4$.

\begin{lemma}
  \label{lem:Zo_exactnesssec}
  The following sequence is exact for any $r \ge 4$.
  \begin{equation}
    \label{mlocalZ}
        \begin{tikzcd}[column sep=small]
      0\arrow{r}{} &[1em]
      \mZ_{r+1}^0(\Ta)\arrow{r}{\grad} &[1em]
      \mZ_{r}^1(\Ta)\arrow{r}{\curl} &[1em]
      \mZ_{r-1}^2(\Ta)\arrow{r}{\div} &[1em]
      \mZ_{r-2}^3(\Ta)\arrow{r}{} &0.
    \end{tikzcd} 
  \end{equation}
\end{lemma}

\begin{proof}
  We shall only prove the exactness of
  $
  \begin{tikzcd}[cramped]
    \mZ_{r-1}^2(\Ta)\arrow{r}{\div} &
    \mZ_{r-2}^3(\Ta)\arrow{r}{} &0,
  \end{tikzcd}
  $ as all the other exactness properties can be shown easily from the
  known exactness of the $\mathring{S}$ sequence, i.e., the analogue
  of \eqref{exact} with boundary conditions, and inclusions similar
  to~\eqref{eq:Zcontains} for the $\mZ$ sequence. To show
  that
  $\dive: \mZ_{r-1}^2(\Ta) \to \mZ_{r-2}^3(\Ta) $ is surjective, let
  $\rho \in \mZ_{r-2}^3(\Ta) \subset \mathring{W}_{r-2}^3(\Ta)$.  By
  Proposition~\ref{prop:rep-div0}, there is an $\omega$ of the form
  $\omega = \mu \psi$ where
  $\psi = \sum_{\ell=2}^{r-2} \mu^\ell \gamma_{r-2-\ell}$ and
  $\gamma_{r-2-\ell} \in [\pol_{r-2-\ell}(T)]^3$, such that
  $\dive \omega=\rho$. The last equality can be rewritten as
  \[
    \grad \mu \cdot \psi = \rho - \mu \,\dive \psi.
  \]
  At the vertices of $T$, all the limit values of the right hand side
  above vanish, as $\rho$ is in $\mZ_{r-2}^2(\Ta)$. Since the three
  $\grad \mu$ vectors on the three faces of $\partial T$
  that meet at a vertex are linearly
  independent, we conclude that $\psi$ vanishes at the vertices of
  $T$. Thus, by the product rule,
  $\grad \omega = \grad (\mu \psi)$ also vanishes at the vertices of
  $T$. Hence, $\omega \in \mZ_{r-1}^2(\Ta)$.
\end{proof}

Let $\ell_i$ for each $i=0,\ldots, 3,$ denote a quadratic function in $\pol_2(T)$ satisfying
\begin{equation}
  \label{eq:ell2}
  \ell_i(x_j) = \delta_{ij}, \qquad \int_T \ell_i\, {\dx} = 0.
\end{equation}
E.g., one may set $\ell_i$ as the sum of a linear function and an edge
bubble, $\ell_i= \lambda_i + c \lambda_j \lambda_k$  (using
the barycentric coordinates $\lambda_m$ of $T$) for some $j \ne k$
where the scalar $c$ is chosen to make $\ell_i$ mean-free.

\begin{lemma}
  \label{lem:Z_exact}
  The following sequence is exact for  $r \ge 4$.  
  \begin{equation}\label{localZ}
    \begin{tikzcd}[column sep=small]
      0\arrow{r}{} &\mathbb{R}\arrow{r}{\subset} &[1em]
      Z_{r+1}^0(\Ta)\arrow{r}{\grad} &[1em]
      Z_{r}^1(\Ta)\arrow{r}{\curl} &[1em]
      Z_{r-1}^2(\Ta)\arrow{r}{\div} &[1em]
      Z_{r-2}^3(\Ta)\arrow{r}{} &0.
    \end{tikzcd} 
  \end{equation}
\end{lemma}
\begin{proof}
  All the exactness properties, except the last, follow immediately
  from the exactness of the $S$ sequence. To prove the surjectivity of
  $\dive : Z_{r-1}^2(\Ta) \to Z_{r-2}^3(\Ta)$, let
  $\rho \in Z_{r-2}^3(\Ta)$. Using the $\ell_i$ in~\eqref{eq:ell2},
  set $\rho_2 = \sum_{i=0}^3 \ell_i \rho(x_i),$ which
  is in $Z_{r-2}^3(\Ta)$ as $r -2 \ge 2$. Moreover, $\rho - \rho_2$ is
  in $\mZ_{r-2}^3(\Ta)$. Hence, using Lemma~\ref{lem:Zo_exactnesssec},
  we can find $\omega_1 \in \mZ_{r-1}^2(\Ta)$ such that
  $\dive \omega_1=\rho - \rho_2$. By standard exactness results, there
  is an $\omega_2 \in [\pol_{3}(T)]^3$ such that
  $\dive \omega_2= \rho_2$.  Hence,
  $\omega=\omega_1+\omega_2 \in Z_{r-1}^2(\Ta)$ satisfies
  $\dive \omega=\rho$.
\end{proof}

\begin{lemma}
  \label{lem:dimZ}
  The dimensions of the $\mZ$ and $Z$ spaces are as follows.
  \begin{alignat}{5}
    \label{eq:dimZ-0}
    \dim \mZ_r^0(\Ta) 
    & = \dim \mathring{S}_r^0(\Ta), 
    &
    \dim Z_r^0(\Ta)
    &=     \dim S_r^0(\Ta), &\quad(r \ge 1),
    \\
    \label{eq:dimZ-1}
    \dim \mZ_r^1(\Ta)
    & = (2r-3)(r-3)(r-2), 
    &\quad
    \dim Z^1_r(\Ta)
    &=  2r^3 - 3r^2 + 13 r -4  &(r \ge 4).
    \intertext{The following  formulas for
      the dimensions of  $\mZ^2_r(\Ta)$ and $\mZ^3_r(\Ta)$ 
      hold for $r \ge 2$, while those for $Z^2_r(\Ta)$ and 
      $ Z^3_r(\Ta)$ hold for $r \ge 1$:}
    \label{eq:dimZ-2}
    \dim \mZ^2_r(\Ta)
    & =2 r^3 -3 r^2 + 7r -15, 
    &
    \dim Z^2_r(\Ta)
    &=
    2r^3 + 3r^2 + 7 r - 9, 
    \\     \label{eq:dimZ-3}
    \dim \mZ^3_r(\Ta) & =
    \dim W_r^3(\Ta) - 13,
    &
    \dim Z^3_r(\Ta)
    &= \dim W_r^3(\Ta) - 8.
  \end{alignat}
\end{lemma}
\begin{proof}
  Dimensions in~\eqref{eq:dimZ-0} are obvious.  Those
  in~\eqref{eq:dimZ-2} are also obvious  from \eqref{z2=v1} and
  Lemma~\ref{lem:dimVoV}.  We proceed to prove~\eqref{eq:dimZ-3}
  followed by~\eqref{eq:dimZ-1}.
  Let
  \begin{equation}
    \label{eq:Zhat}
    \hat{Z}^3_r (\Ta) = \{ w \in W_r^3(\Ta): \; w(x_i) =  0\}.     
  \end{equation}
  We claim that
  \begin{equation}
    \label{eq:Z3}
    Z_r^3(\Ta) = \hat{Z}^3_r(\Ta) \oplus \pol_1(T), \qquad \text{ for
    } r \ge 1.
  \end{equation}
  Indeed, given any $z \in Z_r^3(\Ta)$, constructing a
  $\lambda = \sum_{i=0}^4 \lambda_i z(x_i)$, and putting
  $\hat z = z - \lambda$, we have the decomposition
  $ z = \hat z + \lambda,$ with $\hat z \in \hat{Z}^3_r(\Ta)$ and
  $\lambda \in \pol_1(T).$ This proves~\eqref{eq:Z3} since the reverse
  inclusion is obvious. It is easy to count the dimension of
  $\hat{Z}^3_r (\Ta)$: the constraints $w(x_i)=0$ form three
  linearly independent constraints 
  at each $x_i$, so 
  $
    \dim \hat{Z}^3_r (\Ta) = \dim W_r^3(\Ta) - 12.
  $
  Then,~\eqref{eq:Z3} yields
  \[
    \dim Z_r^3(\Ta)
    = \dim \hat{Z}^3_r (\Ta)  + 4 = \dim W_r^3(\Ta) - 8.
  \]

  To count the dimension of $\mZ_r^3(\Ta)$, let us start with
  $\tZ_r^3 (\Ta) = \{ w \in Z_r^3(\Ta): \int_T w \, {\dx} =
  0\},$ whose dimension is obviously $\dim Z_3^r(\Ta) - 1$.  We claim
  that, with
  $\mpol_2(T) = \text{span}\{\ell_0, \ell_1, \ell_2, \ell_3\}$, where
  $\ell_i$ is as in~\eqref{eq:ell2}, we have
  \begin{equation}
    \label{eq:Zo3}
    \tZ_r^3(\Ta) = \mZ_r^3(\Ta) \oplus \mpol_2(T), \qquad \text{ for }
    r \ge 2.
  \end{equation}
  Indeed, given any $\tilde z \in \tZ_r^3(\Ta)$, setting
  $\ell = \sum_{i=0}^4 \ell_i \tilde z(x_i)$ and
  $\mathring{z} = \tilde z - \ell$, we have the decomposition
  $\tilde z = \mathring z + \ell$ with $\mathring z \in \mZ_r^3(\Ta)$
  and $\ell \in \mpol_2(T)$. Combined with the obvious reverse
  inclusion, we obtain~\eqref{eq:Zo3}. Consequently, 
  $
    \dim \mZ^3_r(\Ta)
     = \dim \tZ_3^r(\Ta) - \dim \mpol_2(T)
     = \dim Z_3^r(\Ta) - 5.
  $
  This finishes the proof of~\eqref{eq:dimZ-3}.

  It only remains to prove~\eqref{eq:dimZ-1}, for which we use the
  already proved exactness. Restricting to $r \ge 4$ in order to apply
  Lemma~\ref{lem:Zo_exactnesssec},  the rank-nullity theorem gives 
  \begin{alignat*}{1}
    \dim  \mZ_r^1(\Ta)
    & = \dim \curl (\mZ_r^1(\Ta))
    + \dim  \grad (\mZ_{r+1}^0(\Ta))
    \\
    & = \dim  \mZ_{r-1}^2(\Ta)-\dim  \mZ_{r-2}^3(\Ta)
    + \dim  \mZ_{r+1}^0(\Ta).
  \end{alignat*}
  This proves one of the equalities stated in~\eqref{eq:dimZ-1}. In
  the same way, Lemma~\ref{lem:Z_exact} yields
  $
    \dim  Z_r^1(\Ta)
    =  \dim  Z_{r-1}^2(\Ta)-\dim  Z_{r-2}^3(\Ta)
    + \dim  Z_{r+1}^0(\Ta) -1,
  $
  whenever $r \ge 4$, thus proving the other equality.
\end{proof}

\subsection{Degrees of freedom and commuting projections for the $Z$ spaces}

The degrees of freedom for $Z_{r+1}^0(\Ta)$ are simply  given by
\eqref{eqn:Sigmadofs} with $r$ replaced with $r + 1$, so we start with
those of $ Z_{r}^1 (\Ta)$. We shall show that
any $\omega \in Z_{r}^1 (\Ta)$, $r \ge 4$, is
uniquely determined by  
the following dofs:
\begin{subequations}
\label{Z1dofs}
\begin{alignat}{4}
\label{Z1dof1}
&D^\alpha \omega(a), \; D^\beta \curl \omega(a),\qquad && |\alpha|\le 1,\; |\beta|=1, \;
  a\in \Delta_0(T),\quad &&\dofcnt{($80$ dofs)}\\
\label{Z1dof2}
&\int_e \omega \cdot\kappa\,{\ds},\quad && \kappa\in
[\pol_{r-4}(e)]^3,\; e\in \Delta_1(T),\quad &&
\dofcnt{($18(r-3)$ dofs)}
\\
\label{Z1dof3}
&\int_e (\curl \omega)\cdot \kappa\, {\ds},\quad
&& \kappa\in [\pol_{r-5}(e)]^3,\; e\in \Delta_1(T),\quad &&
\dofcnt{($18(r-4)$ dofs)}\\
\label{Z1dof4}
&\int_F (\omega \cdot n^\F) \kappa\, {\dA}, \quad && \kappa\in
\pol_{r-3}(F),\; F\in \Delta_2(T),\qquad &&\dofcnt{($2(r-2)(r-1)$ dofs)}\\
\label{Z1dof5}
&\int_F (n^\F \times (\omega \times n^\F)) \cdot \kappa\, {\dA}, \quad && \kappa\in {D}_{r-4}(F),\; F\in \Delta_2(T),\qquad &&\dofcnt{($4(r-2)(r-4)$ dofs)}\\
\label{Z1dof6}
&\int_F (\curl \omega \times n^\F) \cdot \kappa\, {\dA}, \quad &&
\kappa\in [\pol_{r-4}(F)]^2, \;  F\in \Delta_2(T),\quad &&\dofcnt{($4(r-3)(r-2)$ dofs)}\\
\label{Z1dof7}
&\int_T \omega \cdot \kappa \, {\dx}, \quad && \kappa\in \grad
\mZ_{r+1}^{0}(\Ta), && \dofcnt{($\frac 2 3 (r-3)(r-2)(r-1)$ dofs)}
\\
\label{Z1dof8}
&\int_T \curl \omega \cdot \kappa \, {\dx}, \quad && \kappa\in 
\curl \mZ_r^1(\Ta),
&& \dofcnt{($\frac 1 3  (r - 3) (r - 2) (4 r - 7)$ dofs)}.
\end{alignat}
\end{subequations}
Here \eqref{Z1dof1} counts as 80 dofs, rather than 84, since at each
vertex we have the identity $\div\curl \omega=0$.  In~\eqref{Z1dof5},
the space $D_{r-4}(F)$ is the Raviart-Thomas space (defined
in~\eqref{eq:RT}).  In~\eqref{Z1dof6}, we have committed the usual
abuse and written that $\kappa$ is in $ [\pol_{r-4}(F)]^2$ instead of
the (isomorphic) tangent plane of $F$. The count of dofs
in~\eqref{Z1dof8} follows from Lemmas~\ref{lem:Zo_exactnesssec}
and~\ref{lem:dimZ}.

For $\omega\in Z_{r-1}^2(\Ta),$ $r\ge 4$, we define
the following dofs:
\begin{subequations}
\label{Z2dofs}
\begin{alignat}{4}
\label{Z2dof1}
&D^\alpha \omega(a),\qquad &&|\alpha|\le 1\quad && a\in \Delta_0(T),\qquad && \dofcnt{($48$ dofs)}\\
\label{Z2dof2}
&\int_e \omega\cdot\kappa\,{\ds}, \quad && \kappa\in [\pol_{r-5}(e)]^3 &&  e\in \Delta_1(T),\qquad &&
\dofcnt{($18(r-4)$ dofs)}\\
\label{Z2dof3}
&\int_F \omega \cdot \kappa\,{\ds}, \quad && \kappa\in [\pol_{r-4}(F)]^3 &&  F\in \Delta_2(T),\quad 
&& \dofcnt{($6(r-3)(r-2)$ dofs)}\\
\label{Z2dof4}
&\int_T \omega \cdot \kappa\, {\dx}, \quad && \kappa\in \curl \mZ_r^1(\Ta)
&& &&\dofcnt{($\frac 1 3  (r - 3) (r - 2) (4 r - 7)$ dofs)}
\\
\label{Z2dof5}
&\int_T (\dive \omega)\,  \kappa\, {\dx}, \qquad && \kappa\in \dive \mZ_{r-1}^2(\Ta)
&& && \dofcnt{($\frac 2 3 (r+1)r(r-1)-13$ dofs).} \hspace{-2.25cm} 
\end{alignat}
\end{subequations}
The counts in~\eqref{Z2dof4} and~\eqref{Z2dof5} follow from 
Lemmas~\ref{lem:Zo_exactnesssec} and~\ref{lem:dimZ}.

For any $r \ge 4$,  a function $\omega \in Z_{r-2}^3 (\Ta)$
is uniquely determined by
the following dofs:
\begin{subequations}
\label{Z3dofs}
\begin{alignat}{4}
\label{}
&\omega(a),  \quad && a \in \Delta_0(T) \quad && \dofcnt{($4$ dofs)}\\
&\int_T \omega \, {\dx}, &&\quad && \dofcnt{($1$ dof)}\\
\label{}
&\int_T \omega \, \kappa \, {\dx}, \qquad &&\kappa\in
\mZ_{r-2}^3(\Ta), \qquad && \dofcnt{($\frac 2 3 (r+1)r(r-1)-13$ dofs).}
\end{alignat}
\end{subequations}
The unisolvency of~\eqref{Z3dofs} is obvious  
from our definition of $\mZ_{r-2}^3(\Ta)$, so we
shall now focus on proving the
unisolvency of the dofs in \eqref{Z2dofs} and \eqref{Z1dofs}.

\begin{lemma}
  \label{lem:Z2unisolvency}
For any $r \ge 4$, the dofs \eqref{Z2dofs} uniquely determine a function in $Z_{r-1}^2(\Ta)$.
\end{lemma}
\begin{proof}
The total number of dofs in \eqref{Z2dofs} is exactly the dimension of
the space  $Z_{r-1}^2(\Ta)$. Hence, we only need to show that if
$\omega \in Z_{r-1}^2(\Ta)$ and the dofs in~\eqref{Z2dofs} vanish,
then $\omega \equiv 0$. The dofs in \eqref{Z2dof1} and \eqref{Z2dof2}
make $\omega$ vanish on all the edges of $T$. Then
the dofs in~\eqref{Z2dof3} make $\omega$ vanish on all the
faces of $T$. This, together with the zero first derivatives
of $\omega$ at vertices (due to~\eqref{Z2dof1}) shows that 
$\omega \in \mZ_{r-1}^2(\Ta)$. Then
\eqref{Z2dof5} implies   $\dive \omega=0$. Using the exactness of the sequence \eqref{mlocalZ} and the dofs \eqref{Z2dof4} we conclude that $\omega \equiv 0$. 
\end{proof}

\begin{lemma}
    \label{lem:Z1unisolvency}
For any $r \ge 4$, the dofs \eqref{Z1dofs} uniquely determine a function in $Z_r^1(\Ta)$. 
\end{lemma}
\begin{proof}
We first note that $\dim Z_{r}^1(\Ta)$ is equal to the total number of
dofs in \eqref{Z1dofs}. To prove unisolvency, consider an 
$\omega\in Z_{r}^1(\Ta)$ for which all
the dofs in \eqref{Z1dofs} vanish. Using the dofs
\eqref{Z1dof1}--\eqref{Z1dof5}, the inclusion $Z_r^1(\Ta) \subset
V_r^1(\Ta),$ and Remark \ref{remark121}, we conclude that  $\omega$
vanishes on $\partial T$. From \eqref{Z1dof1} and \eqref{Z1dof3} we
have that $\curl \omega$ vanishes on all the edges of $T$. Then using
\eqref{Z1dof6} we conclude that $\curl \omega \times n$ vanishes on
$\partial T$. Also, $\curl \omega \cdot n$ vanishes on $\partial T,$
since we have already established that $\omega$ vanishes on $\partial
T$. Thus $\curl \omega$ vanishes on $\partial T.$ Using the vertex
dofs again, we conclude that $\omega \in \mZ_r^1(\Ta)$. Now, the dofs
\eqref{Z1dof8} show that  $\curl \omega =0$ on $T$. By the exactness
property \eqref{mlocalZ}  and the dofs \eqref{Z1dof7} we conclude that $\omega \equiv 0$.
\end{proof}

\begin{theorem}
  \label{thm:PiZ}
  Let $\Pi_k^Z$ denote the canonical interpolant defined using
  the dofs of $Z_{r+1-k}^k(\Ta)$. Then, for $r\ge 5$,
  the following diagram commutes
  \[
    \begin{tikzcd}
      \mathbb{R} \ar{r}
      & C^\infty(T)  \ar{r}{\grad} \ar{d}{\Pi_0^Z}
      & \left[C^\infty(T)\right]^3  \ar{r}{\curl} \ar{d}{\Pi_1^Z}
      & \left[C^\infty(T)\right]^3 \ar{r}{\dive} \ar{d}{\Pi_2^Z}
      & C^\infty(T)  \ar{r} \ar{d}{\Pi_3^Z}
      & 0
      \\
      \mathbb{R} \ar{r}
      & Z_{r+1}^0(\Ta) \ar{r}{\grad} 
      & Z_{r}^1(\Ta) \ar{r}{\curl}
      & Z_{r-1}^2(\Ta) \ar{r}{\dive}
      & Z_{r-2}^3(\Ta) \ar{r}
      & 0. 
    \end{tikzcd}
  \]
\end{theorem}
\begin{proof}
  First, we show that for any $\rho \in C^\infty(T)$,
  $ \grad \Pi_0^Z \rho= \Pi_1^Z \grad \rho.$ We do this by showing
  that all dofs of \eqref{Z1dofs} vanish when applied to
  $u= \grad \Pi_0^Z \rho- \Pi_1^Z \grad \rho \in Z_{r}^1(\Ta)$ and
  applying Lemma~\ref{lem:Z1unisolvency}.

  Using the vertex dofs of $Z^0_{r+1}(\Ta)$---see
  \eqref{eqn:C1DOF1}---and \eqref{Z1dof1}, together with
  $\curl \grad=0,$ we see that the dofs of~\eqref{Z1dof1} vanish on
  $u$.  To show that the dofs of \eqref{Z1dof2} applied to $u$ on 
  edge $e$, namely
  $\int_e u \cdot \kappa \,{\ds},$ also vanish, we split
  $\kappa \in \pol_{r-4}(e)^3$ into tangential and normal components
  $\kappa = \kappa_e t_e + \kappa_+n^+_e + \kappa_-n_e^-$ and proceed.
  By~\eqref{eqn:Sigmadofs} with $r+1$ in place of $r$, 
\begin{align*}
  \int_e  \grad (\Pi_0^Z \rho) \cdot  n_e^\pm \, \kappa_\pm\,{\ds}
  & =\int_e  \grad \rho \cdot  n_e^\pm \, \kappa_\pm\,{\ds},
  && \text{by~\eqref{eqn:C1DOF22}},\\
  \int_e  \grad (\Pi_0^Z \rho-\rho) \cdot  t_e \; \kappa_e\,{\ds}
  & =-\int_e  (\Pi_0^Z \rho-\rho) \; \partial_{t_e} \kappa_e
    \,{\ds}=0,
  &&  \text{by~\eqref{eqn:C1DOF1}--\eqref{eqn:C1DOF21}}.
\end{align*}
Together with~\eqref{Z1dof2}, we conclude that
$\int_e u \cdot \kappa \,{\ds} = \int_e(\grad \Pi_0^Z \rho-
\Pi_1^Z \grad \rho)\cdot \kappa \,{\ds} = 0.$
Proceeding to the dofs of  \eqref{Z1dof3} applied to $u$, for any
$\kappa\in [\pol_{r-5}(e)]^3$, 
\begin{equation*}
\int_e (\curl  u)\cdot \kappa\, {\ds}= -\int_e (\curl \Pi_1^Z \grad \rho)\cdot \kappa\, {\ds}=-\int_e (\curl \grad \rho)\cdot \kappa\, {\ds}=0.
\end{equation*}
Thus all vertex and edge dofs vanish when applied to $u$.

Next, consider the face and inner dofs. Let 
$w_{\F} = {n^\F} \times (w \times n^{\F})$ denote the tangential component of a
vector field $w$ on $F$, let $\gradF, \curlF, \rotF,$ and $\divF$ denote
the standard surface differential operators on $F$, and let $\nu^{\F}$
denote the outward unit normal on $\partial F$ lying in the tangent
plane of $F$.
It is easy to see that the dofs of \eqref{Z1dof4}  vanish on $u$
using~\eqref{eqn:C1DOF32} with $r+1$ in place of $r$. For the dofs in
\eqref{Z1dof5}, let $\kappa\in {D}_{r-4}(F)$ and let $\nu^{\F}$ denote
the outward unit normal on $\partial F$ lying in the tangent plane of
$F$. By the properties of the Raviart-Thomas space, $\nu^{\F} \cdot
\kappa|_e$ is in $\pol_{r-5}(e)$.  
Applying  \eqref{Z1dof5} and then integrating by parts,
\begin{align*}
  \int_F  u_{\F} \cdot \kappa \, {\dA}
   &= \int_F  \gradF (\Pi_0^Z \rho-\rho) \cdot
    \kappa\, {\dA}
    \\
   & =-\int_F (\Pi_0^Z \rho-\rho)  \cdot \divF \kappa\,
    {\dA}
    +\int_{\partial F} (\Pi_0^Z \rho-\rho)  \nu^{\F} \cdot \kappa\,
    {\ds}. 
\end{align*}
Applying~\eqref{eqn:Sigmadofs} with $r+1$ in place of $r$, the last
term above vanishes due to~\eqref {eqn:C1DOF21} and the penultimate
term vanishes due to \eqref{eqn:C1DOF31}.  Hence the dofs of
\eqref{Z1dof5} applied to $u$ vanish.  The dofs of \eqref{Z1dof6} and
\eqref{Z1dof8} applied to $u$ are, of course, zero simply because
$\curl \grad$ vanishes. Finally, the dofs of \eqref{Z1dof7} applied to
$u$ yield zero due to \eqref{eqn:C1DOF4}.

Let us proceed to show the second commuting diagram property, namely
$ \curl \Pi_1^Z \rho= \Pi_2^Z \curl \rho$ for all
$\rho \in [C^\infty(T)]^3.$ Putting
$u=\curl \Pi_1^Z \rho- \Pi_2^Z \curl \rho \in Z_{r-1}^2(\Ta)$, we now
show that all dofs in~\eqref{Z2dofs} vanish on $u$. (Then the result
follows from Lemma~\ref{lem:Z2unisolvency}.) It is easy to see that
the vertex dofs applied to $u$ are zero due to \eqref{Z1dof1} and
\eqref{Z2dof1}, and that the edge dofs applied to $u$ are zero due to
\eqref{Z1dof3} and~\eqref{Z2dof2}.  For the face dofs applied to $u$,
namely $\int_F u \cdot \kappa\, {\ds}$, we proceed by splitting
$\kappa$ into its normal component $\kappa_n n^{\F}$ and the remaining
tangential component $\kappa_t = n \times (\kappa \times n)$. The
latter gives
\begin{align*}
\int _F u \cdot \kappa_t  = \int_F   u \times n^{\F} \cdot \kappa
  \times n^{\F} =
  \int_F (\curl \Pi_1^Z \rho) \times n^{\F} \cdot \kappa \times n^{\F}
  - \int_F (\Pi_2^Z \curl \rho ) \times n^{\F} \cdot  \kappa \times n^{\F}.
\end{align*}
Applying \eqref{Z2dof3}  to the last term and \eqref{Z1dof6} to the
penultimate term, the result is zero. To show that the face
dofs with the normal component are also zero, we use~\eqref{Z2dof3}
and integrate by parts:
\begin{align*}
  \int_F u \cdot  \kappa_nn^{\F}
  & = \int_F \curl (\Pi_1^Z \rho - \rho) \cdot n^{\F} \kappa_n
    = \int_F \curlF (\Pi_1^Z \rho - \rho)_{\F}
    \, \kappa_n
    = \int_F  ( \Pi_1^Z \rho - \rho)_{\F}
    \, \rotF\kappa_n,  
\end{align*}
where the boundary terms arising from the integration by parts were
zeroed out due to~\eqref{Z1dof2}. The last term above vanishes
by~\eqref{Z1dof5}. Hence all the face dofs
$\int_F u \cdot \kappa\, {\ds}$ vanish. It is easy to see that
the inner dofs in~\eqref{Z2dof4} and~\eqref{Z2dof5} applied to $u$
also yield zero. Hence $u=0$. 

The final commuting diagram property stated in the theorem,
$ \dive \Pi_2^Z \rho= \Pi_3^Z \dive \rho, $ for all
$\rho \in [C^\infty(T)]^3,$ is also proved using similar arguments
using the previously established lemmas.
\end{proof}

\section{Local elasticity complexes}\label{sec:local}

\subsection{Derived exact sequences}

The two exact sequences of spaces, \eqref{localV} and \eqref{localZ},
that we have developed in the previous sections can now be put
together to deduce an elasticity sequence. We do so  by applying
Proposition~\ref{prop:bgg_basic}. To this end, we need connecting
operators between the sequences as in~\eqref{eq:5}.
Let $\M$ denote the space of $3\times 3$
matrices. Define $\Xi : \M \to \M$ by 
$
\Xi M=M' -\text{tr}(M) \mathbb{I},
$
where $\mathbb{I}$ denotes the identity matrix and $(\cdot)'$ denotes the transpose. We note that this operator is invertible and 
\begin{equation}
  \label{eq:8}
\Xi^{-1} M=M'-\frac{1}{2}\text{tr}(M) \mathbb{I}.
\end{equation}
As usual, we let
$  \text{sym } M= \frac{1}{2}(M+M'),$
  $  \text{skw } M= \frac{1}{2}(M-M'),$
and put $\K = \skw(\M)$ and $\SSS = \sym(\M)$.
We define  $\mskw : \V \to \K$ by
$$
 \text{mskw} \begin{pmatrix}
v_1 \\
v_2 \\
v_3
\end{pmatrix}
 =
  \begin{pmatrix}
0  &  -v_3 & v_2 \\
v_3 & 0 & -v_1 \\
-v_2& v_1 & 0 
\end{pmatrix}
$$
and set  $\text{vskw}= \text{mskw}^{-1} \circ \text{skw }$. It is easy
to check that the following two identities hold:
\begin{gather}
\label{alg1}
  \dive \Xi= 2\text{vskw } \curl,
  \\
  \label{alg2}
  \Xi \grad=-\curl \text{ mskw}.
\end{gather}
These identities imply that the following diagram commutes:
\begin{equation}\label{2level-diagram}
\begin{tikzcd}
Z_{r+1}^0(\Ta) \otimes {\V}\arrow{r}{\grad} & Z_{r}^1(\Ta)
\otimes {\V} \arrow{r}{\curl} & Z_{r-1}^2(\Ta)  \otimes
{\V}\arrow{r}{\dive} &  Z_{r-2}^3(\Ta) \otimes
{\V}
\\
V_{r}^0(\Ta) \otimes {\V}\arrow{r}{\grad}\arrow[ur,
"-\mathrm{mskw}"]& V_{r-1}^1(\Ta) \otimes {\V}
\arrow{r}{\curl}\arrow[ur, "\Xi"] &V_{r-2}^2(\Ta) \otimes
{\V} \arrow{r}{\dive}\arrow[ur, "2\mathrm{vskw}"] &
V_{r-3}^3(\Ta) \otimes {\V}.
 \end{tikzcd} 
\end{equation}

\begin{lemma}\label{lem:skw-onto}
   Both 
   $\vskw: \mV_{r-2}^2(\Ta)\otimes {\V}\rightarrow \hat
   Z_{r-2}^3(\Ta)\otimes {\V}$ and
   $\vskw: V_{r-2}^2(\Ta)\otimes {\V}\rightarrow 
   Z_{r-2}^3(\Ta)\otimes {\V}$ 
   are surjective operators. 
\end{lemma}
\begin{proof}
  Given any $z\in \hat Z_{r-2}^3(\Ta)\otimes {\V}$, we must find some
  $v\in \mV_{r-2}^2(\Ta)\otimes {\V}$ such that $2\vskw v=z$. First,
  we take any $\tilde{v}\in \mV_{r-2}^2(\Ta)\otimes {\V}$ satisfying
  $\int_{T}2\vskw\tilde{ v}\, {\dx}=\int_{T}z\,
  {\dx}$. Then consider
  $z-2\vskw \tilde{ v}\in \mZ_{r-2}^{3}(\Ta)\otimes {\V}$. Due to the
  exactness of the $\mZ$ sequence \eqref{mlocalZ}, there exists
  $w\in \mZ_{r-1}^{2}(\Ta)\otimes {\V}$ such that
  $\div w=z-2\vskw \tilde{ v}$.  Then $ v=\curl\Xi^{-1}w+\tilde{ v}$
  is in $ \mV_{r-1}^{2}\otimes {\V}$ and $2\vskw( v)=z$
  by~\eqref{alg1}. The proof of the other surjectivity is similar and
  easier.
\end{proof}

\begin{theorem} \label{thm:elasticitysequence}
  The sequence \vspace{-0.3cm}  
  \[ 
    \begin{tikzcd}
      \!\!
      \begin{bmatrix}
        Z_{r+1}^0(\Ta) \!\otimes\! {\V} \\
        V_{r}^0(\Ta) \!\otimes\! {\V}
      \end{bmatrix}\!\!
      \arrow{r}{
        \left[\begin{smallmatrix}
            \!\grad\!, \; -\!\mskw \!
          \end{smallmatrix}\right]}
      &[3.2em]
      \!\! Z_{r}^1(\Ta)  \!\otimes\! {\V} \!\!
      \arrow{r}{\curl \Xi^{-1} \curl}
      &[2.6em] 
      \!\!V_{r-2}^2(\Ta) \!\otimes\! {\V}\!\!
      \arrow{r}{
        \begin{bmatrix}
          \!2\! \vskw\!\! \\ \dive 
        \end{bmatrix}
      }
      &[0.9em]
      \begin{bmatrix}
        Z_{r-2}^3(\Ta) \!\otimes\! {\V} \\
        V_{r-3}^3(\Ta) \!\otimes\! {\V}
      \end{bmatrix}
    \end{tikzcd} 
  \]
  is exact for $r \ge 4$. Moreover, the last operator is surjective. 
\end{theorem}
\begin{proof}
  Identities~\eqref{eq:8} and \eqref{z2=v1} imply that
  $\Xi: V_{r-1}^1(\Ta)\otimes \V \rightarrow Z_{r-1}^2(\Ta) \otimes
  \V$ is a bijection. The exactness of the top and bottom sequences
  in~\eqref{2level-diagram} were established in
  Lemmas~\ref{lem:Z_exact} and \ref{lem:exactV}, for $r\ge 4$ and
  $r\ge 3$, respectively.  Hence
  the proof of the stated exactness reduces to an application of
  Proposition~\ref{prop:bgg_basic}. The statement on surjectivity of 
  the last map also follows from Proposition~\ref{prop:bgg_basic}
  after using Lemma~\ref{lem:skw-onto}.
\end{proof}

Note that the kernel of $[\grad, -\mskw]$ is
$\{(a-b\wedge x, -b) : a, b\in {\V}\}$.  Indeed, let
$(u, v) \in\ker [\grad, -\mskw]$. Then $\grad u-\mskw v=0$ implies
that ${\varepsilon}(u):=\sym\grad u=0$.  Therefore
$\left .u\right|_{T}=a_{T}+b_{T}\wedge x$ is an infinitesimal rigid
body motion for some constant vectors $a_{T}$ and $b_{T}$ on each
$T\in \Ta$. For any face $F=\partial T_{1}\cap \partial T_{2}$ with
$T_{1}, T_{2}\in \Ta$, if we choose the origin of the Euler vector
field $x$ on $F$, then $a_{T_{1}}=a_{T_{2}}$, $b_{T_{1}}=b_{T_{2}}$
due to the $C^{0}$ continuity of $u$. This implies that
$\left .u\right|_{T_{1}}$ and $\left .u\right|_{T_{2}}$ are the
restriction of the same affine vector field to $T_{1}$ and $T_{2}$,
respectively.  Therefore globally $u=a+b\wedge x$, for some constant
vectors $a, b \in {\V}$. Then we obtain $v=-b$.  The
conclusion also follows from a general algebraic construction (c.f.,
\cite{arnold2020complexes}). With this algebraic machinery, one can
show that the cohomology of the elasticity complex of
Theorem~\ref{thm:elasticitysequence} is isomorphic to the product of the
cohomology of the $Z$-complex and the $V$-complex. In particular, one
obtains the kernel of $[\grad, -\mskw]$ to be isomorphic to
${\V}\times {\V}$ since the kernel of
$\grad: Z_{r+1}^0(\Ta) \otimes {\V} \to Z_{r}^1(\Ta) \otimes
{\V}$ and the kernel of
$\grad: V_{r}^0(\Ta) \otimes {\V} \to V_{r-1}^1(\Ta) \otimes
{\V}$ are both equal to ${\V}$.

An analogous result holds for the spaces with boundary conditions. To
see this, first note that instead of \eqref{2level-diagram}, we now
have the following commuting diagram:
\[
\begin{tikzcd}[column sep=small]
  \mZ_{r+1}^0(\Ta) \!\otimes\! {\V}\arrow{r}{\grad}
  &[0.6em] \mZ_{r}^1(\Ta)
  \!\otimes\! {\V} \arrow{r}{\curl}
  &[0.6em] \mZ_{r-1}^2(\Ta)  \!\otimes\!
  {\V}\arrow{r}{\dive}
  &[0.5em]  \hat{Z}_{r-2}^3(\Ta) \!\otimes\!
  {\V} \arrow{r}{\int}
  &  \V 
  \\
  \mV_{r}^0(\Ta) \!\otimes\! {\V}\arrow{r}{\grad}\arrow[ur,
  "-\mathrm{mskw}"]& \mV_{r-1}^1(\Ta) \!\otimes\! {\V}
  \arrow{r}{\curl}\arrow[ur, "\Xi"] &\mV_{r-2}^2(\Ta) \!\otimes\!
  {\V} \arrow{r}{\dive}\arrow[ur, "2\mathrm{vskw}"] &
  \mV_{r-3}^3(\Ta) \!\otimes\! {\V}.
\end{tikzcd} 
\]
Here we have modified the $\mZ$ sequence slightly
from~\eqref{lem:Zo_exactnesssec}. We have used the space
$\hat Z^3_r(\Ta)$ defined in~\eqref{eq:Zhat} since
$\vskw(\mV_{r-2}^2(\Ta) \otimes {\V})$ is generally
contained only in $\hat{Z}_{r-2}^3(\Ta) \otimes {\V}$, not
$\mZ_{r-2}^3(\Ta) \otimes {\V}$. The last operator in the
top sequence is just $z \mapsto \int_T z \,dx$ and the sequence is
exact even with this modification. Hence the proof of the next result
is completely analogous to that of
Theorem~\ref{thm:elasticitysequence}. (Note that surjectivity
of the last map is not claimed in the theorem.)

\begin{theorem} \label{thm:elasticitysequence-0}
  The following sequence is exact for $r \ge 4$: 
  \[ 
    \begin{tikzcd}
      \!\!
      \begin{bmatrix}
        \mZ_{r+1}^0(\Ta) \!\otimes\! {\V} \\
        \mV_{r}^0(\Ta) \!\otimes\! {\V}
      \end{bmatrix}\!\!
      \arrow{r}{
        \left[\begin{smallmatrix}
            \!\grad\!, \; -\!\mskw \!
          \end{smallmatrix}\right]}
      &[3.2em]
      \!\! \mZ_{r}^1(\Ta)  \!\otimes\! {\V} \!\!
      \arrow{r}{\curl \Xi^{-1} \curl}
      &[2.6em] 
      \!\!\mV_{r-2}^2(\Ta) \!\otimes\! {\V}\!\!
      \arrow{r}{
        \begin{bmatrix}
          \!2\! \vskw\!\! \\ \dive 
        \end{bmatrix}
      }
      &[0.9em]
      \begin{bmatrix}
        \hat{Z}_{r-2}^3(\Ta) \!\otimes\! {\V} \\
        \mV_{r-3}^3(\Ta) \!\otimes\! {\V}
      \end{bmatrix}.
    \end{tikzcd} 
  \]
\end{theorem}

\subsection{Discrete elasticity complex}

We now proceed to
define an elasticity complex useful for designing mixed methods with
strongly imposed symmetry.  To do this, we first note that
$\curl \Xi^{-1} \curl $ maps skew symmetric matrices to $0$. Indeed,
\begin{equation}\label{613}
 \curl \Xi^{-1} \curl \text{mskw}=-\curl \Xi^{-1}\, \Xi \grad=-\curl \grad=0.
\end{equation}
Also, note that $\curl \Xi^{-1} \curl$ maps all matrix fields to
symmetric matrix fields because
\begin{equation}\label{713}
 \vskw \curl \Xi^{-1} \curl=\frac{1}{2}\dive \Xi \, \Xi^{-1} \curl =\frac{1}{2} \dive \curl=0. 
\end{equation}
Our elasticity complexes will be formed using
the following spaces:
\begin{align*}
&U_{r+1}^0(\Ta)  = Z_{r+1}^0(\Ta) \!\otimes\! {\V}, 
 && \mU_{r+1}^0(\Ta)  = \mZ_{r+1}^0(\Ta) \!\otimes\! {\V},
\\
&U_{r}^1(\Ta) = 
\{\sym(u) :  u\in Z_{r}^{1}(\Ta)\!\otimes\! {\V}\},
&&
\mU_{r}^1(\Ta)  = 
\{\sym(u) :  u\in \mZ_{r}^{1}(\Ta)\!\otimes\! {\V}\},
\\
&U_{r-2}^2(\Ta)= \{ \omega \in V_{r-2}^{2}(\Ta)\!\otimes\! {\V}:
\text{skw } \omega=0 \}, 
&&
\mU_{r-2}^2(\Ta)= \{ \omega \in \mV_{r-2}^{2}(\Ta)\!\otimes\! {\V}:
\text{skw } \omega=0 \}, 
\\
&U_{r-3}^3(\Ta) =V_{r-3}^3(\Ta) \!\otimes\! {\V},
&&
   \mU_{r-3}^3(\Ta)
     =  {\{ u \in V_{r-3}^3(\Ta) \!\otimes\! {\V}: u \perp \RM \}}.
\end{align*}
Recall that $ \inc u = { \curl (\curl u)'}$ for a symmetric matrix
field $u.$ In fact, when $u$ is symmetric, 
$(\curl u)'=\Xi^{-1} \curl u$ as
\begin{equation}
  \label{eq:trcurl0forsym}
  \text{tr}(\curl u) =0,
\end{equation}
so $\inc u=\curl \Xi^{-1} \curl u$.  The elasticity complexes with the
newly defined $U$ spaces are as follows.
\begin{equation}
  \label{strong-symmetry-loc}
  \begin{tikzcd}[column sep=small]
    0\arrow{r}{}
    &
    \RM\arrow{r}{\subset}
    &[0.3em]
    U_{r+1}^0(\Ta)\arrow{r}{{\varepsilon}}
    &[0.3em]
    U_{r}^1(\Ta) \arrow{r}{\inc} 
    &[1em]
    U_{r-2}^2(\Ta) \arrow{r}{\dive}
    &[1em] U_{r-3}^3(\Ta)\arrow{r}{} &0,
  \end{tikzcd}
\end{equation}
Here is the analogue with
boundary conditions:
\begin{equation}
  \label{strong-symmetry-loc-0}
  \begin{tikzcd}[column sep=small]
    0\arrow{r}{}
    &
    \mU_{r+1}^0(\Ta)\arrow{r}{{\varepsilon}}
    &[0.5em]
    \mU_{r}^1(\Ta) \arrow{r}{\inc} 
    &[1.2em]
    \mU_{r-2}^2(\Ta) \arrow{r}{\dive}
    &[1.2em] \mU_{r-3}^3(\Ta)\arrow{r}{} &0.
  \end{tikzcd}
\end{equation}

\begin{theorem}\label{thm:strong-symmetry-loc}
  The sequences \eqref{strong-symmetry-loc} and
  \eqref{strong-symmetry-loc-0} are exact sequences for $r \ge 4$.
\end{theorem}
\begin{proof}
  First we must show that~\eqref{strong-symmetry-loc} is a
  complex. By~\eqref{613},
  $\curl\Xi^{-1}\curl\circ {\varepsilon}=\curl\Xi^{-1}\curl\circ
  {\grad} = 0$. Also, it's obvious that
  $\div \circ \curl\Xi^{-1}\curl =0$. Hence it suffices to verify that the
  operators map into the right spaces.  Let $w \in
  U_{r+1}^0(\Ta)$. Then we have
  $\grad w \in Z_{r}^{1}(\Ta)\otimes {\V}$ and therefore
  ${\varepsilon} (w) \in U_{r}^1(\Ta)$.  Next, consider a
  $u \in U_{r}^1(\Ta)$. Then $\omega=\curl \Xi^{-1} \curl u $ has zero
  skew symmetric part due to \eqref{713}, so is in $ U_{r-2}^2(\Ta).$
  Finally, if
  $v \in U_{r-2}^2(\Ta) \subseteq V_{r-2}^2(\Ta) \otimes{\V}$, then
  $\dive v \in V_{r-3}^3(\Ta) \otimes {\V} = U_{r-3}^3(\Ta).$

  Now we prove exactness.  Let $u \in U_{r-3}^3(\Ta)$. We use 
  the surjectivity of the last map in the exact sequence of 
  Theorem~\ref{thm:elasticitysequence}. Accordingly, for
  $(0, u) \in [Z_{r-2}^3(\Ta) \otimes {\V}] \times [V_{r-3}^3(\Ta)
  \otimes {\V}]$,  there is a $w \in V_{r-2}^{2}(\Ta)\otimes
  {\V}$ such that $\dive w=u$ and $2 \text{vskw
  }w=0$. Thus, $w \in U_{r-2}(\Ta)$ and $\div w = u$, establishing the
  surjectivity of $\div$ in~\eqref{strong-symmetry-loc}. 

  Next, let $u \in U_{r-2}^2(\Ta)$ with $\dive u=0$.  Then, $u$ is in
  the kernel of the last operator in the exact sequence of
  Theorem~\ref{thm:elasticitysequence}. Hence, there is a
  $ v \in Z_{r}^1(\Ta) \otimes {\V}$ such that
  $\curl \Xi^{-1} \curl v=u$. But, by \eqref{613},
  $\curl \Xi^{-1} \curl (\sym v)= \curl \Xi^{-1} \curl v= u$.  Thus we
  have found a function $w =\sym v $ in $ U_{r}^1(\Ta)$ satisfying
  $\curl \Xi^{-1} \curl w=u$.

  Finally, let $u \in U_{r}^1(\Ta)$ with $\curl \Xi^{-1} \curl
  u=0$. Then $u = \sym(z)$ for some $z \in Z_r^1(\Ta) \otimes {\V}$
  and $\curl \Xi^{-1} \curl z = 0$ by \eqref{613}.  By
  Theorem~\ref{thm:elasticitysequence}, $z= \grad w-\mskw s$ for some
  $w \in Z_{r+1}^0(\Ta) \otimes {\V}=U_{r+1}^0(\Ta)$ and
  $s \in V_{r}^0(\Ta) \otimes {\V}$. Then
  $ u=\sym z= {\varepsilon}(w)- \sym (\mskw s)={\varepsilon}(w).$

  The proof of exactness of \eqref{strong-symmetry-loc-0} proceeds
  similarly using Theorem~\ref{thm:elasticitysequence-0} in place of
  Theorem~\ref{thm:elasticitysequence}, except where it concerns
  the surjectivity of the last map: to prove that $\div$
  in~\eqref{strong-symmetry-loc-0} is onto, consider a
  $u \in \mU_{r-3}^3(\Ta).$ By the exactness of the sequence
  \eqref{mlocalV} we have a $v \in \mV_{r-2}^2(\Ta) \times {\V}$ such
  that $\dive v=u$. For any constant vector $c \in \mathbb{R}^3,$
  since $u \perp \RM$, 
  \[
    \int_T\text{ vskw } v \cdot c=\int_T v : \text{ mskw } c=\int_T v :
    \grad (c \times x)=
    -\int_T \!\dive v \cdot( c \times  x)
    = -\int_T \!u \cdot( c \times  x)=0.
  \]
  Therefore, $2 \text{ vskw } v \in \mZ_{r-2}^3(\Ta)$. By
  the exactness of the sequence \eqref{mlocalZ}, there exists an 
  $m \in \mZ_{r-1}^2(\Ta)$ such that $ \dive m=2 \text{ vskw }
  v$. Hence, by \eqref{alg1} we get
  $2 \text{ vskw } \curl( \Xi^{-1} m)=2 \text{ vskw } v$ and so
  $\dive w=u$ where $w=v- \curl(\Xi^{-1} m) \in \mU_{r-2}^2(\Ta)$.
\end{proof}


\begin{lemma}
  \label{lem:Udims}
  When  $r \ge 4$, 
  \begin{align}    
    \dim U^0_{r+1}(\Ta)    & = 2r^3 + 16r + 12,
    &
       \dim \mU^0_{r+1}(\Ta) &= 2(r-3)(r-2)(r-1),
    \\    
    \dim U_r^1 (\Ta) & =  4r^3 - 3r^2 + 17 r - 6,
    &
      \dim \mU_r^1(\Ta)  & = 4r^3 - 21 r^2 + 29r - 6,
    \\
    \dim U_{r-2}^2 (\Ta) & = 4 r^3 - 9r^2 + 5r - 12,
    & \dim \mU_{r-2}^2(\Ta) & = r(r-1)(4r-11),
    \\
     \dim U_{r-3}^3 (\Ta)  & = 2 r (r-1)(r-2),
    & \dim \mU_{r-3}^3(\Ta) & = 2r^3 - 6r^2 + 4r - {6}.
  \end{align}
\end{lemma}
\begin{proof}
  By Lemma~\ref{lem:skw-onto},
  \begin{align*}
    \dim U_{r-2}^{2}(\Ta)
    &=\dim V^{2}_{r-2}(\Ta)\otimes {\V}-\dim Z_{r-2}^3(\Ta)\otimes {\V}
      =4r^{3}-9r^{2}+5r-12, \\
    \dim \mU^{2}_{r-2}(\Ta)
    &
      =\dim \mV^{2}_{r-2}(\Ta)\otimes {\V}-\dim \hat Z_{r-2}^3(\Ta)\otimes {\V}=r(r-1)(4r-11).
  \end{align*}
  The dimensions of the spaces with form degrees $0$ and $3$ easily
  follow from the previously established dimensions. Finally,
  $\dim U_r^1(\Ta)$ and $\dim \mU_r^1(\Ta)$ are computed
  using the exactness results of Theorem~\ref{thm:strong-symmetry-loc}.
\end{proof}

\subsection{An $H(\inc)$-conforming finite element}

The next result gives more insight into the structure of $U_r^1(\Ta)$,
and in particular, shows that $U_r^1(\Ta)$ is a conforming subspace of
$H(\inc, T) = \{ s \in H^{1}(T, \SSS): \; \inc s \in L^2(T, \SSS)\}$ on
the Alfeld split.  After proving it, we shall develop dofs that are
designed to help enforce global conformity in $H(\inc)$.  Let
$\mathcal{P}_{r}(\Ta; \mathbb{S})$ denote the space of symmetric
matrices whose entries are polynomials of degree $r$ on each
tetrahedron of the Alfeld split $\Ta$, and let
$H^{1}(\Omega; \mathbb{S})$ denote the space of symmetric matrix
fields with each entry in $H^{1}(\Omega)$.

\begin{theorem}
  \label{thm:charU}
  We have the following characterizations
  of $U_r^1(\Ta)$ and $\mU_r^1(\Ta)$.
  \begin{align*}
    U_r^1(\Ta)
    =\{ u\in H^{1}(T; \mathbb{S}): \;
    & 
      u\in \pol_{r}(\Ta; \mathbb{S}), \;
      (\curl u)' \in W_{r-1}^1(\Ta) \otimes {\V}, \\
    &u \text{ is $C^{1}$ at vertices of } T \text{ and }
      \inc u \text{ is $C^{0}$ at vertices of } T \}.
    \\
    \mU_r^1(\Ta)
    =\{ u\in \mathring{H}^{1}(T; \mathbb{S}): \;
    & 
      u\in \pol_{r}(\Ta; \mathbb{S}), \;
      (\curl u)' \in \mW_{r-1}^1(\Ta) \otimes {\V},
    \text{ all first order }\\
    &\text{derivatives of } u \text{ and } \inc u
      \text{ vanishes  at the vertices of } T \}.
  \end{align*}
\end{theorem}
\begin{proof}
  Let $M_r^1(\Ta)$ denote the space on the
  right hand side of the first equality.
  We claim that 
  \begin{equation}
    \label{eq:9}
    U_r^1(\Ta)  \subseteq M_r^1(\Ta).
  \end{equation}
  Indeed, if $u \in U_r^1(\Ta)$, then 
   $u = \sym z$ for some
  $z \in Z_r^1(\Ta) \otimes {\V}$, so by  \eqref{alg2},
  \begin{align*}
    \Xi^{-1} \curl z
    & = \Xi^{-1} \curl u + \Xi^{-1} \curl \skw z
      = \Xi^{-1} \curl u  - \grad \vskw z.    
  \end{align*}  
  Since the last term is in $W_{r-1}^1(\Ta) \otimes {\V}$ and
  $\Xi^{-1} \curl z$ is in $L_{r-1}^1(\Ta) \otimes {\V}$, we conclude
  that $ \Xi^{-1} \curl u$ is in $W_{r-1}^1(\Ta) \otimes {\V}$.  By
  Proposition~\ref{prop:extraC}(\ref{item:2_prop:extraC}), $z$ is
  $C^1$ at the vertices, hence so is $u=\sym z$. Moreover, since
  $\curl z$ is $C^1$ at the vertices, by~\eqref{613},
  $\curl \Xi^{-1} \curl u = \curl \Xi^{-1} \curl z$ is $C^0$ at the
  vertices. This proves~\eqref{eq:9}.

  To prove the reverse inclusion, let ${m} \in M_r^1(\Ta).$
  Put $\sigma = \curl \Xi^{-1} \curl {m}$. By the
  definition of $M_r^1(\Ta)$, we know that $\sigma$ is $C^0$ at the
  vertices, and moreover,
  $ \Xi^{-1} \curl {m} \in W_{r-1}^1(\Ta) \otimes {\V}$ so that $\sigma$
  is in $W_{r-2}^2(\Ta) \otimes {\V}$. Hence $\sigma$ is in
  $V_{r-2}^2(\Ta) \otimes {\V}.$ In fact, $\sigma$ is in the kernel
  of the last operator in the exact sequence of
  Theorem~\ref{thm:elasticitysequence} since $\dive \sigma=0$ and
  $\vskw(\sigma)=0$ due to \eqref{713}. Hence there is a
  $z \in Z_r^1(\Ta) \otimes {\V}$ such that
  $\sigma = \curl \Xi^{-1} \curl z$.

  Now consider $q = \Xi^{-1} \curl (m - z).$ Clearly, $\curl q=0$. In
  addition, the definitions of $M_r^1(\Ta)$ and $Z_r^1(\Ta)$ imply
  that $q $ is in $W_{r-1}^1(\Ta) \otimes {\V}$. Hence the exactness
  of the $W$-sequence yields a $v$ in $W_r^0(\Ta) \otimes {\V} $
  such that $\grad v = q$. By
  Proposition~\ref{prop:extraC}(\ref{item:2_prop:extraC}), $z$ is
  $C^1$ at the vertices, so $q$ is $C^0$ at the vertices, which in
  turn implies that $v$ is $C^1$ at the vertices of $T$.

  To finish the proof, put $\theta = m + \mskw(v)$. Then $\theta$ is
  $C^1$ at the vertices of $T$, and by \eqref{alg2},
  \begin{align*}
    \curl \theta
    & = \curl m + \curl \mskw v =
      \curl m {-} \Xi \grad v = \curl z.
  \end{align*}
  Hence $\theta$ is in $Z_r^1(\Ta) \otimes {\V}$. Since  
  $m = \sym(\theta)$, we conclude that $m \in U_r^1(\Ta)$.

  The proof of the characterization of $\mU^1_r(\Ta)$ is similar. 
\end{proof}

For further study of the complex, we collect some identities in the
next lemma, several of which involve surface operators we now
discuss. Let $F \in \Delta_2(T)$ and let $n$ be its unit normal vector
pointing out of $T$. Fix two tangent vectors $t_1, t_2$ in $n^\perp$,
such that the ordered set $(b_1, b_2, b_3) = (t_1, t_2, n)$ is a
orthonormal right-handed basis for $\mathbb{R}^3$. Any matrix field
$u: T \to \mathbb{R}^{3 \times 3}$ can be written as
$ \sum_{i, j=1}^3 u_{ij}b_i b_j'$ with scalar components
$u_{ij}: T \to \R$.  Let $ u_{nn} = n' u n$ and
$ \trF u = \sum_{i=1}^2 t_i' u t_i.$ With $s = t_1, t_2, n,$ or any
linear combination thereof, let
\begin{equation}
  \label{eq:11}
  u_{\FF} =\sum_{i, j = 1}^2 u_{ij}t_i t_j',
  \qquad
  u_{\F s}= \sum_{i=1}^2 ( s'ut_i) t_i',
  \qquad
  u_{s\F}= \sum_{i=1}^2 (t_i' u s) t_i,  
\end{equation}
Equivalently,
$u_{\FF} = Qu Q$, $u_{\F s} =s' uQ,$ and $u_{s \F} = Q u s$, where
$P = nn'$ and $Q = I-P$.
Next, considering scalar-valued (component) functions
$\phi, w_i, q_i$ and $u_{ij}$, we rewrite the standard surface operators
we have used before (in the proof of Theorem~\ref{thm:PiZ}) on the
left, while defining additional operations needed on the right using
the left definitions:
\begin{align*}
  & \gradF \phi  = (\partial_{t_1} \phi)t_1 +(\partial_{t_2}\phi)  t_2,
  &
    \gradF (w_1t_1 + w_2 t_2)  = t_1 (\gradF w_1)' + t_2(\gradF w_2)',&
  \\
  &\rotF \phi  =(\partial_{t_2}\phi) t_1 - (\partial_{t_1} \phi) t_2,
  &
    \rotF (q_1t_1' + q_2 t_2') = t_1 (\rotF q_1)' + t_2(\rotF q_2)',&
  \\
  & \curlF (w_1 t_1+w_2 t_2)  = \partial_{t_1} w_2 -\partial_{t_2} w_1, 
  &   \curlF u_\FF
    = t_1' \,\curlF (u_{\F t_1})' + t_2'\,\curlF (u_{\F t_2})'.&
\end{align*}
For vector functions $v$, let 
$v_\F = Q v = n \times (v \times n)$. It is easy to see that
\begin{equation}
  \label{eq:10}
  n \cdot \curl v = \curlF v_\F, \qquad
  (\grad v)_\FF = \gradF v_\F. \qquad n \times \rotF \phi = \gradF \phi.
\end{equation}

\begin{lemma}\label{lemma:id}
  For a (smooth enough) matrix-valued function $u$,
\begin{subequations} \label{id}
  \begin{alignat}{1}
      \label{curlid}
      s' \,(\curl u) \,n\,  & = \curlF (u_{\F s})', \text{ for any } s
      \in \R^3,
      \\
      \label{id4}
      \big[(\curl u)'\big]_{\F n}& = \curlF  u_{\FF}.  
      \intertext{If in addition $u$ is symmetric, then}
      \label{id1}
      (\inc u )_{nn}
      &=   \curlF (\curlF  u_{\FF})', 
      \\
            \label{id2}
      (\inc u)_{\F n} &= \curlF\big[(\curl u)'\big]_{\FF},
      \\
      \label{id3}
      \trF \curl u
      & =-\curlF (u_{\F n})'.  
     \intertext{For a (smooth enough) vector-valued function $v$,}
      2 (\curl \varepsilon (v))'
      & = \grad \curl v \label{more1},
      \\
      2\left[(\curl \varepsilon (v))'\right]_{\FF}
      & =  \gradF (\curl v)_\F \label{more2},
      \\
      \label{more3}
      \curl v
      & = n (\curlF v_\F) + \rotF (v \cdot n) + n \times \partial_n v,
      \\
      \label{more4}
      2[\varepsilon(v)]_{n\F} & = 2 [\varepsilon(v)_{\F n}]' 
      =
      \gradF (v\cdot n)
      + \partial_n v_\F,
      \\
      \label{more5}
      \trF (\rotF v_\F') & = \curlF v_\F.
    \end{alignat}
\end{subequations}
\end{lemma}
\begin{proof}
  The first identity \eqref{curlid} follows from~\eqref{eq:10}.
  The second follows from the first:
  \begin{align*}
    \big[(\curl u)'\big]_{\F n}
    & =  \sum_{i=1}^2 \left[n'(\curl u)' t_i\right]\, t_i'
    &&
       \text{by~\eqref{eq:11}}
    \\
    &
      = \sum_{i=1}^2 t_i' \,\curlF (u_{\F t_i})' 
      && \text{by~\eqref{curlid}},
  \end{align*}
  which
  equals $\curlF u_\FF$ per our definition. 
  To prove~\eqref{id1},
  \begin{align*}
    (\inc u)_{nn} & = n' (\curl (\curl u)')  n
                   = \curlF \left(\left[(\curl u)'\right]_{\F n}\right)'
    && \text{by~\eqref{curlid}}
    \\
                 & = \curlF (\curlF u_\FF)'
    && \text{by~\eqref{id4}}.
  \end{align*}
  To prove~\eqref{id2},
  we use~\eqref{id4}
  to obtain  
  $[(\inc u)']_{\F n} 
  = \curlF [(\curl u)']_\FF$
  and use the symmetry of $\inc u$. 
  To prove~\eqref{id3}, we use \eqref{eq:trcurl0forsym}, noting that
  trace is invariant under a basis change, to obtain
  $0 = \tr(\curl u) = \trF(\curl u) + n' (\curl u) n$.
  The last term, by~\eqref{curlid}, 
  equals  $\curlF (u_{\F  n})',$ thus proving~\eqref{id3}.
  To prove~\eqref{more1}, 
  \begin{alignat*}{1}
    2\curl \varepsilon(v)
    &=-2\curl  (\skw \grad v)= -2\curl (\text{mskw} \text{ vskw } \grad v)\\
    &=\Xi \grad (2\text{vskw } \grad v)= \Xi \grad (\curl v)
  \end{alignat*}
  which equals the right hand side of~\eqref{more1} since
  $\tr (\grad \curl v)=0$.  Equation~\eqref{more2} follows from
  \eqref{more1} and~\eqref{eq:10}.  Proving the remaining
  identities involves mere definition chasing.
\end{proof}

The identities in the next lemma are obtained by integration by parts
on a face.  Stokes theorem gives, for $q=q_1 t_1' + q_2 t_2',$
\begin{equation}
  \label{eq:15}
  \int_F \curlF u_\FF \cdot q \, \dA
  = \int_F u_\FF : \rotF q \, \dA
  + \int_{\partial F} u_\FF t \cdot q' \, \ds.
\end{equation}
Here and in the sequel, $t$ denotes a unit tangent vector along
$\partial F$ (not to be confused with $t_i$) oriented with respect to
$n$ to satisfy the right hand rule. We use ``$\cdot$'' for inner
products between row-vectors (as in the first term of~\eqref{eq:15})
as well between column vectors (as in the last term of~\eqref{eq:15}),
and we use ``$:$'' to denote the Frobenius inner product between
matrices.  Using $\mathbb{I}_F = t_1t_1' + t_2 t_2'$, we define
$\devF u_\FF = u_\FF - \mathbb{I}_F(\trF u_\FF)/2 $, which is used
below.

\begin{lemma}
  \label{lem:incintgF}
  Let $u$ be a symmetric matrix-valued function and let $q_i$ and
  $\phi$ be scalar-valued functions.  For smooth enough $u,$
  $q=q_1 t_1' + q_2 t_2',$ and $ \phi$, we have
  \begin{align}
    \label{eq:13}
    \int_F (\inc u)_{nn} \phi
    & = \int_F u_\FF : \rotF(\rotF \phi)'\; \dA
    \\ \nonumber 
    &
      + \int_{\partial F} (\curlF u_\FF) t\, \phi\, ds
      + \int_{\partial F} u_\FF t \cdot (\rotF \phi)' \,\ds, 
    \\ \label{eq:14}
    \int_F (\inc u)_{\Fn} \cdot q \, \dA
    &=
      \int_F \left[ (\curl u)' \right]_\FF : \devF \rotF q \, \dA
      -\frac 1 2 \int_F u_{\nF} \cdot \rotF \curlF q'\; \dA
    \\ \nonumber 
    & +
      \int_{\partial F} \left[ (\curl u)' \right]_\FF t \cdot q' \, \ds
      - \frac 1 2 \int_{\partial F} (u_{n \F}\cdot t) \,\curlF q' \, \ds.
  \end{align}
\end{lemma}
\begin{proof}
  By Stokes theorem and  Lemma~\ref{lemma:id}'s \eqref{id1},
  \[
    \int_F (\inc u)_{nn} \phi\, \dA
    = \int_{\partial F} (\curlF u_\FF)' \cdot t \phi\, \ds
    + \int_F (\curlF u_\FF)' \cdot \rotF \phi\, \dA.
  \]
  In the last term, writing $(\curlF u_\FF)' \cdot \rotF \phi$ as
  $\curlF u_\FF \cdot (\rotF \phi)'$ and integrating by parts again
  using~\eqref{eq:15}, we  prove~\eqref{eq:13}. To prove~\eqref{eq:14}, we
  start by using \eqref{id2} and~\eqref{eq:15}:
  \begin{equation}
    \label{eq:12}
    \int_F (\inc u)_{\Fn} \cdot q
    =
    \int_{\partial F}  \left[ (\curl u )'\right]_\FF t \cdot q' \,
    \ds
    + \int_F \left[ (\curl u )'\right]_\FF : \rotF q\, \dA.
  \end{equation}
  Note that
  $ [ (\curl u )']_\FF : \rotF q = [ (\curl u )']_\FF:\devF \rotF q +
  \frac 1 2 (\trF[ (\curl u )']_\FF) (\trF \rotF q).$ Also, by
  \eqref{id3},
  $\trF[ (\curl u )']_\FF = \trF (\curl u) = -\curlF (u_{\Fn})'$, and
  by \eqref{more5}, $\trF \rotF q = \curlF q'.$ Hence the last term
  of~\eqref{eq:12}, after a further integration by parts, becomes
  \[
    \int_F \left[ (\curl u )'\right]_\FF : \rotF q\, \dA
    =   -\frac 1 2 \int_F u_{\Fn}' \cdot \rotF\curlF q' \, \ds - \frac 1 2
    \int_{\partial F} u_{\Fn}'\cdot t \,\curlF q'\, \ds.  
  \] 
  Thus~\eqref{eq:14} follows after using $u_{\Fn}' = u_\nF$.
\end{proof}

The next lemma is an exactness result in 2D, which we
state on the face $F$
for our purposes. Let
$\veps_\F (v)= \sym \gradF v$ for $ v \in t_1 H^1(F) + t_2
H^1(F)$. Spaces like the latter will be abbreviated to \textcolor{black}{$[H^1(F)]^{2}$}
(e.g., $[\pol_{r-5}(F)]^2$ in the next lemma). Let $b_\F$
denote the face bubble, i.e., the product of the three barycentric
coordinates of the vertices of $F$.

\begin{lemma}\label{curlF}
  Let $u_\FF$ be as in~\eqref{eq:11} with $u_{ij}$ in $\pol_r(F)$ and
  $u_\FF' = u_\FF$. If $\curlF (\curlF u_\FF)'=0$ and both
  $u_\FF|_{\partial F}=0$ and
  $(t \cdot \curlF u_\FF)|_{\partial F}=0$, then
  $u_\FF=\varepsilon_\F(b_\F^2 \phi)$ where
  $\phi \in [\pol_{r-5}(F)]^2$.
\end{lemma}
\begin{proof}  
  Since $\curlF (\curlF u_\FF)' =0$ and the tangential component of
  $(\curlF u_\FF)'$ vanishes on $\partial F$, we have
  $(\curlF u_\FF)'=\gradF (b_\F \psi)$ for some
  $\psi \in \pol_{r-2}(F)$. Put $g=b_F \psi (t_1 t_2'- t_2
  t_1')$. Observe that $\gradF(b_F \psi)=(\curlF g)'$ and $\sym(g) = 0$.
  Thus, $\curlF(u_\FF-g)=0$ and $u_\FF-g$ vanishes on $\partial F$. Hence,
  there exists $\phi \in [\pol_{r-5}(F)]^2$ such that
  $\gradF (b_\F^2 \phi)= u_\FF-g$.  We conclude by noting that
  $u_\FF= \sym u_\FF= \varepsilon_\F(b_\F^2 \phi)$.
\end{proof}

With these preparations, we proceed to develop degrees of freedom for
$U_r^1(\Ta).$ Instead of directly using the definition of $U_r^1(\Ta)$
as the symmetric part of another space, we use its alternate
characterization in Theorem~\ref{thm:charU} to design its dofs.  Let
$t_e$ denote a unit tangent vector (of arbitrarily fixed orientation)
along an edge $e$. We will need the space of rigid displacements
within a face $F$, namely
$\RM(F):=\{d_1 t_1'+ d_2 t_2'+c ((x \cdot t_1) t_2' - (x \cdot t_2)
t_1'): c, d_i \in \mathbb{R} \}$.

\begin{lemma}
  \label{lem:U1dofs}
  For any $r \ge 4$, the functionals
\begin{subequations}
  \label{eqn:revU1dofs}
  \begin{alignat}{4}
    \label{eqn:revU1dofs-V1}
    & D^{\alpha} u(a),\quad &&
      |\alpha| \le 1,  a\in \Delta_0(T),
      && \dofcnt{\rm{($96$)}}
    \\
    \label{eqn:revU1dofs-V2}
    &  \inc u(a),\quad &&
      a\in \Delta_0(T),
      && \dofcnt{\rm{($24$)}}
    \\      
    \label{eqn:revU1dofs-E1}
    &\int_e u: \kappa,
    \quad && \kappa \in   \sym [\pol_{r-4}(e)]^{3 \times 3}, e \in \Delta_1(T),
      &&    \dofcnt{\rm{($6\cdot 6(r-3)$)}}
    \\
    \label{eqn:revU1dofs-E2}
    &\int_e (\curl u)'t_e \cdot \kappa, \quad &&
        \kappa \in [{\pol}_{r-3}(e)]^3, e\in \Delta_1(T) \quad
      && \dofcnt{\rm{($6\cdot 3 \cdot (r-2)$)}}
    \\
    \label{eqn:revU1dofs-F2}
    &\int_e (\inc u) n^\F \cdot  \kappa, \quad &&
\kappa \in  [{\pol}_{r-4}(e)]^3, e \in \Delta_1(F),  F\in \Delta_2(T) \quad
      && \dofcnt{\rm{($ 4 \cdot 3 \cdot 3 (r-3)$)}}
    \\
    \label{eqn:revU1dofs-F3-a}
    &\int_F (\inc u)_{nn}\,\kappa, \quad &&  \kappa \in  {\pol}_{r-5}(F)/\pol_1(F),  F\in \Delta_2(T)
      && \dofcnt{$\left(4\cdot\left[\frac{1}{2}(r-4)(r-3) 
            - 3\right]\right)$}
    \\
    \label{eqn:revU1dofs-F3-b}
    &\int_F (\inc u )_{\F n} \cdot \kappa, \quad 
    &&  \kappa \in [{\pol}_{r-5}(F)]^2 /\RM(F),
       F\in \Delta_2(T),
      && \dofcnt{$\left(4\cdot 
          \left[2\frac 1 2 (r-4)(r-3) - 3 \right]\right)$}
    \\  
    \label{eqn:revU1dofs-F4}
    &\int_F  u_{\FF} : \kappa , \quad &&  \kappa
    \in \varepsilon_\F\!\left(b_\F^2\,
      [{\pol}_{r-5}(F)]^2\right), 
      F\in \Delta_2(T)
      && \dofcnt{\rm{($4 \cdot 2 \cdot \frac{1}{2}(r-4)(r-3)$)}}
    \\
    \label{eqn:revU1dofs-F5}
    &\int_F [(\curl u)']_{\FF} \!:\! \kappa, \quad && \kappa \in \gradF(b_\F
    [{\pol}_{r-3}(F)]^2),   F\in \Delta_2(T)
      && \dofcnt{\rm{($4 \cdot 2 \cdot \frac{1}{2}(r-2)(r-1)$)}}
      \\
    \label{eqn:revU1dofs-F6}
    &\int_F u_{\F n}\cdot \kappa , \quad && \kappa \in  \gradF(b_\F^2 
      {\pol}_{r-5}(F)),
      \quad F\in \Delta_2(T),
      && \dofcnt{\rm{($4 \cdot \frac{1}{2}(r-4)(r-3)$)}}
    \\
    \label{eqn:revU1dofs-F1}
    &\int_F u_{nn}\, \kappa, \quad && \kappa \in {\pol}_{r-3}(F), \qquad
      F\in \Delta_2(T)\qquad
      && \dofcnt{\rm{($4 \cdot \frac{1}{2}(r-2)(r-1)$)}}
     \\
    \label{eqn:revU1dofs-T1}
    &\int_T \inc u : \inc  \kappa, \quad && \kappa \in \mU^{1}_{r}(\Ta), \quad 
      && \dofcnt{\rm{($2r^3-9r^2+7r+6$)}}\\
    \label{eqn:revU1dofs-T2}
    &\int_T u:  \varepsilon(\kappa), \quad && \kappa \in \mU^{0}_{r+1}(\Ta),
      &&  \dofcnt{\rm{($ 2(r-3)(r-2)(r-1)$)}}
  \end{alignat}
\end{subequations}
form a unisolvent set of degrees of freedom for $U_{r}^1(\Ta)$.
\end{lemma}

\begin{proof}
  The number of dofs add up to the dimension of
  $U_{r}^1(\Ta)$ given in Lemma~\ref{lem:Udims}.
  Suppose that all dofs of \eqref{eqn:revU1dofs}
  vanish for a $u \in U_{r}^1(\Ta)$. We must show that $u \equiv 0$.
  The following conclusions are immediate from
  (\ref{eqn:revU1dofs-V1}, \ref{eqn:revU1dofs-E1}), 
  (\ref{eqn:revU1dofs-V2}, \ref{eqn:revU1dofs-F2}), and  
  (\ref{eqn:revU1dofs-V1}, \ref{eqn:revU1dofs-E2}), respectively:
  \begin{equation}
    \label{eq:16}
    u|_e =0, \qquad (\inc u)n^\F|_e=0, \qquad (\curl u)' t|_e = 0,
    \qquad  \text{ for } e \in \Delta_1(T), \; F \in \Delta_2(T).
  \end{equation}
  In particular, the last equality, in conjunction with \eqref{id4} of
  Lemma~\ref{lemma:id} shows that
  $0 = n' (\curl u)' t= (\curlF u_{\FF})' t $ on $\partial F$. Hence
  all terms on the right hand side of~\eqref{eq:13} vanish when
  $\phi \in \pol_1(F)$. Thus Lemma~\ref{lem:incintgF},  combined with
  \eqref{eqn:revU1dofs-F3-a} and~\eqref{eq:16}, yields
  $(\inc u)_{nn} = 0$ on all $F \in \Delta_2(T)$.  Next, before
  proceeding to use \eqref{eqn:revU1dofs-F3-b}, observe that any
  $q = c ((x \cdot t_1) t_2' - (x \cdot t_2) t_1')$, with $c \in \R$,
  has $\rotF q = c(t_1 t_1'+t_2 t_2')$. Hence
  $\devF \rotF \RM(F) = 0$. Similarly, 
  \textcolor{black}{$\rotF \curlF [\RM(F)]' = 0$.}
  Hence all terms on the right hand side of~\eqref{eq:14} vanish for
  $q \in \RM(F)$.
  Therefore,  Lemma~\ref{lem:incintgF}  combined with
  \eqref{eqn:revU1dofs-F3-b}, \eqref{eq:16}, and the observation that
  $(\inc u)_\nF = [(\inc u)_\Fn]'$ leads us to conclude that $(\inc
  u)_\nF =0$ on all faces of $T$. Thus, $(\inc u) n |_{\partial T} = 0$.

  In particular, due to~\eqref{id1}, $\curlF (\curlF u_{\FF})' = 0$ on
  any $F\in \Delta_2(T)$. This implies, by virtue of~\eqref{eq:16} and
  Lemma~\ref{curlF}, that $u_{\FF}=\varepsilon_\F(b_\F^2 \phi)$ for
  some $\phi \in [\pol_{r-5}(F)]^2$, and hence by
  \eqref{eqn:revU1dofs-F4}, we have $u_{\FF}$ vanishes on $F$. This
  implies, by~\eqref{id4}, that $[(\curl u)']_{\F n}$ vanishes on $F$,
  i.e., $n' (\curl u)'Q|_{\partial T} = 0.$ In fact, all of
  $(\curl u)'Q$ vanishes on $\partial T$, as we shall now show.

  On a face $F$, by~\eqref{id2}, $\curlF [(\curl u)']_{\FF} = 0$, and
  by~\eqref{eq:16}, $[(\curl u)']_{\FF} t|_{\partial F} = 0.$ Hence,
  $(\curl u)'_{\FF}=\gradF(b_\F \phi)$ for a
  $\phi \in [\pol_{r-3}(F)]^2$. The dofs of~\eqref{eqn:revU1dofs-F5}
  then show that $ [(\curl u)']_{\FF}$ vanishes on $F$. We have
  already shown that $[(\curl u)']_{\F n}$ also vanishes. Hence
  $(\curl u)'Q|_{\partial T}=0$. Since we know from the first
  characterization of Theorem~\ref{thm:charU} that
  $(\curl u)' \in W_{r-1}^1(\Ta) \otimes {\V}$, we have just shown
  that $(\curl u)' \in \mW_{r-1}^1(\Ta) \otimes {\V}$.

  Next, we proceed to show that $u|_{\partial T} = 0$.  On a face $F$,
  since $[(\curl u)']_{\FF} =0,$ by~\eqref{id3},
  $ \trF [(\curl u)']_{\FF} = \trF (\curl u)' = \curlF u_{\F n}=0.$
  Moreover, by~\eqref{eq:16}, $u_{\F n}|_{\partial F} =0$, so
  $u_{\F n}=\gradF(b_\F^2 \phi)$ \textcolor{black}{for $\phi \in \pol_{r-5}(F)$}. Thus,
  by \eqref{eqn:revU1dofs-F6} we conclude that $u_{\F n}$ vanishes on
  $F$. Of course, we also have $u_{nn}|_{\partial T} = 0$ due to
  \eqref{eqn:revU1dofs-F1}. Thus $u|_{\partial T} = 0$.

  We are now in a position to apply the second characterization of
  Theorem~\ref{thm:charU}, to conclude that $u \in
  \mU_r^1(\Ta)$. Hence \eqref{eqn:revU1dofs-T1} implies
  $\inc u \equiv 0$ on $T$.  Using the exactness of
  the sequence   \eqref{strong-symmetry-loc} given by 
  Theorem~\ref{thm:strong-symmetry-loc} and 
  the dofs of~\eqref{eqn:revU1dofs-T2}, we
  see that $u \equiv 0$ on $T$.
\end{proof}

\subsection{Degrees of freedom of the remaining elements}

In this subsection, we give unisolvent dofs for the other spaces in
the complex, $U^0_{r+1}(\Ta),$ $U_{r-2}^2(\Ta)$ and $U^3_{r-3}(\Ta)$.
We begin by proposing the following dofs for $U_{r+1}^0(\Ta)$ and
proving them to be unisolvent.

\begin{lemma}
 \label{lem:U0dofs}
  For any $r \ge 4$, the functionals
\begin{subequations}
\label{eqn:Sigmadofs-new}
\begin{alignat}{4}
\label{eqn:U0_DOF1}
&D^\alpha \omega (a),\quad &&  |\alpha|\le 2,
\; a\in \Delta_0(T),\quad
&&
\dofcnt{$(120)$}\\
\label{eqn:U0_DOF21}
&\int_e \omega \cdot \kappa {\ds},\quad
&& \kappa\in [\pol_{r-5}(e)]^3,  \; e\in\Delta_1(T),\quad
&& \dofcnt{$(18(r-4))$} \\
\label{eqn:U0_DOF22}
&\int_e \frac{\p \omega}{\p n_e^\pm}\cdot\kappa\,{\ds},\quad 
&& \kappa\in [\pol_{r-4}(e)]^3,  e\in\Delta_1(T),\quad &&
\dofcnt{$(36(r-3))$}\\
\label{eqn:U0_DOF31}
&\int_F [\varepsilon(\omega)]_{\F n} \cdot   \kappa ,
\quad 
&& \kappa \in \gradF\!(b_\F^2 \pol_{r-5}(F)), F \in\Delta_2(T),\quad
 && \dofcnt{$(2(r-4)(r-3))$}\\
\label{eqn:U0_DOF33}
&\int_F \varepsilon_\F(\omega_\F) : \varepsilon_\F(b_\F^2 \kappa) \, 
\quad 
&& \kappa\in [\pol_{r-5}(F)]^2, F\in\Delta_2(T),\quad
&&\dofcnt{$(4(r-4)(r-3))$}\\
\label{eqn:U0_DOF32a}
&\int_F \frac{\partial(\omega \cdot n)}{\partial n}\,  \kappa ,
\quad 
&& \kappa\in \pol_{r-3}(F), F \in\Delta_2(T),\quad
&&  \dofcnt{$(2(r-2)(r-1))$}\\
\label{eqn:U0_DOF32b}
&\int_F \left[(\curl \varepsilon (\omega))'\right]_{\FF} \!: \kappa,
\quad 
&& \kappa\in \gradF \!(b_\F [{\pol}_{r-3}(F)]^2),  F\in\Delta_2(T),\quad
&& \dofcnt{$(4(r-2)(r-1))$} \\
\label{eqn:U0_DOF4}
&
\int_T {\varepsilon( \omega) :\varepsilon(\kappa)}\,{\dx},
\quad 
&& \kappa\in \mU_{r+1}^{0}(\Ta),
&&\dofcnt{$(2 (r-3)(r-2)(r-1))$}
\end{alignat}
\end{subequations}
form a unisolvent set of degrees of freedom for $U_{r+1}^0(\Ta)$.
\end{lemma}
\begin{proof}
  It is easily verified that the total number of dofs equals the
  dimension of $U_{r+1}^0(\Ta)$ given in Lemma~\ref{lem:Udims}.
  Consider an $\omega \in U_{r+1}^0(\Ta)$ on which the dofs
  \eqref{eqn:Sigmadofs-new} vanish.  By standard arguments, dofs
  \eqref{eqn:U0_DOF1}, \eqref{eqn:U0_DOF21} and \eqref{eqn:U0_DOF22}
  imply that
  \begin{equation}
    \label{eq:17}
    \omega|_e =0, \qquad (\grad \omega)|_e = 0, \qquad \text{ for } e
    \in \Delta_1(T).
  \end{equation}
  Hence, on any face $F \in \Delta_2(T)$, we have
  $\omega_\F \in b_\F^2 [P_{r-5}(F)]^2$, so~\eqref{eqn:U0_DOF33}
  implies that $\omega_\F = 0.$ Also note that~\eqref{eq:17} implies
  $(\curl \omega)_\F \in b_\F [P_{r-3}(F)]^2$, so the dofs of
  \eqref{eqn:U0_DOF32b}, in view of the identity \eqref{more2},
  imply that $(\curl \omega)_\F =0$. This in turn implies, after
  taking the cross product with $n$ on both sides of \eqref{more3} and
  using~\eqref{eq:10}, that
  \begin{equation}
    \label{eq:18}
    \partial_n \omega_\F = \gradF (\omega \cdot n).  
  \end{equation}
  Hence \eqref{more4} yields
  $ 2[\veps(\omega)]_\nF = \gradF (\omega\cdot n) + \partial_n
  \omega_\F = 2 \gradF (\omega\cdot n)$. The latter is in
  $ b_\F^2 [P_{r-5}(F)]^2$ due to~\eqref{eq:17}, so the dofs of
  \eqref{eqn:U0_DOF31} give $\omega\cdot n = 0$ on $F$.  Combining
  with the already shown $\omega_\F \equiv 0$, we summarize: all
  components of $\omega$ vanish on $\partial T$.

  Next, we will show that all first derivatives of $\omega$ also
  vanish on $\partial T$. Consider an $F \in \Delta_2(T)$ and let $K$
  denote one of the subtetrahedra $T_i$ which has $F$ as a face. Then,
  since $\omega\cdot n$ vanishes on $F$, there must exist a
  $p \in \pol_r(K)$ such that $\omega \cdot n = \mu p$. Since
  $\partial_n (\omega \cdot n)|_F = (\partial_n \mu) p|_F$ vanishes on
  $\partial F$ by~\eqref{eq:17}, there exists a
  $\psi \in \pol_{r-3}(F)$ such that $p = b_\F \psi$, i.e.,
  $\partial_n (\omega \cdot n)|_F = (\partial_n \mu) b_\F
  \psi$. Hence~\eqref{eqn:U0_DOF32a} yields
  $\partial_n (\omega\cdot n) =0$. By~\eqref{eq:18},
  $\partial_n \omega_\F$ also vanishes. Since
  $\omega|_{\partial T} \equiv 0$, all the tangential derivatives of
  $\omega$ also vanish, so we conclude that
  $(\grad \omega)|_{\partial T} \equiv 0$.  Thus
  $\omega \in \mU_{r+1}^{0}(\Ta)$. Therefore, \eqref{eqn:U0_DOF4}
  shows that $\omega$ vanishes.
\end{proof}

\begin{lemma}[Dofs of the stress space]
  \label{lem:U2dofs}
  For any $r \ge 4$, the functionals
  \begin{subequations}
    \label{eqn:U2dofs}
    \begin{alignat}{4}
      \label{eqn:U2dofs-a}
      & \sigma(a),\quad
      && a\in \Delta_0(T),\quad &&\dofcnt{\rm{($6\times 4$ dofs)}}
      \\
      \label{eqn:U2dofs-b}
      &\int_e \sigma n^\F\cdot \kappa\, {\ds},\ && \kappa \in
      [\pol_{r-4}(e)]^{3}, \;  e\in \Delta_1(F), \;  F\in
      \Delta_2(T),\quad &&
      \dofcnt{\rm{($3\times 12(r-3)$ dofs)}}
      \\
      \label{eqn:U2dofs-c}
      &\int_F \sigma n^\F\cdot \kappa\, {\dA}, \quad
      && \kappa\in
      [{\pol}_{r-5}(F)]^{3}, \; f\in \Delta_2(T),\qquad
      &&\dofcnt{\rm{($3\times 2(r-3)(r-4)$ dofs)}}
      \\
      \label{eqn:U2dofs-d2}
      &\int_T \sigma : \kappa\, {\dx}, \quad
      && \kappa\in \inc \mU^{1}_{r-1}(\Ta),
      &&
      \dofcnt{\rm{($2r^3 - 9 r^2 +7r +{6}$ dofs)}}
      \\
      \label{eqn:U2dofs-d1}
      &\int_T \dive\sigma\cdot \kappa\, {\dx}, \quad
      && \kappa\in  \mU^{3}_{r-3}(\Ta),
      && \dofcnt{\rm{($2r^3 - 6r^2 + 4r -{6}$ dofs)}}
    \end{alignat}
  \end{subequations}
  form a unisolvent set of degrees of freedom for $U_{r-2}^2(\Ta)$.
\end{lemma}
\begin{proof}
  The stated counts of the dofs (obtained using the exactness given in
  Theorem~\ref{thm:strong-symmetry-loc} and standard dimensions) sum
  up to the expression for $\dim U_{r-2}^2(\Ta)$ in
  Lemma~\ref{lem:Udims}.  To prove unisolvency, let all dofs
  of~\eqref{eqn:U2dofs} vanish for a
  $\sigma \in U_{r-2}^2(\Ta) \subset V_{r-2}^2(\Ta) \otimes \V$. Then,
  \eqref{eqn:U2dofs-a}--\eqref{eqn:U2dofs-c}, together with
  $\skw \sigma (x_i)=0$ yield $\sigma n |_{\partial T} = 0$ (as in the
  proof of Lemma~\ref{lem:Z2unisolvency}). Thus, by
  Lemma~\ref{lem:V_Vo}, $\sigma$ is in $\mV_{r-2}^2 \otimes \V$, and
  since $\skw \sigma \equiv 0$, we conclude that
  $\sigma \in \mU_{r-2}^2(\Ta)$.  By
  Theorem~\ref{thm:strong-symmetry-loc},
  $\mU^{3}_{r-3}(\Ta) = \dive \mU^{2}_{r-2}(\Ta)$, so the dofs of
  \eqref{eqn:U2dofs-d1} yield $\dive \sigma = 0$. Using the exactness
  result of Theorem~\ref{thm:strong-symmetry-loc} again, we conclude
  that there is a $u \in \mU_{r-1}^1(\Ta)$ such that
  $\sigma = \inc u$. Hence, using~\eqref{eqn:U2dofs-d2}, we conclude
  that $\sigma = 0$.
\end{proof}

Finally, the dofs of $U_{r-3}^3(\Ta)$ are just ``three
copies'' of the dofs of $V_{r-1}^3(\Ta)$. It is immediate that any
$w\in U_{r-3}^3(\Ta)$ is uniquely defined by the following
functionals:
\begin{subequations}
  \label{eq:U3dofs}  
  \begin{alignat}{4}
  \label{eq:U3dofs-a}  
  & \int_T {w \cdot \kappa} \, {\dx},
  & \quad \kappa \in \RM, & \qquad \qquad\dofcnt{(${6}$ dofs)},
  \\
  \label{eq:U3dofs-b}  
  &\int_T w \cdot v\, {\dx},
  & \quad\qquad  v \in \mU_{r-3}^3(\Ta),
  & \qquad \qquad \dofcnt{($ 2r^3 - 6r^2 + 4r - {6}$ dofs)}.
\end{alignat}
\end{subequations}

\subsection{Commuting projections}

In this subsection, we study the accompanying cochain projectors of
our elasticity complex \eqref{strong-symmetry-loc}.  These projections
are simply the canonical finite element interpolants, denoted by
$\Pi_j^U,$ $j=0, 1, 2, 3$, defined by the already given degrees of
freedom of $U^0_{r+1}(\Ta), U^1_r(\Ta), U_{r-2}^2(\Ta),$ and
$U_{r-3}^3(\Ta)$, respectively.  Note that $\Pi_3^U$ is just the $L^2$
projection into $U_{r-3}^3(\Ta)$ since
$\mU_{r-3}^3(\Ta) \oplus {\RM} = U_{r-3}^3(\Ta)$.

\begin{theorem}
  The following diagram commutes for the indicated degrees $r$:
\begin{equation}
  \begin{tikzcd}[column sep=small, row sep=large]
    \RM \arrow{r}{\subset}
    & C^{\infty}(\bar T)\otimes \V
    \ar{d}{\Pi_0^U}
    \arrow{r}{\veps}
    \arrow[dr, phantom, "{\scriptstyle{(r\ge 4)}}"]
    & [0.5em] C^{\infty}(\bar T)\otimes \mathbb{S}
    \arrow{r}{\inc}
    \ar{d}{\Pi_1^U}
    \arrow[dr, phantom, "{\scriptstyle{(r\ge 4)}}"]
    &[1em]
    C^{\infty}(\bar T) \otimes    \mathbb{S}
    \arrow{r}{\div}
    \ar{d}{\Pi_2^U}
    \arrow[dr, phantom, "{\scriptstyle{(r\ge 6)}}"]
    &[1em]
    C^{\infty}(\bar T) \otimes {\V}
    \arrow{r}
    \ar{d}{\Pi_3^U}
    &0
    \\
    \RM\arrow{r}{\subset}
    &[0.3em]
    U_{r+1}^0(\Ta)\arrow{r}{{\varepsilon}}
    &[0.3em]
    U_{r}^1(\Ta) \arrow{r}{\inc} 
    &[1em]
    U_{r-2}^2(\Ta) \arrow{r}{\dive}
    &[1em] U_{r-3}^3(\Ta)\arrow{r}{} &0.
  \end{tikzcd}
\end{equation}
\end{theorem}
\begin{proof}
  First, we show that
  \begin{equation}
    \label{eq:19}
    \dive \Pi_2^U \sigma = \Pi_3^U \dive \sigma
    \qquad \text{ for } r \ge 6.    
  \end{equation}
  Since $w = \dive \Pi^U_2 \sigma - \Pi_3^U \dive \sigma$ is in
  $U_{r-3}^3(\Ta)$, it suffices to prove that the dofs of
  \eqref{eq:U3dofs} applied to $w$ vanish.  It is obvious
  from~\eqref{eqn:U2dofs-d1} that the dofs of~\eqref{eq:U3dofs-b}
  applied to $w$ vanish. The dofs of~\eqref{eq:U3dofs-a} also vanish
  because, for $\kappa \in \RM$
  \begin{align*}
    \int_T w   \cdot \kappa\, {\dx}
    & =
      \int_T \dive \Pi^U_2 \sigma \cdot \kappa   \,
      {\dx} - \int_T \dive \sigma  \cdot \kappa  \,      {\dx}
    && \text{ by~\eqref{eq:U3dofs-a}}
    \\
    & = \int_{\partial T}  (\Pi^U_2 \sigma) \, n  \cdot \kappa\,{\ds}
      - \int_{\partial T} \sigma\, n  \cdot \kappa \, {\ds} = 0
    && \text{ by~\eqref{eqn:U2dofs-c}}
  \end{align*}
  for $r \ge {6}$, so~\eqref{eq:19} follows.

  Next, we show that
  \begin{equation}
    \label{eq:20}
    \inc \Pi_1^U u = \Pi_2^U \inc u
    \qquad \text{ for } r \ge 4.    
  \end{equation}
  Let $\sigma= \inc \Pi_1^U u - \Pi_2^U \inc u$. To show that the dofs
  \eqref{eqn:U2dofs} vanish on $\sigma$, we begin by noting that
  \eqref{eqn:revU1dofs-V2} and \eqref{eqn:U2dofs-a} imply that the
  dofs \eqref{eqn:U2dofs-a} vanish for $\sigma.$ Similarly, the dofs
  \eqref{eqn:U2dofs-b} applied to $\sigma$ also vanish due to
  \eqref{eqn:revU1dofs-F2} and \eqref{eqn:U2dofs-b}.
  To show that the dofs of \eqref{eqn:U2dofs-c} also vanish on
  $\sigma$, we split them into normal and tangential parts after using
  \eqref{eqn:U2dofs-c} on $\inc u$:
  \begin{equation}
    \label{eq:22}
    \int_F
    \sigma n^\F \cdot  \kappa
    =g_n + g_\F, \quad
    g_n = \int_F \left[\inc( \Pi_1^U u - u) \right]_{nn} \kappa_n,
    \quad
    g_\F = 
    \int_F  \left[\inc( \Pi_1^U u - u) \right]_{\nF} \kappa_\F.     
  \end{equation}
  Note that $g_n$ vanishes for any
  $\kappa_n \in \pol_{r-5}(F) / \pol_1(F)$ due to
  \eqref{eqn:revU1dofs-F3-a}. In fact, $g_n$ vanishes for any
  $\kappa_n \in \pol_{r-5}(F)$, as we now show. Observe that by
  \eqref{eq:13} of Lemma~\ref{lem:incintgF}, for a
  $p_1 \in \pol_1(F)$,
  \begin{align}
    \label{eq:21}
    \int_F \sigma_{nn} \,p_1
    =  \int_{\partial F} \left[\curlF ( \Pi_1^U u - u)_\FF\right]
    t\, p_1\, ds
      + \int_{\partial F} ( \Pi_1^U u - u)_\FF t \cdot (\rotF p_1)' \,\ds.
  \end{align}
  By~\eqref{id4},
  $\curlF ( \Pi_1^U u - u)_\FF t\, p_1 = [\curl ( \Pi_1^U u -
  u)']_{\Fn} t \,p_1 = \curl ( \Pi_1^U u - u)': p_1 n t',$ so the
  first term on the right hand side of~\eqref{eq:21} vanishes by
  \eqref{eqn:revU1dofs-E2}.  The last term of~\eqref{eq:21} also
  vanishes because
  $( \Pi_1^U u - u)_\FF t \cdot (\rotF p_1)' = Q( \Pi_1^U u - u) Q t
  \cdot (\rotF p_1)' = ( \Pi_1^U u - u) Q t \cdot Q(\rotF p_1)' = (
  \Pi_1^U u - u) : Q(\rotF p_1)' t $ thus allowing us to apply
  \eqref{eqn:revU1dofs-E1} whenever $r -4 \ge 0$. Thus
  from~\eqref{eq:21} we conclude that $g_n=0$ for all
  $\kappa_n \in \pol_{r-5}(F) = \pol_1(F) \oplus \pol_{r-5}(F) /
  \pol_1(F)$.  Next consider $g_\F$. Obviously,
  \eqref{eqn:revU1dofs-F3-b} shows that $g_\F=0$ whenever
  $\kappa_\F \in [\pol_{r-5}(F)]^2/\RM(F)$. However, for
  $\kappa_\F \in \RM(F)$, we may conduct a similar argument as above
  but now using \eqref{eq:14} of Lemma~\ref{lem:incintgF}, to conclude
  that $g_\F =0$ for all $\kappa_\F \in [\pol_{r-5}(F)]^2$.  Thus,
  returning to~\eqref{eq:22}, we have
  $\int_F \sigma n^\F \cdot \kappa=0$ for all
  $\kappa \in [\pol_{r-5}(F)]^3$, i.e., all dofs of
  \eqref{eqn:U2dofs-c} applied to $\sigma$ vanish.  It is easy to see
  that the remaining dofs of \eqref{eqn:U2dofs-d2} and
  \eqref{eqn:U2dofs-d1} applied to $\sigma$ also vanish, thus
  finishing the proof of \eqref{eq:20}.

  Finally, we will prove that
  \begin{equation}
    \label{eq:23}
    \varepsilon(\Pi_0^U \omega) = \Pi_1^U \varepsilon  (\omega),
    \qquad \text{ for } r \ge 4.    
  \end{equation}
  Letting
  $u= \varepsilon(\Pi_0^U \omega) -\Pi_1^U \varepsilon (\omega)$, it
  is enough to show that the dofs of \eqref{eqn:revU1dofs} applied to
  $u$ vanish. Let us dispose off the obvious implications first: {\it
    (i)} $\inc \circ \,\varepsilon =0$ implies that the dofs of
  \eqref{eqn:revU1dofs-V2}, \eqref{eqn:revU1dofs-F2},
  \eqref{eqn:revU1dofs-F3-a}, \eqref{eqn:revU1dofs-F3-b} and
  \eqref{eqn:revU1dofs-T1} applied to $u$ vanish; {\it(ii)}
  \eqref{eqn:U0_DOF1}, \eqref{eqn:U0_DOF31}, \eqref{eqn:U0_DOF33},
  \eqref{eqn:U0_DOF32a}, \eqref{eqn:U0_DOF32b} applied to $\omega,$
  taken together with 
  \eqref{eqn:revU1dofs-V1}, \eqref{eqn:revU1dofs-F6},
  \eqref{eqn:revU1dofs-F4}, \eqref{eqn:revU1dofs-F1},
  \eqref{eqn:revU1dofs-F5} applied to $\veps(\omega)$, each 
    respectively imply that \eqref{eqn:revU1dofs-V1},
  \eqref{eqn:revU1dofs-F6}, \eqref{eqn:revU1dofs-F4},
  \eqref{eqn:revU1dofs-F1}, \eqref{eqn:revU1dofs-F5} applied to $u$
  vanish.  It only remains to prove that the dofs of
  \eqref{eqn:revU1dofs-E1} and \eqref{eqn:revU1dofs-E2} applied to $u$
  vanish.
  To this end, it is useful to employ the edge-based orthonormal basis
  $\{ n_e^+, n_e^-, t_e\}$ and write
  $\kappa \in \sym [\pol_{r-4}(e)]^{3 \times 3}$ as
  $ \kappa = \kappa_{11} n_e^{+} (n_e^{+})' + \kappa_{12} \big(n_e^+
  (n_e^-)' + n_e^- (n_e^+)'\big) + \kappa_{13} \big(n_e^+ t_e' + t_e
  (n_e^{+})'\big) + + \kappa_{22} n_e^-(n_e^-)' +
  \kappa_{23}\big(n_e^+(n_e^-)' + n_e^- (n_e^+)'\big) + \kappa_{33}
  t_e t_e'$ where $\kappa_{ij} \in \pol_{r-4}(e).  $ Then,
  \begin{alignat*}{2}
    \int_e u: \kappa
    & = \int_e [\varepsilon(\Pi^U_0 \omega)
    -\Pi^U_1\varepsilon(\omega)] : \kappa 
    = \int_e \varepsilon(\Pi^U_0 \omega -\omega) : \kappa
    & \text{by~\eqref{eqn:revU1dofs-E1}}
    \\
    & = \int_e \grad(\Pi^U_0  \omega - \omega) :  \kappa 
    = \int_e \grad(\Pi^U_0  \omega - \omega) : (\kappa_{13} n_e^+ t_e'
    + \kappa_{33} t_e t_e')
    &\quad \text{by~\eqref{eqn:U0_DOF22}}
    \\
    & =  \int_e \grad(\Pi^U_0  \omega - \omega)  t_e  \cdot ( \kappa_{13} n_e^+  +\kappa_{33} t_e). 
  \end{alignat*}
  Now that the integrand contains a tangential derivative, we may
  integrate by parts, to see that the integral vanishes after an
  application of~\eqref{eqn:U0_DOF1} and \eqref{eqn:U0_DOF21}.  Thus
  the dofs of \eqref{eqn:revU1dofs-E1} applied to $u$ vanish.  To
  examine the dofs \eqref{eqn:revU1dofs-E2}, letting
  $\kappa \in [{\pol}_{r-3}(e)]^3$, we note that
   \begin{alignat*}{2}
     \int_e (\curl u)'t_e \cdot \kappa=
     & \int_e \big[\curl \veps(\Pi_0^U \omega-\omega)\big]'t_e \cdot \kappa \quad && \text{by~\eqref{eqn:revU1dofs-E2}} \\
    =& \frac{1}{2} \int_e  \big[\grad \curl (\Pi_0^U \omega-\omega)\big] t_e \cdot \kappa  \quad && \text{by~\eqref{more1}} \\
    =&   -\frac{1}{2} \int_e  \curl (\Pi_0^U \omega-\omega)
    \cdot \partial_t\kappa \quad && \text{by~\eqref{eqn:U0_DOF1}}
    \end{alignat*}
    where in the last step, we have integrated by parts, and put
    $\partial_t \kappa = (\grad \kappa) t_e$. The curl in the
    integrand above can be decomposed into terms involving
    $\partial_t (\Pi_0^U \omega-\omega)$ and those involving
    $\partial_{n_e^\pm}(\Pi_0^U \omega-\omega)$. The former terms can
    be integrated by parts yet again, which after
    using~\eqref{eqn:U0_DOF1} and \eqref{eqn:U0_DOF21}, vanish. The
    latter terms also vanish by~\eqref{eqn:U0_DOF22}
    which we may apply as $\partial_t \kappa$ is of degree at most $r-4$.
\end{proof}

\section{Global complexes}   \label{sec:global}

We have developed a number of new finite elements on Alfeld splits in
the previous sections.  In this section, we briefly discuss how the
elements on Alfeld splits may be put together to construct global
finite element spaces.  Throughout this section, $\Omega$ denotes a
contractible polyhedral domain in $\R^3$, subdivided by $ \Th$, a
conforming tetrahedral mesh (and $h$ denotes maximal element
diameter).  Let $\Tha$ be the refinement obtained performing an Alfeld
split to each mesh tetrahedron $T\in \Th$, e.g., by connecting the
barycenter of $T$ with its vertices.  We consider finite element
spaces on $\Tha$ built using the previously discussed elements. Every
local dof we defined previously was associated to a subsimplex, so it
is standard to go from the local dofs to the global dofs associated to
the simplicial complex $\Th$. {\color{black}We use $\#_{k}$ to denote the number of $k$-dimensional simplexes in $\Th$. For example,  $\#_{0}$ and  $\#_{1}$ are the numbers of vertices and edges, respectively.}

\subsection{The global $V$ complex}

We begin with the standard finite element sequence.
Let $ W_{r}^0(\mcta)$, $W_{r}^1(\mcta)$, $W_{r}^2(\mcta)$, and
$W_{r}^3(\mcta)$, denote the standard conforming finite element
subspaces of $H^1(\Omega), H(\curl, \Omega), H(\div, \Omega), $ and
$L^2(\Omega)$ whose elements when restricted to a mesh element
$T \in \mct$ are in $ W_{r}^0(\Ta)$, $W_{r}^1(\Ta)$, $W_{r}^2(\Ta)$,
and $W_{r}^3(\Ta)$, respectively. Let 
\begin{alignat*}{1}
  V_{r}^0(\mcta)
  &=\{ \omega \in C^1(\Omega):
  \omega \text{ is } C^2 \text{ at vertices of } \mct,\;
  \omega|_T \in V_{r}^0(\Ta) \text{ for all } T \in \mct\},
  \\
  V_{r}^1(\mcta)
  &  =\{ \omega \in [C^0(\Omega)]^3:
  \text{ $\omega$ is $C^1$ at
    vertices of $\mct$}, \;
  \omega|_T \in
  V_{r}^1(\Ta) \text{ for all } T \in \mct \},
  \\
  V_{r}^2(\mcta)
  & = \{ \omega \in H(\div, \Omega):
  \text{ $\omega$ is $C^0$ at vertices of $\mct$},\;
  \omega|_T \in V_{r}^2(\Ta)
  \text{ for all } T \in \mct \},
  \\
  V_{r}^3(\mcta)
  & =   W_{r}^3(\mcta).
\end{alignat*}
After inheriting the global $V$  dofs from the prior local $V$ dofs, we may
define global interpolation operators into these finite element spaces
in the canonical way. Then the following global analogue of
Theorem~\ref{thm:PiV-local} can be easily proved.

\begin{theorem}
  Let $\Pih i V$ denote the canonical global finite element
  interpolant onto $V_{r-i}^i(\mcta)$. Then for $r\ge 5$ the
  following diagram commutes:
  \begin{equation*}
    \begin{tikzcd}
      {C^{\infty}(\Omega)}\arrow{r}{\grad}\arrow{d}{\Pih 0 V}
      &{[C^{\infty}(\Omega)]}^{3}\arrow{d}{\Pih 1 V} \arrow{r}{\curl}
      &{[C^{\infty}(\Omega)]}^{3}\arrow{r}{\div}\arrow{d}{\Pih 2 V}
      &{C^{\infty}(\Omega)}\arrow{d}{\Pih 3 V}
      \\
      V_r^0(\mcta)\arrow{r}{\grad} &V_{r-1}^1(\mcta)\arrow{r}{\curl}
      &V_{r-2}^2(\mcta)\arrow{r}{\div} &V_{r-3}^3(\mcta).
    \end{tikzcd} 
  \end{equation*}
\end{theorem}

An exactness result analogous to Lemma~\ref{lem:exactVo} also holds
for these global $V$ spaces.  In order to prove it, we are not able to
use the projections $\Pih i V$ directly, since the functions we will
apply them to are not sufficiently regular. This technical problem is
overcome in the proof below by zeroing out the degrees of freedom
requiring higher regularity and
using the well-known existence of a
regular potential
(see e.g, \cite{costabel2010bogovskiui}):
\begin{align}
  \label{eq:H1_rgtinvdiv}
  & \forall  u \in L^2(\Omega), \; \exists
    v\in [H^{1}(\Omega)]^{3}\text{ such that }\div v=u.
\end{align}
\revjj{With the above-mentioned modified interpolant
  and~\eqref{eq:H1_rgtinvdiv}, the global exactness
  for $r\ge 5$ follows easily
  as seen below. The $r=4$ case is also interesting, but since no
  local dofs for gluing $V_{r}^{0}(T^{A})$ are known for this case,
  the same proof does not work.  Yet, we are able to prove the partial
  exactness result that $\div: V_{2}^2(\mcta)\to V_{1}^3(\mcta)$ is
  onto when $r=4$, using a technique inspired by Stenberg~\cite[Theorem
  1]{stenberg2010nonstandard}, who showed how dofs in standard mixed
  methods (for the Poisson problem) can be reduced by imposing vertex
  continuity. Our $V_{r-2}^2(\mcta)$ space has similar continuity
  restrictions at the vertices of $\Th$.}


\begin{theorem}\label{exact-global-V}
  The sequence
  \begin{equation}\label{globalV}
    \begin{tikzcd}[cramped]
      0\arrow{r}{}
      &\mathbb{R}\arrow{r}{\subset}
      & V_r^0(\mcta)\arrow{r}{\grad} &V_{r-1}^1(\mcta)\arrow{r}{\curl} &V_{r-2}^2(\mcta)\arrow{r}{\div} &V_{r-3}^3(\mcta)\arrow{r}{} &0,
    \end{tikzcd} 
  \end{equation}
  is exact for $r \ge 5$. \revjj{When $r=4$,
    the divergence operator remains surjective.}
\end{theorem}
\begin{proof}
  To show that $\div: V_{r-2}^2(\mcta) \rightarrow V_{r-3}^3(\mcta)$
  is onto, let $v \in V_{r-3}^3(\mcta)$. By~\eqref{eq:H1_rgtinvdiv},
  there exists an $\omega \in [H^1(\Omega)]^3$ such that
  $\dive \omega=v$.  \revjj{We now proceed to modify $\Pih 2 V$ and
    apply it to $\omega$. The first modification involves zeroing out
    vertex dofs to avoid taking values of $\omega$ at the vertices.
    The second modification involves a rearrangement of $r-3$ face
    dofs that helps prove the theorem's assertion for the $r=4$
    case. These modifications result in the $\Pith 2 V$ given next. For
    every $F \in \Delta_2(T)$, arbitrarily choose an edge of $F$ and
    denote it by~$e_{F}$.}  Then define
  $\Pith 2 V \omega \in V_{r-2}^2(\mcta)$ such that on an element
  $ T \in \Th$, the function $\omega_T = (\Pith 2 V \omega)|_T$ is
  given by the equations
  \begin{subequations}
  \begin{alignat}{4}
    \label{eqn:StenbergHDiv1-g}
    &\omega_T(a)=0,\quad && a\in \Delta_0(T)
    \\
    \label{eqn:StenbergHDiv2-g}
    &\int_e (\omega\cdot n^\F) \kappa\, {\ds},\ && \kappa \in
\pol_{r-4}(e), \; {\color{black} e\in \Delta_1(F), \; e\neq e_{F}},\;  F\in \Delta_2(T)\\
\label{eqn:StenbergHDiv3-g}
&\int_F (\omega\cdot n^\F)\kappa\, {\dA}, \quad &&{\color{black} \kappa\in
{\pol}_{r-4}(F)}, \; F\in \Delta_2(T)\\
     \label{eqn:StenbergHDiv4-g}
    &\int_T \omega_T \cdot \kappa\, {\dx}
    = \int_T \omega \cdot \kappa\, {\dx},
    \quad &&  \kappa \in \curl \mV_{r-1}^1(\Ta)
    \\
    \label{eqn:StenbergHDiv5-g}
    &\int_T (\Div\omega_T) \,\kappa \, {\dx}
    = \int_T (\Div\omega) \, \kappa \, {\dx},
    \quad && \kappa \in \mV_{r-3}^3 (\Ta) 
  \end{alignat}
\end{subequations}
These equations, \revjj{being a minor modification of previously given 
unisolvent dofs~\eqref{eqn:StenbergHDiv}, can easily be shown to}
uniquely define a $\omega_T \in V_{r-2}^2(\Ta)$, so
$\Pith 3 V \omega$ is a well defined function in $V_{r-2}^2(\mcta)$
for any $\omega$ in $H^1(\om)^3$ (e.g., the integral on the right hand
side of \eqref{eqn:StenbergHDiv3-g} is bounded for any $\omega$ in
$H^1(\om)^3$ by a trace theorem).  Now, \revjj{when $r\ge 4$,} for any
constant $\kappa$, we have
$ \int_T \div( \omega_T - \omega) \kappa = \int_{\partial T} (
\omega_T - \omega) \cdot n \kappa = 0 $ by
\eqref{eqn:StenbergHDiv3-g}. Hence \eqref{eqn:StenbergHDiv5-g} yields
$\dive \Pith 2 V \omega= \Pih 3 V \dive \omega= \Pih 3 V v = v$.
\revjj{This
proves the stated surjectivity of divergence for $r= 4$ as well as for
$r\ge 5$.}

\revjj{Continuing, restricting to the $r\ge 5$ case,}
for any $u\in {\ker} (\curl, V_{r-1}^1(\mcta))$, there exists
$v\in W_{r}^{0}(\mcta)$ such that $u=\grad v$ by the exactness of the
standard finite element de Rham complex (the $W$ sequence).  Since $u$
is $C^{1}$ at the vertices of $\Th$, $v$ is $C^{2}$ at the vertices,
so $v\in V_{r}^0(\mcta)$.

Finally,
  to show that
  $\curl V_{r-1}^1(\mcta) = {\ker}(\div, V_{r-2}^2(\mcta)),$
  it suffices to prove that their dimensions are equal.
  To this end, we note from \eqref{eqn:Sigmadofs}, \eqref{eqn:VecHermitedofs}, \eqref{eqn:StenbergHDiv}, and \eqref{Wdofs} that the following dimension count holds:
\begin{align*}
\dim
  (V_{r}^{0}(\mcta))&=10\#_{0}+(3r-13)\#_{1}+(r^{2}-7r+13)\#_{2}+\frac
                      2 3 (r-4)(r-3)(r-2)\#_{3},\\
\dim (V_{r-1}^{1}(\mcta))&=12\#_{0}+3(r-4)\#_{1}+\frac 3 2 (r-2)(r-3)\#_{2}+(2r^{3}-9r^{2}+19r-27)\#_{3},\\
  \dim (V_{r-2}^{2}(\mcta))&=3\#_{0}+\frac 1 2 (r+2)(r-3) \#_{2}
                             +(2r^{3}-5r^{2}+3r-12)\#_{3},\\
\dim (V_{r-3}^{3}(\mcta))&=\frac 2 3 r(r-1)(r-2)\#_{3}.
\end{align*}
By the exactness properties we have already proven,
$\dim \curl V_{r-1}^1(\mcta) = \dim V_{r-1}^1(\mcta) - \dim V_r^0(\mcta) +
1$ and
$\dim {\ker}(\div, V_{r-2}^2(\mcta)) = \dim V_{r-2}^2(\mcta) - \dim
V_{r-3}^2(\mcta)$. These numbers are equal because the Euler formula,
together with the dimensions given above, yields
$ \dim (V_{r}^{0}(\mcta))-\dim (V_{r-1}^{1}(\mcta))+\dim
(V_{r-2}^{2}(\mcta))-\dim (V_{r-3}^{3}(\mcta))=1$.
\end{proof}

\subsection{The global $Z$ complex}

Let $Z_{r}^0(\mcta) =V_r^0(\mcta),$ and 
\begin{alignat*}{1}
Z_{r}^1(\mcta) &=\{ \omega\in [C^0(\Omega)]^3:
\curl \omega \in [C^0(\Omega)]^3, \;
 \text{$\omega$ and $\curl \omega$ are $C^1$ at vertices of $\mct$},
 \\
 &\hspace{3.1cm}\text{ and }
 \omega|_T \in Z_{r}^1(\Ta) \text{ for all } T \in \mct \},
 \\
 Z_{r}^2(\mcta) &= \{ \omega \in [C^0(\Omega)]^3:
 \text{ $\omega$ is $C^1$ at vertices of $\mct$},\;
 \omega|_T \in Z_{r}^2(\Ta) \text{ for all } T \in \mct\},
 \\
 Z_{r}^3(\mcta)& =\{ \omega \in L^2(\Omega): \;
  \text{ $\omega$ is $C^0$ at vertices of $\mct$}, \;
  \omega|_T \in Z_{r}^3(\Ta) \text{ for all } T \in \mct, 
\}.
\end{alignat*}
These spaces inherit global dofs from the previously given local $Z$
dofs. The following global analogue of Theorem~\ref{thm:PiZ} can
be easily proved.

\begin{theorem} \label{exact-global-Z}
  Let $\Pih i Z$ denote the canonical global finite element
  interpolant onto $Z_{r-i+1}^i(\mcta)$. Then for \revj{$r\ge 4$} the
  following diagram commutes:
  \begin{equation*}
    \begin{tikzcd}
     {C^{\infty}(\Omega)}\arrow{r}{\grad}\arrow{d}{\Pi^{Z}_{0}}
     &{[C^{\infty}(\Omega)]}^{3}\arrow{d}{\Pi^{Z}_{1}}
     \arrow{r}{\curl}
     &{[C^{\infty}(\Omega)]}^{3}\arrow{r}{\div}\arrow{d}{\Pi^{Z}_{2}}
     &{C^{\infty}(\Omega)}\arrow{d}{\Pi^{Z}_{3}}
     \\
     Z_{r+1}^0(\mcta)\arrow{r}{\grad}
     &Z_r^1(\mcta)\arrow{r}{\curl} &Z_{r-1}^2(\mcta)\arrow{r}{\div}
     &Z_{r-2}^3(\mcta).
    \end{tikzcd} 
  \end{equation*}
\end{theorem}

\begin{theorem}
  The sequence
  \begin{equation}\label{globalZ}
    \begin{tikzcd}[cramped]
      0\arrow{r}{} &\mathbb{R}\arrow{r}{\subset} & Z_{r+1}^0(\mcta)\arrow{r}{\grad} &Z_r^1(\mcta)\arrow{r}{\curl} &Z_{r-1}^2(\mcta)\arrow{r}{\div} &Z_{r-2}^3(\mcta)\arrow{r}{} &0
    \end{tikzcd} 
  \end{equation}
  is exact for \revj{$r \ge 4$}.
\end{theorem}

\begin{proof}
  Let $\mathcal{A}$ denote the set of all (constant) trace-free
  $3 \times 3$ matrices.  To show that $\dive$ is onto, let
  $v \in Z_{r-2}^3(\mcta)$. By~\eqref{eq:H1_rgtinvdiv}, there exists
  an $\omega \in [H^1(\Omega)]^3$ such that $\dive \omega=v$.  Let
  $\Pith 2 Z \omega \in Z_{r-1}^2(\mcta)$ be such that on each
  $T \in \mcta$, its element restriction
  $\omega_T = (\Pith 2 Z \omega)|_T$ satisfies
\begin{subequations}
\begin{alignat*}{4}
  &\omega_T(a)=0,
  && && a\in \Delta_0(T),
  \\
  &(\grad \omega_T)(a): \kappa =0,   &&  \kappa \in \mathcal{A},
  && a\in \Delta_0(T),\qquad && 
  \\
  &\tr(\grad \omega_T)(a)=\tr(\grad \omega)(a),
  && &&a\in \Delta_0(T),\qquad &&
  \\
&\int_e \omega_T\cdot\kappa\,{\ds} =0,\quad &&  \kappa\in [\pol_{r-5}(e)]^3, &&   e\in \Delta_1(T),\qquad &&
\hspace{-.5cm}\\
&\int_F \omega_T \cdot \kappa\,{\ds}= \int_F \omega \cdot \kappa\,{\ds}, \quad &&  \kappa\in [\pol_{r-4}(F)]^3, &&   F\in \Delta_2(T),\qquad &&
\hspace{-1cm}\\
&\int_T \omega_T \cdot \kappa\, {\dx}= \int_T  \omega \cdot \kappa\, {\dx}, \quad &&  \kappa\in \curl \mZ_r^1(\Ta),
&& &&\hspace{-2.25cm}\\
&\int_T \dive  \omega_T \cdot \kappa\, {\dx} =\int_T \dive \omega \cdot \kappa\, {\dx},  \qquad &&  \kappa\in \dive \mZ_{r-1}^2(\Ta).
&& && \hspace{-2.25cm} 
\end{alignat*}
\end{subequations}
\revjj{Here we have used
the same technique of zeroing out certain dofs that we used in 
the proof of Theorem~\ref{exact-global-V}}.
The right hand sides of the equations above are bounded since $\omega$ is
in $[H^1(\om)]^3$ and since
$\tr(\grad \omega) = \dive \omega = v \in Z^3_{r-2}(\Ta)$.  These
equations uniquely determine $\Pith 2 Z \omega \in Z_{r-1}^2(\mcta)$
due to the unisolvency of~\eqref{Z2dofs} proved in
Lemma~\ref{lem:Z2unisolvency}. An argument analogous to the one we
used to prove that $\dive \Pi_2^Z= \Pi_3^Z \dive$ in the proof of
Theorem~\ref{thm:PiZ} now yields
$\dive \Pith 2 Z \omega=\dive \omega=v$.
It is easy to prove that
$\grad: Z_{r+1}^0(\mcta) \rightarrow {\ker} (\curl, Z_{r}^1(\mcta))$
is onto (see proof of Theorem~\ref{exact-global-V}).

Finally, we perform a dimension count of the global degrees of freedom
to show that
$\curl: Z_{r}^1(\mcta)\rightarrow {\ker}(\div, Z_{r-1}^2(\mcta))$ is
onto. To this end, we note from \eqref{eqn:Sigmadofs}, \eqref{Z1dofs},
\eqref{Z2dofs}, and \eqref{Z3dofs} that the following dimension count
holds:
\begin{align*}
\dim
  (Z_{r+1}^{0}(\mcta))&=10\#_{0}+(3r-10)\#_{1}+(r^{2}-5r+7)\#_{2}+\frac
                      2 3 (r-3)(r-2)(r-1)\#_{3},\\
\dim (Z_{r}^{1}(\mcta))&=20\#_{0}+3(2r-7)\#_{1}+ \frac 5 2 (r-2)(r-3)\#_{2}\\
  \quad &+\left(\frac 2 3 (r-3)(r-2)(r-1)+
          \frac 1 3
          (r-3)(r-2)(4r-7)\right)\#_{3},\\
  \dim (Z_{r-1}^{2}(\mcta))&=12\#_{0}+3(r-4)\#_{1}+\frac 3 2
                           (r-2)(r-3)\#_{2}\\
  \quad &+\left( \frac 1 3 (r-3)(r-2)(4r-7)
          +\frac 2 3 (r+1)r(r-1)-13 \right)\#_{3},\\
\dim (Z_{r-2}^{3}(\mcta))&=\#_{0}+\left(\frac 2 3 (r+1)r(r-1)-12\right)\#_{3}.
\end{align*}
By the Euler formula, we have 
$$
\dim (Z_{r+1}^{0}(\mcta))-\dim (Z_{r}^{1}(\mcta))+\dim (Z_{r-1}^{2}(\mcta))-\dim (Z_{r-2}^{3}(\mcta))=1,
$$
\revjj{which shows that
$\dim\curl Z_{r}^1(\mcta) = \dim{\ker}(\div, Z_{r-1}^2(\mcta))$.}
\end{proof}

\subsection{The global $U$ complex}

The global elasticity complex consists of 
\begin{alignat*}{2}
U_{r+1}^0(\mcta) & =Z_{r+1}^0(\mcta), &
U_{r}^1(\mcta) &
=\left \{\sym (u) : u\in  Z_{r}^1(\mcta)\otimes {\V}\right \},\\
U_{r-2}^2(\mcta)&=  \{ \omega \in U_{r-2}^{2}(\mcta)\otimes {\V}: \skw
\omega=0 \}, \quad 
&U_{r-3}^3(\mcta)& =Z_{r-3}^3(\mcta).
\end{alignat*}
\revjj{To show that these spaces form an exact global complex, we
  follow the same procedure as for the local complex, starting with a}
global analogue of Theorem~\ref{thm:elasticitysequence}. Note that
like in the local case, the global space $V^1_{r-1}(\mcta) \otimes \V$
is in bijective correspondence with $Z_{r-1}^2(\mcta) \otimes \V$ via
$\Xi$. Also,
$\vskw: V^2_{r-2}(\mcta) \otimes \V \mapsto Z^3_{r-2}(\mcta \otimes
\V)$ is easily seen to be surjective.

\begin{theorem} \label{thm:elasticitysequence-g}
  \revjj{For $r \ge 5$,} the sequence \vspace{-0.3cm}  
  \[ 
    \begin{tikzcd}
      \!\!
      \begin{bmatrix}
        Z_{r+1}^0(\mcta) \!\otimes\! {\V} \\
        V_{r}^0(\mcta) \!\otimes\! {\V}
      \end{bmatrix}\!\!
      \arrow{r}{
        \left[\begin{smallmatrix}
            \!\grad\!, \; -\!\mskw \!
          \end{smallmatrix}\right]}
      &[3.2em]
      \!\! Z_{r}^1(\mcta)  \!\otimes\! {\V} \!\!
      \arrow{r}{\curl \Xi^{-1} \curl}
      &[2.6em] 
      \!\!V_{r-2}^2(\mcta) \!\otimes\! {\V}\!\!
      \arrow{r}{
        \begin{bmatrix}
          \!2\! \vskw\!\! \\ \dive 
        \end{bmatrix}
      }
      &[0.9em]
      \begin{bmatrix}
        Z_{r-2}^3(\mcta) \!\otimes\! {\V} \\
        V_{r-3}^3(\mcta) \!\otimes\! {\V}
      \end{bmatrix}
    \end{tikzcd} 
  \]
  is exact
  \revjj{and the kernel of the first operator above is isomorphic to
      $\RM$.
    When $r = 4$,  the last operator remains surjective.}
\end{theorem}
\begin{proof}
  \revjj{The case $r\ge 5$ follows by the $r\ge 5$ case of
    Theorem~\ref{exact-global-V} and
    Proposition~\ref{prop:bgg_basic}'s
    item~(\ref{item:1:prop:bgg_basic}). The statement for $r=4$ follows
    from the surjectivity of the divergence asserted by
    Theorem~\ref{exact-global-V} in the $r=4$ case and
    Proposition~\ref{prop:bgg_basic}'s
    item~(\ref{item:2:prop:bgg_basic}).}
\end{proof}

\begin{theorem}
  \label{thm:charU-g}
  $U_r^1(\mcta) =\{ u\in H^{1}(\Omega; \mathbb{S}): $
      $ (\curl u)' \in W_{r-1}^1(\mcta) \otimes {\V},$
      $u$ is $C^{1}$ at the mesh vertices of $\Th$, 
      $\inc u$ is $C^{0}$ at the mesh vertices of $\Th$, and
      $u|_T\in U_r^1(\Ta)$ for all mesh elements $T \in \Th \}.$
\end{theorem}
\begin{proof}
  This can be proved along the lines of the proof of
  Theorem~\ref{thm:charU-g}
  using Theorem~\ref{thm:elasticitysequence-g}.
\end{proof}

\begin{theorem}
  The following sequence of global finite element spaces 
  \begin{equation}\label{strong-symmetry}
    \begin{tikzcd}[column sep=small]
      0\arrow{r}{}
      &
      \RM \arrow{r}{\subset}
      &[0.4em]
      U_{r+1}^0(\mcta)\arrow{r}{{\varepsilon} }
      &[0.6em] U_{r}^1(\mcta) \arrow{r}{\inc}
      &[1em] U_{r-2}^2(\mcta) \arrow{r}{\dive}
      &[0.7em] U_{r-3}^3(\mcta)\arrow{r}{} 
      & 0 .
    \end{tikzcd} 
  \end{equation}
  is a complex and is exact (on
  contractible domains) {\color{black} for $r\geq 4$}.
\end{theorem}
\begin{proof}
For $r\geq 5$, the proof is along the lines of the proof of
  Theorem~\ref{thm:strong-symmetry-loc} using
  Theorem~\ref{thm:elasticitysequence-g}.
  
  \revjj{For $r=4$, first note that the surjectivity of
    $\div: U_{r-2}^2(\mcta) \to U_{r-3}^3(\mcta)$ follows from Theorem
    \ref{thm:elasticitysequence-g}.}  \revjj{Next, we show that
    $\veps( U_{r+1}^0(\mcta)) = \ker(\inc, U_r^1(\mcta))$ for $r=4$.
    Any $u \in U_r^1(\mcta)$ with $\inc u = 0$ may be written as
    $\veps(v)$ for some $v\in H^2(\om)$ by the exactness
    of~\eqref{sobolev-elasticity}. Now, on each mesh element
    $T \in \Th$, split into an Alfeld split $\Ta$, the local exactness
    result of Theorem~\ref{thm:strong-symmetry-loc}, applied with
    $r=4$, shows that there is a $w_T \in U_{r+1}^0(\Ta)$ satisfying
    $\veps(w_T) = u|_T$. In other words, $\veps( w_T - u)|_T = 0$,
    which implies that on each $T$, the function $v$ must equal a
    polynomial of the form $v|_T = w_T + r_T$ for some
    $r_T \in \RM(T) \subset [\pol_1(T)]^3$. Thus $u = \veps(v)$ and
    $v \in H^2(\Omega) \cap \pol_5(\mcta) \subseteq U_5^0(\mcta)$.  To
    complete the proof of exactness, we now only need to show that
    $\curl: U_{r}^1(\mcta)\rightarrow {\ker}(\div, U_{r-2}^2(\mcta))$
    is onto.}
  To this end, we note from \eqref{eqn:Sigmadofs},
    \eqref{eqn:revU1dofs}, \eqref{eqn:U2dofs}, and \eqref{eq:U3dofs}
    that the following dimension count holds:
\begin{align*}
\dim
  (U_{r+1}^{0}(\mcta))&=10\#_{0}+(3r-10)\#_{1}+(r^{2}-5r+7)\#_{2}+N_{0}\#_{3},\\
\dim (U_{r}^{1}(\mcta))&=30\#_{0}+3(3r-8)\#_{1}+ 3/2 (3 r^2 - 11 r + 4)\#_{2}+N_{1}\#_{3},\\
  \dim (U_{r-2}^{2}(\mcta))&=6\#_{0}+3/2 (r - 3) (r + 2)\#_{2}+N_{2}\#_{3},\\
\dim (U_{r-3}^{3}(\mcta))&=N_{3}\#_{3}.
\end{align*}
By the definition of interior degrees of freedom, we have $N_{0}+N_{2}=N_{1}+N_{3}-6$. 
By the Euler formula, we have 
$$
\dim (U_{r+1}^{0}(\mcta))-\dim (U_{r}^{1}(\mcta))+\dim (U_{r-2}^{2}(\mcta))-\dim (U_{r-3}^{3}(\mcta))=6,
$$
which \revjj{shows that $\dim \curl U_{r}^1(\mcta) = \dim {\ker}(\div, U_{r-2}^2(\mcta))$.}
\end{proof}

\appendix

\section{Supersmoothness}     \label{sec:supersmoothness}

Consider a tetrahedron $T$ and its Alfeld split $\{ T_i\}$ as in the
rest of the paper.  Proposition~\ref{prop:extraC}'s
items~(\ref{item:0_prop:extraC}) and (\ref{item:1_prop:extraC}) are a
consequence of the following fact proved in~\cite{Alfel84}: if
$v \in C^1(T)$ and $v|_{T_i}$ is in $C^\infty(T_i)$, then $v$ is $C^2$
at the vertices of $T$. Such serendipitous ``supersmoothness'' at some
points was observed on triangles earlier~\cite{Farin80}. In
Theorem~\ref{thm:supersmooth} below, we establish a
supersmoothness result in the same spirit for $1$-forms. In fact, the
earlier result of Alfeld follows from the theorem, as noted in
Corollary~\ref{cor:supersmoothness}.
Items~(\ref{item:2_prop:extraC}) and~(\ref{item:3_prop:extraC}) of
Proposition~\ref{prop:extraC} follow from the arguments below. (The
proof will show that the 
assumption that $v|_{T_i}$ is infinitely smooth can be
relaxed, but this generalization is not important for our purposes.)

\begin{theorem}
  \label{thm:supersmooth}
  Suppose $v$ is in $C^0(T)^3$, $v_i = v|_{T_i} \in C^\infty(T_i)$,
  and $\curl v$ is $C^0$ at the vertices $x_i$ of $T$. Then $v$ is
  $C^1$ at $x_i$.
\end{theorem}
\begin{proof}
  Let $F_{ij} = \partial T_i \cap \partial T_j$ and let $T F_{ij}$
  denote the tangent plane of $F_{ij}$.  Let $c \in {\V}$ and
  $\tau \in TF_{ij}$.  The first observation needed for this proof is
  that
  \begin{equation}
    \label{eq:3}
    c \cdot (\grad v_i) \tau = c \cdot (\grad v_j)  \tau
    \quad \text{ on }  F_{ij}.
  \end{equation}
  This is because the continuity of $v$ requires $(v _i - v_j)\cdot c$ to
  vanish on $F_{ij}$ for any $c \in {\V}$, so its tangential
  derivatives also vanish on $F_{ij}$. 

  We claim that at a vertex of $T$ on $F_{ij}$, we also have
  \begin{equation}
    \label{eq:4}
    \tau \cdot (\grad v_i) c = \tau \cdot (\grad v_j) c.
  \end{equation}
  To show this, consider $x_1$, a common vertex of $T$ and
  $F_{23}$. Then, since the scalar $\tau \cdot (\grad v_i)(x_1) \, c$
  equals its transpose, we have
  \begin{align*}
    \tau \cdot (\grad v_2)\, c
    & = c \cdot (\grad v_2)' \tau && \text{ at } x_1,
    \\
    & = c \cdot \left((\grad v_2)' - (\grad v_2) \right) \tau
      + c \cdot (\grad v_2)  \tau
    \\
    & = c \cdot \left((\grad v_2)' - (\grad v_2) \right) \tau
      + c \cdot (\grad v_3)  \tau && \text{ by~\eqref{eq:3}},
    \\
    & = c \cdot \left((\grad v_3)' - (\grad v_3) \right) \tau
      + c \cdot (\grad v_3)  \tau && \text{ as $\curl v$ is
                                     $C^0$ at $x_1$}
    \\
    & =  \tau  \cdot (\grad v_3)\, c.
  \end{align*}
  This argument can be repeated at other vertices to finish the
  proof of~\eqref{eq:4}.

  Now we are ready to show that $v$ is $C^1$ at $x_i$. Let $\tau_i =
  (x_i - z) / \| x_i - z\|$ and $\tau_{ij} = (x_i - x_j) / \| x_i -
  x_j \|$. Without loss of generality, we focus on one vertex,
  say $x_1$. At $x_1$,
  \begin{subequations}
    \label{eq:6}
    \begin{align}
      c \cdot (\grad v_2) \tau_0
      & =     c \cdot (\grad v_3) \tau_0
      & \tau_0 (\grad v_2) c & = \tau_0 (\grad v_3) c
      \\
      c \cdot (\grad v_2) \tau_{10}
      & = c \cdot (\grad v_3) \tau_{10}
      & \tau_{10} (\grad v_2) c & = \tau_{10} (\grad v_3) c.
    \end{align}
  \end{subequations}
  The left equalities follow from~\eqref{eq:3} and the right ones
  from~\eqref{eq:4}. Furthermore, at $x_1$ we have 
  \begin{align*}
    \tau_{12} \cdot (\grad v_2 ) \tau_{13}
    & = \tau_{12} \cdot (\grad v_0) \tau_{13}
    && \text{by~\eqref{eq:3} applied to $F_{20}$}       
    \\
    & = \tau_{12} \cdot (\grad v_3) \tau_{13}
    && \text{by~\eqref{eq:4} applied to $F_{03}$}.
  \end{align*}
  Therefore, $\tau_{12} \cdot (\grad v_2 )\tau_{13} =
  \tau_{12} \cdot (\grad v_3 )\tau_{13}$ at $x_1$.
  Writing $\tau_{12}$ as a linear combination of
  $\tau_0, \tau_{10}$ and $\tau_{13}$, and using the equalities
  of~\eqref{eq:6} in the right panel, we conclude that 
  \begin{equation}
    \label{eq:7}
    \tau_{13} \cdot (\grad v_2) \tau_{13}
    = \tau_{13} \cdot (\grad
    v_3) \tau_{13} \qquad \text{ at } x_1.
  \end{equation}
  The identities of~\eqref{eq:7} and~\eqref{eq:6} together yield the
  equality of $\grad v_2$ and $\grad v_3$ at $x_1$. Repeating this
  argument for every pair of $v_i$ meeting at a vertex, the proof is
  finished.
\end{proof}

\begin{corollary}
  \label{cor:supersmoothness}
  If $w \in C^1(T)$ and $w|_{T_i}$ is in $C^\infty(T_i)$, then $w$ is
  $C^2$ at the vertices of $T$.
\end{corollary}
\begin{proof}
  This follows by applying Theorem~\ref{thm:supersmooth} with
  $v = \grad w$.
\end{proof}

\bibliographystyle{abbrv}
\bibliography{references}

\begin{thebibliography}{10}

\bibitem{Alfel84}
P.~Alfeld.
\newblock A trivariate {Clough}--{Tocher} scheme for tetrahedral data.
\newblock {\em Computer Aided Geometric Design}, 1(2):169--181, 1984.

\bibitem{amstutz2019incompatibility}
S.~Amstutz and N.~Van~Goethem.
\newblock {The incompatibility operator: from Riemann’s intrinsic view of
  geometry to a new model of elasto-plasticity}.
\newblock In {\em Topics in Applied Analysis and Optimisation}, pages 33--70.
  Springer, 2019.

\bibitem{angoshtari2015differential}
A.~Angoshtari and A.~Yavari.
\newblock Differential complexes in continuum mechanics.
\newblock {\em Archive for Rational Mechanics and Analysis}, 216(1):193--220,
  2015.

\bibitem{arnold2008finite}
D.~Arnold, G.~Awanou, and R.~Winther.
\newblock Finite elements for symmetric tensors in three dimensions.
\newblock {\em Mathematics of Computation}, 77(263):1229--1251, 2008.

\bibitem{arnold2007mixed}
D.~Arnold, R.~Falk, and R.~Winther.
\newblock Mixed finite element methods for linear elasticity with weakly
  imposed symmetry.
\newblock {\em Mathematics of Computation}, 76(260):1699--1723, 2007.

\bibitem{arnold2018finite}
D.~N. Arnold.
\newblock {\em Finite element exterior calculus}.
\newblock SIAM, 2018.

\bibitem{arnold2006defferential}
D.~N. Arnold, R.~S. Falk, and R.~Winther.
\newblock Differential complexes and stability of finite element methods {II}:
  The elasticity complex.
\newblock In {\em Compatible spatial discretizations}, pages 47--67. Springer,
  2006.

\bibitem{arnold2006finite}
D.~N. Arnold, R.~S. Falk, and R.~Winther.
\newblock Finite element exterior calculus, homological techniques, and
  applications.
\newblock {\em Acta numerica}, 15:1--155, 2006.

\bibitem{ArnolFalkWinth07}
D.~N. Arnold, R.~S. Falk, and R.~Winther.
\newblock Mixed finite element methods for linear elasticity with weakly
  imposed symmetry.
\newblock {\em Math. Comp.}, 76(260):1699--1723 (electronic), 2007.

\bibitem{arnold2020complexes}
D.~N. Arnold and K.~Hu.
\newblock Complexes from complexes.
\newblock {\em arXiv preprint arXiv:2005.12437}, 2020.

\bibitem{arnold2002mixed}
D.~N. Arnold and R.~Winther.
\newblock Mixed finite elements for elasticity.
\newblock {\em Numerische Mathematik}, 92(3):401--419, 2002.

\bibitem{calabi1961compact}
E.~Calabi.
\newblock {On compact, Riemannian manifolds with constant curvature I}.
\newblock {\em Proc. Sympos. Pure Math., Amer. Math. Soc., Providence, RI,
  1961}, pages 155--180, 1961.

\bibitem{vcap2001bernstein}
A.~{\v{C}}ap, J.~Slov{\'a}k, and V.~Sou{\v{c}}ek.
\newblock {Bernstein-Gelfand-Gelfand sequences}.
\newblock {\em Annals of Mathematics}, 154(1):97--113, 2001.

\bibitem{christiansen2011linearization}
S.~H. Christiansen.
\newblock {On the linearization of Regge calculus}.
\newblock {\em Numerische Mathematik}, 119(4):613--640, 2011.

\bibitem{christiansen2018nodal}
S.~H. Christiansen, J.~Hu, and K.~Hu.
\newblock {Nodal finite element de Rham complexes}.
\newblock {\em Numerische Mathematik}, 139(2):411--446, 2018.

\bibitem{christiansen2019finite}
S.~H. Christiansen and K.~Hu.
\newblock Finite element systems for vector bundles: elasticity and curvature.
\newblock {\em arXiv preprint arXiv:1906.09128}, 2019.

\bibitem{ciarlet2009intrinsic}
P.~G. Ciarlet, L.~Gratie, and C.~Mardare.
\newblock Intrinsic methods in elasticity: a mathematical survey.
\newblock {\em Discrete \& Continuous Dynamical Systems-A}, 23(1\&2):133, 2009.

\bibitem{costabel2010bogovskiui}
M.~Costabel and A.~McIntosh.
\newblock {On Bogovski{\u\i} and regularized Poincar{\'e} integral operators
  for de Rham complexes on Lipschitz domains}.
\newblock {\em Mathematische Zeitschrift}, 265(2):297--320, 2010.

\bibitem{eastwood2000complex}
M.~Eastwood.
\newblock A complex from linear elasticity.
\newblock In {\em Proceedings of the 19th Winter School" Geometry and
  Physics"}, pages 23--29. Circolo Matematico di Palermo, 2000.

\bibitem{Farin80}
G.~Farin.
\newblock B{\'e}zier polynomials over triangles.
\newblock Technical Report TR/91, Department of Mathematics, Brunel University,
  Uxbridge, UK, 1980.

\bibitem{farrell2020reynolds}
P.~E. Farrell, L.~Mitchell, L.~R. Scott, and F.~Wechsung.
\newblock {A Reynolds-robust preconditioner for the Reynolds-robust
  Scott-Vogelius discretization of the stationary incompressible Navier-Stokes
  equations}.
\newblock {\em arXiv preprint arXiv:2004.09398}, 2020.

\bibitem{fu2018exact}
G.~Fu, J.~Guzman, and M.~Neilan.
\newblock {Exact smooth piecewise polynomial sequences on Alfeld splits}.
\newblock {\em Mathematics of Computation}, 2020.

\bibitem{vangoethem2010kroener}
N.~V. Goethem.
\newblock {The non-Riemannian dislocated crystal: A tribute to Ekkehart
  Kr\"oner (1919-2000)}.
\newblock {\em Journal of Geometric Mechanics}, 2:303, 2010.

\bibitem{gopalakrishnan2012second}
J.~Gopalakrishnan and J.~Guzm{\'a}n.
\newblock A second elasticity element using the matrix bubble.
\newblock {\em IMA Journal of Numerical Analysis}, 32(1):352--372, 2012.

\bibitem{guzman2018inf}
J.~Guzm{\'a}n and M.~Neilan.
\newblock Inf-sup stable finite elements on barycentric refinements producing
  divergence--free approximations in arbitrary dimensions.
\newblock {\em SIAM Journal on Numerical Analysis}, 56(5):2826--2844, 2018.

\bibitem{hauret2007diamond}
P.~Hauret, E.~Kuhl, and M.~Ortiz.
\newblock Diamond elements: a finite element/discrete-mechanics approximation
  scheme with guaranteed optimal convergence in incompressible elasticity.
\newblock {\em International Journal for Numerical Methods in Engineering},
  72(3):253--294, 2007.

\bibitem{hu2014family}
J.~Hu and S.~Zhang.
\newblock A family of conforming mixed finite elements for linear elasticity on
  triangular grids.
\newblock {\em arXiv preprint arXiv:1406.7457}, 2014.

\bibitem{hu2015family}
J.~Hu and S.~Zhang.
\newblock A family of symmetric mixed finite elements for linear elasticity on
  tetrahedral grids.
\newblock {\em Science China Mathematics}, 58(2):297--307, 2015.

\bibitem{johnson1978some}
C.~Johnson and B.~Mercier.
\newblock Some equilibrium finite element methods for two-dimensional
  elasticity problems.
\newblock {\em Numerische Mathematik}, 30(1):103--116, 1978.

\bibitem{Krone60}
E.~Kr{\"o}ner.
\newblock Allgemeine {K}ontinuumstheorie der {V}ersetzungen und
  {E}igenspannungen.
\newblock {\em Arch. Rational Mech. Anal.}, 4:273--334 (1960), 1960.

\bibitem{lee2007robust}
Y.-J. Lee, J.~Wu, J.~Xu, and L.~Zikatanov.
\newblock Robust subspace correction methods for nearly singular systems.
\newblock {\em Mathematical Models and Methods in Applied Sciences},
  17(11):1937--1963, 2007.

\bibitem{li2018regge}
L.~Li.
\newblock {\em Regge finite elements with applications in solid mechanics and
  relativity}.
\newblock PhD thesis, University of Minnesota, 2018.

\bibitem{Nedel86}
J.-C. N{\'e}d{\'e}lec.
\newblock A new family of mixed finite elements in {${\bf R}\sp 3$}.
\newblock {\em Numer. Math.}, 50(1):57--81, 1986.

\bibitem{NeuntSchob20}
M.~Neunteufel and J.~Sch{\"o}berl.
\newblock {Avoiding Membrane Locking with Regge Interpolation}.
\newblock {\em Preprint: arXiv:1907.06232}, 2020.

\bibitem{regge1961general}
T.~Regge.
\newblock General relativity without coordinates.
\newblock {\em Il Nuovo Cimento (1955-1965)}, 19(3):558--571, 1961.

\bibitem{schoeberlthesis}
J.~Sch{\"o}berl.
\newblock {\em Robust multigrid methods for parameter dependent problems}.
\newblock PhD thesis, Johannes Kepler Universit\"at Linz, 1999.

\bibitem{stenberg2010nonstandard}
R.~Stenberg.
\newblock A nonstandard mixed finite element family.
\newblock {\em Numerische Mathematik}, 115(1):131--139, 2010.

\end{thebibliography}

\end{document}